\documentclass[a4 paper, 12pt]{article}

\usepackage{amsmath,mathrsfs,pictexwd,dcpic}
\usepackage{amssymb}
\usepackage[left = 3cm,right=2cm,top=3cm]{geometry}
\usepackage{amsthm}
\usepackage{caption}
\usepackage{epstopdf}
\usepackage{graphicx}
\usepackage{newlfont,amsfonts}
\usepackage{graphics}
\usepackage{ytableau}
\usepackage{enumitem}
\theoremstyle{definition}
\newtheorem{thm}{Theorem}[section]
\newtheorem{lem}[thm]{Lemma}
\newtheorem{Def}[thm]{Definition}

\newtheorem{rem}{Remark}
\newtheorem{ex}[thm]{Example}

\newtheorem{pro}[thm]{Proposition}
\allowdisplaybreaks
\title{Study of $p^{k}$-Eulerian polynomials and $p^{k}$-Fibonacci numbers for every odd prime $p$ and $k\geq0$}
\author{M. Parvathi, A. Tamilselvi and D. Hepsi\\
	Ramanujan Institute for Advanced Study in Mathematics,\\ University of Madras, Chennai, INDIA.\\
	E.mail: tamilselvi.riasm@gmail.com}
\date{}
\begin{document}
	\maketitle
	\begin{abstract}
		In this paper, we define the notion of descent for the paths in the $p$-Bratteli diagram.   This leads to the definition of $p^{k}$-Eulerian polynomials, whose coefficients count the number of paths with a given number of descents. We provide a method for constructing the $p^{k}$-Eulerian polynomials at each vertex. Furthermore, we compute the total number of descents of all paths ending at a given vertex as the corresponding $p^{k}$-Fibonacci numbers. We show that 	the derivative of the $p^{k}$-Eulerian polynomial evaluated at 1 for a fixed vertex equals the corresponding $p^{k}$-Fibonacci number. Finally, we discuss the generating function for the sequence of $p^{k}$-Fibonacci numbers and the  recurrence relations they satisfy.
	\end{abstract}
	
	
	secondary 05E16.
	
	
	Fibonacci numbers.
	\section{Introduction}
	The Eulerian polynomial $A_{d}(x)=\sum\limits_{k=1}^{d}A(d,k)x^{k},$ where $A(d,k)$ denotes the Eulerian number, which is the number of permutations $\sigma\in S_{d}$ with exactly $k-1$ descents, was introduced by Leonhard Euler in 18th century \cite{[RS]}. Eulerian polynomials have found applications in enumerative combinatorics, geometry, and beyond. Later, researchers continued to explore the properties of Eulerian polynomials. This was further extended in \cite{[HS]}, \cite{[KI]} and other works, which has led us to study this area.
	
	The Fibonacci sequence was introduced by Leonardo of Pisa (Fibonacci) in 1202 through his book Liber Abaci(The book of Calculation), originally arising from a problem on rabbit population growth. Fibonacci numbers have many real-life applications which are closely connected to natural patterns \cite{[SS]}. They also appear in different areas of mathematics and beyond. Recently, we find researchers working on Fibonacci numbers in connection with matrix theory \cite{[MA]} and also in the solution of Diophantine equations involving powers of $k$-Fibonacci numbers \cite{[GS]} and further geometric function theory involving $q$-Fibonaacci numbers \cite{[AA]} and so on. In this paper, we work on \textbf{$p^{k}$-Fibonacci numbers} which are completely different from the above.

	We make this paper self contained. In this paper, we investigate a new combinatorial structure that arises from the study of paths in the $p$-Bratteli diagram whose vertices are the hook partitions. Each such path consists of blocks and we introduce a descent by comparing the horizontal nodes(or vertical nodes) of the blocks added along the path. This notion of descent is similar to the descent defined for permutations in \cite{[BS]}.

	Building on this classical foundation, we define a variant of the Eulerian polynomial using the paths in the $p$-Bratteli diagram. For a fixed vertex of a given hook shaped partition, we consider all the paths ending at the vertex, and define a polynomial where the coefficient of $q^{i}$ denotes the total number of paths having exactly $'i'$ descents. We call this polynomial as \textbf{$p^{k}$-Eulerian polynomial} and the coefficients as \textbf{$p^{k}$-Eulerian number} for all $'i'.$ The sign balance for each vertex in the $p$-Bratteli diagram is discussed, following the approach in \cite{[ZY]}. This paper was the starting point for computing the sign balance,  $p^{k}$-Eulerian polynomial and $p^{k}$-Fibonacci number.
	
	Finally, we compute the total number of descents of all paths ending at the vertex $\lambda^{2r}_{p^{k}[2sp-(2s+1)],x^{s}_{k}+l^{s}_{k}},$ for $s\geq k+2,$ denoted by $\mathcal{M}_{\lambda^{2r}_{p^{k}[2sp-(2s+1)],x^{s}_{k}+l^{s}_{k}}},$ will be referred to as the \textbf{$p^{k}$-Fibonacci numbers}. The reason for calling it as $p^{k}$-Fibonacci numbers will be justified in Theorem \ref{rr} in the last section. The generating functions for the sequence of $p^{k}$-Fibonacci numbers  $\{\mathcal{M}_{\lambda^{2r}_{p^{k}[2sp-(2s+1)],x^{s}_{k}+l^{s}_{k}}}\}^{\infty}_{s= k+2}$ is discussed for every odd prime $p$ and every $k\geq0$.
	
	\section{$p$-Bratteli diagram}
	In this section, we define the vertices and edges of the $p$-Bratteli diagram. The hook partitions will correspond to the vertices. We also explain the ordering and  arragement of these vertices within the $p$-Bratteli diagram. Each path in the $p$-Bratteli diagram is discussed in detail, and a simplified notation for representing such paths is defined. Each path in the $p$-Bratteli diagram corresponds to a standard $p$-Young tableau defined in \cite{[TH1]}. 
	\begin{Def} \cite{[BS]}
		Suppose $\lambda=(\lambda_{1},\lambda_{2},\dots,\lambda_{l})$ and $\mu=(\mu_{1},\mu_{2},\dots,\mu_{m})$ are partitions of $n.$ Then $\lambda$ dominates $\mu,$ written $\lambda \trianglerighteq \mu,$ if $$\lambda_{1}+\lambda_{2}+\dots+\lambda_{i}\geq \mu_{1}+\mu_{2}+\dots+\mu_{i} $$ for all $i\geq 1.$ If $i>l$ (respectively, $i>m$) then we take $\lambda_{i}$ (respectively $\mu_{i}$) to be zero.
	\end{Def}	
	\begin{Def}
		Let $\lambda$ be a partition of size $n.$ It is said to be a hook partition of size $n$ if it is of the form $\lambda = (n-i,1^{i}),$ where the corresponding Young diagram consists of $n-i$ horizontal nodes and $i$ vertical nodes.
	\end{Def}
	\noindent\textbf{Notation}
	\begin{itemize}
		\item $B_{m,n}$ stands for a block in a hook partition $\lambda$ of size $m+n$ which contains `$m$' nodes horizontally and `$n$' nodes vertically, where the vertical nodes, we mean the nodes in the first column other than the first row.
		\item $\lambda\overset{B_{m,n}}{\hookrightarrow}\mu$ denotes the partition $\mu$ which is obtained from the partition $\lambda$ by adding the block $B_{m,n}.$	
		\item $\lambda\overset{B_{m,n}}{\hookleftarrow}\mu$ denotes the partition $\lambda$ which is obtained from the partition $\mu$ by removing the block $B_{m,n}.$
		\item $j^{k}_{t}=t\sum\limits_{i=0}^{k-1}p^{i},$ where $0\leq t \leq p-1$ and $k\geq 1.$
		\item $[a,\dots,b]$ denotes the integers from $a$ to $b.$ 
	\end{itemize}
	
	\subsection{Vertex set of the $p$-Bratteli diagram}\label{vertex}
	\subsubsection*{Vertices at the $2r-1$-th floor:}
	$$W^{2r-1}= \bigcup_{k=-1}^{r-2} W^{2r-1}_{k} $$
	where
	\begin{itemize}
		\item  $W_{-1}^{2r-1} = \{\lambda^{2r-1}_{p^{r-1}(p-1),i}/0\leq i<p^{r-1}(p-1)\}$ where $r\geq1.$
		\item $W^{2r-1}_{k} = \{\lambda^{2r-1}_{p^{k}[(2(r-k)-1)p-2(r-k)],x^{r-k}_{k}-p^{k}(p-1)+l^{r-k-1}_{k+1}}/0\leq l^{r-k-1}_{k+1}<p^{k+1},~x^{r-k}_{k}=p^{k}[(r-k)p-(r-k+1)]\}$ for $k\geq 0$
	\end{itemize}
	The vertex subsets are placed as follows $$W^{2r-1}_{-1},~W^{2r-1}_{r-2},~W^{2r-1}_{r-3},\dots,~W^{2r-1}_{1},~W^{2r-1}_{0}.$$
	
	The vertices in each vertex subset $W^{2r-1}_{k}$ are arranged according to the dominance ordering.
	\subsubsection*{Vertices at the $2r$-th floor, $r\geq 1$:}
	$$V^{2r}=\bigcup_{k=-1}^{r-1}V^{2r}_{k}$$
	where
	\begin{itemize}
		\item  $V^{2r}_{-1} = \{\lambda^{2r}_{p^{r-1}(p-1),i}/0\leq i<p^{r-1}(p-1)\}$
		\item $V^{2r}_{k} = \{\lambda^{2r}_{p^{k}[2(r-k)p-(2(r-k)+1)],x^{r-k}_{k}+l^{r-k}_{k}}/0\leq l^{r-k}_{k} <p^{k},~x^{r-k}_{k}=p^{k}[(r-k)p-(r-k+1)]\}$ for $k\geq 0$ 
	\end{itemize}
	The vertex subsets are placed as follows $$V^{2r}_{-1},~V^{2r}_{r-1},~V^{2r}_{r-2},\dots,~V^{2r}_{1},~V^{2r}_{0}.$$
	The vertices in each vertex subset $V^{2r}_{k}$ are arranged according to the dominance ordering.

	\subsection{Edge Set of the $p$-Bratteli diagram}\label{edge}
	The edges between two vertices of two consecutive floors are defined as:
	\subsubsection*{Edges from $(2r-2)$-th floor to $(2r-1)$-th floor:}
	\begin{itemize}
		\item Let $\lambda^{2r-2}_{p^{r-2}(p-1),i}\in V^{2r-2}_{-1}$ and  $\lambda^{2r-1}_{p^{r-1}(p-1),pi+t}\in W^{2r-1}_{-1},$ then for all $0\leq t\leq p-1,$ there is an edge between the vertices as shown in Figure \ref{1}
		\begin{figure}[h!]
			\centering
			\includegraphics[width=0.5\linewidth]{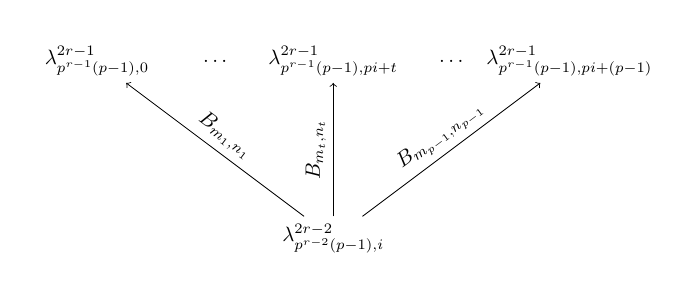}
			\caption{$\lambda^{2r-2}_{p^{r-2}(p-1),i}\hookrightarrow \lambda^{2r-1}_{p^{r-1}(p-1),pi+t}$}
			\label{1}
		\end{figure}
		
		where $m_{t}=p^{r-2}(p-1)^{2}-(i(p-1)+t),~n_{t}=i(p-1)+t.$
		\item Let $\lambda^{2r-2}_{p^{k}[2(s-1)p-(2s-1)],x^{s-1}_{k}+l^{s-1}_{k}}\in V^{2(r-1)}_{k}$ and $\lambda^{2r-1}_{p^{k}[(2s-1)p-2s],x^{s-1}_{k}+pl^{s-1}_{k}+t}\in W^{2r-1}_{k},$ then for all $0\leq t\leq p-1,$ there is an edge between the vertices as shown in Figure \ref{2}
		
		\begin{figure}[h!]
			\centering
			\includegraphics[width=0.6\linewidth]{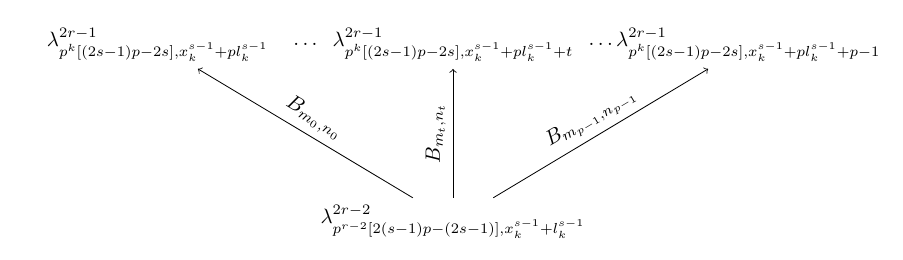}
			\caption{$\lambda^{2r-2}_{p^{k}[2(s-1)p-(2s-1)],x^{s-1}_{k}+l^{s-1}_{k}}\hookrightarrow \lambda^{2r-1}_{p^{k}[(2s-1)p-2s],x^{s-1}_{k}+pl^{s-1}_{k}+t}$}
			\label{2}
		\end{figure}
		
		where $s=r-k,~m_{t}=p^{k}(p-1)-((p-1)l^{s-1}_{k}+t)$ and $n_{t}=(p-1)l^{s-1}_{k}+t.$ 
	\end{itemize}
	\subsubsection*{Edges from $(2r-1)$-th floor to $2r$-th floor:}
	\begin{itemize}
		\item Let $\lambda^{2r-1}_{p^{r-1}(p-1),i}\in W^{2r-1}_{-1}$ and $\lambda^{2r}_{p^{r-1}(p-1),i}\in V^{2r}_{-1},$ then there is an edge between the vertices by adding the empty block $$\lambda^{2r-1}_{p^{r-1}(p-1),i}\hookrightarrow\lambda^{2r}_{p^{r-1}(p-1),i}$$ 
		\item Let $\lambda^{2r-1}_{p^{r-1}(p-1),i}\in W^{2r-1}_{-1}$ and $\lambda^{2r}_{p^{r-1}(2p-3),i+p^{r-1}(p-2-t)}\in V^{2r}_{r-1}.$ 
		If $p^{r-1}t\leq i<p^{r-1}(t+1)$ for $0\leq t\leq p-2,$  then there is an edge between the vertices 
		$$\lambda^{2r-1}_{p^{r-1}(p-1),i}  \overset{B_{p^{r-1}t,p^{r-1}(p-2-t)}}{\hookrightarrow}\lambda^{2r}_{p^{r-1}(2p-3),i+p^{r-1}(p-2-t)}$$
		\item Let $\lambda^{2r-1}_{p^{k}[(2s-1)p-2s],x^{s}_{k}-p^{k}(p-1)+l^{s-1}_{k+1}}\in W^{2r-1}_{k}$ and $\lambda^{2r}_{p^{k}[2sp-(2s+1)],x^{s}_{k}+l^{s}_{k}+p^{k}t} \in V^{2r}_{k}.$ If $p^{k}t\leq l^{s-1}_{k+1}<p^{k}(t+1)$ for $0\leq t\leq p-1,$ then there is an edge between the vertices $$\lambda^{2r-1}_{p^{k}[(2s-1)p-2s],x^{s}_{k}-p^{k}(p-1)+l^{s-1}_{k+1}}\overset{B_{p^{k}t,p^{k}(p-1-t)}}{\hookrightarrow}\lambda^{2r}_{p^{k}[2sp-(2s+1)],x^{s}_{k}+l^{s-1}_{k+1}-p^{k}t}$$ 
		where $s=r-k.$
	\end{itemize}
	\subsubsection*{Edges from $2r$-th floor to $(2r-1)$-th floor:}
	\begin{itemize}
		\item Let $\lambda^{2r}_{p^{r-1}(p-1),i}\in V^{2r}_{-1}$ and $\lambda^{2r-1}_{p^{r-1}(p-1),i}\in W^{2r-1}_{-1},$ then there is an edge between the vertices by removing the empty block $$\lambda^{2r-1}_{p^{r-1}(p-1),i}\hookleftarrow\lambda^{2r}_{p^{r-1}(p-1),i}$$ 
		\item  Let $\lambda^{2r}_{p^{r-1}(2p-3),x^{1}_{r-1}+l^{1}_{r-1}} \in V^{2r}_{r-1}$ and $\lambda^{2r-1}_{p^{r-1}(p-1),x^{1}_{r-1}+l^{1}_{r-1}-p^{r-1}(p-2-t)}\in W^{2r-1}_{-1}$, then for $0\leq t\leq p-2$ there is an edge between the vertices as shown in Figure \ref{3}
		\begin{figure}[h!]
			\centering
			\includegraphics[width=0.5\linewidth]{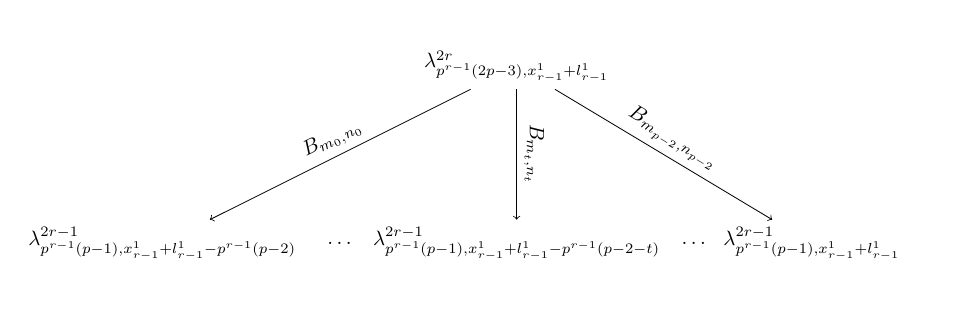}
			\caption{$\lambda^{2r-1}_{p^{r-1}(p-1),x^{1}_{r-1}+l^{1}_{r-1}-p^{r-1}(p-2-t)}\hookleftarrow\lambda^{2r}_{p^{r-1}(2p-3),x^{1}_{r-1}+l^{1}_{r-1}}$}
			\label{3}
		\end{figure}
		
		where $m_{t}=p^{r-1}t,~n_{t}=p^{r-1}(p-2-t).$
		\item Let $\lambda^{2r}_{p^{k}[2sp-(2s+1)],x^{s}_{k}+l^{s}_{k}}\in V^{2r}_{k}$ and $\lambda^{2r-1}_{p^{k}[(2s-1)p-2s],x^{s}_{k}+l^{s}_{k}-p^{k}(p-1-t)}\in W^{2r-1}_{k},$ then for $0\leq t\leq p-1$ there is an edge between the vertices as shown in Figure \ref{4}
		\begin{figure}[h!]
			\centering
			\includegraphics[width=0.5\linewidth]{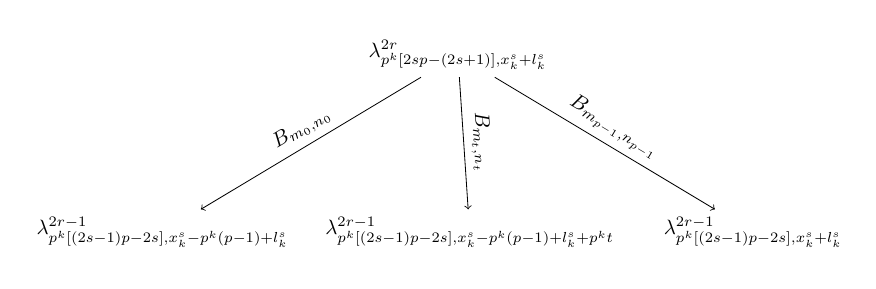}
			\caption{$\lambda^{2r-1}_{p^{k}[(2s-1)p-2s],x^{s}_{k}+l^{s}_{k}-p^{k}(p-1-t)}\hookleftarrow\lambda^{2r}_{p^{k}[2sp-(2s+1)],x^{s}_{k}+l^{s}_{k}}$}
			\label{4}
		\end{figure}
		
		where $s=r-k,~m_{t}=p^{k}t$ and $n_{t}=p^{k}(p-1-t).$ 
	\end{itemize}
	\subsubsection*{Edges from $(2r-1)$-th floor to $2(r-1)$-th floor:}
	\begin{itemize}
		
		\item Let $\lambda^{2r-1}_{p^{r-1}(p-1),i}\in W^{2r-1}_{-1}$ and $\lambda^{2r-2}_{p^{r-2}(p-1),\Bar{i}}\in V^{2r-2}_{-1}.$ If $\bar{i}p\leq i <(\bar{i}+1)p,$ then there is an edge between     
		then there is an edge between the vertices $$\lambda^{2r-2}_{p^{r-2}(p-1),\bar{i}}\overset{B_{p^{r-2}(p-1)^{2}-(i-\bar{i}),(i-\bar{i})}}{\hookleftarrow}\lambda^{2r-1}_{p^{r-1}(p-1),i}$$
		\item Let $\lambda^{2r-1}_{p^{k}[(2s-1)p-2s],x^{s-1}_{k}+l^{s-1}_{k+1}}\in W^{2r-1}_{k}$ and $\lambda^{2r-2}_{p^{k}[2(s-1)p-(2s-1)],x^{s-1}_{k}+\alpha^{s-1}_{k}}\in V^{2r-2}_{k},$ then there is an edge between the vertices 
		\begin{multline*}
			\lambda^{2r-2}_{p^{k}[2(s-1)p-(2s-1)],x^{s-1}_{k}+\alpha^{s-1}_{k}}\overset{B_{p^{k}(p-1)-((p-1)\alpha^{s-1}_{k}+\beta^{s-1}_{k}),(p-1)\alpha^{s-1}_{k}+\beta^{s-1}_{k}}}{\hookleftarrow}  \lambda^{2r-1}_{p^{k}[(2s-1)p-2s],x^{s-1}_{k}+l^{s-1}_{k+1}} 
		\end{multline*}
		where $l^{s-1}_{k+1}= \alpha^{s-1}_{k} p + \beta^{s-1}_{k},$ $0\leq \alpha^{s-1}_{k}< p^{k-1},$ $0\leq \beta^{s-1}_{k} \leq p-1$ and $s=r-k.$
	\end{itemize}
	The $p$-Bratteli diagram consists of vertices belongs to the sets $V^{2r}$ and $W^{2r-1}$ and the corresponding edges defined as above.
	\begin{rem}
		Let $k$ be the fixed non negative integer, then 
		$$|V^{2r}_{k}|=p^{k},~\forall~r>0,~k\geq0$$
		$$|W^{2r-1}_{k}|=p^{k+1},~\forall~r>1,~k\geq 0$$
	\end{rem}
	\begin{rem}\label{syt def}
		For $k\geq0,$ the block which is added to the partition $\lambda^{2r-1}_{p^{k}[(2s-1)p-2s],x^{s}_{k}-p^{k}(p-1)+l^{s-1}_{k+1}}$ $\in W^{2r-1}_{k}$ to obtain the partition $\lambda^{2r}_{p^{k}[2sp-(2s+1)],x^{s}_{k}+l^{s}_{k}}$ $\in V^{2r}_{k}$ is denoted by $B^{2s}_{m_{2s},n_{2s}},$ where $ 2\leq s \leq r.$
		
		Similarly, the block which is added to the partition $\lambda^{2r-2}_{p^{k}[2(s-1)p-(2s-1)],x^{s-1}_{k}+l^{s-1}_{k}}$ $\in V^{2r-2}_{k}$ to obtain the partition $\lambda^{2r-1}_{p^{k}[(2s-1)p-2s],x^{s}_{k}-p^{k}(p-1)+l^{s-1}_{k+1}}\in W^{2r-1}_{k}$ is denoted by
		$B^{2s-1}_{m_{2s-1},n_{2s-1}},$ where $ 2\leq s \leq r.$

		The path starting at the vertex $\lambda^{1}_{p-1,i_{1}},$ $~0\leq i_{1}<p$ and ending at the vertex $\lambda^{k+1}_{p^{k}(p-1),i_{k+1}},~0\leq i_{k+1}<p^{k}(p-1)$ is treated as a single entity and considered as the single block $B^{1}_{p^{k}(p-1)-l,l},~0\leq l <p^{k}(p-1).$ 
		$$\lambda^{1}_{p-1,i_{1}}\overset{}{\hookrightarrow
		}\lambda^{2}_{p-1,i_{1}}\overset{B^{1}_{a_{1},b_{1}}}{\hookrightarrow}\lambda^{3}_{p(p-1),i_{2}}\overset{}{\hookrightarrow}\lambda^{4}_{p(p-1),i_{2}}\overset{B^{1}_{a_{2},b_{2}}}{\hookrightarrow}\lambda^{5}_{p^{2}(p-1),i_{3}}\dots$$ 
		where $a_{j}=p^{j-1}(p-1)^{2}-h_{j},$ $b_{j} = h_{j}$ and $0\leq h_{j}<p^{j-1}(p-1)^{2},~\forall~j\geq 1.$ 
		\begin{multline*}
			\lambda^{k+1}_{p^{k}(p-1),i_{k+1}}\overset{B^{2}_{m_{2},n_{2}}}{\hookrightarrow}\lambda^{k+2}_{p^{k}(2p-3),x^{1}_{k}+l^{s}_{k}}\overset{B^{3}_{m_{3},n_{3}}}{\hookrightarrow}\dots  \overset{B^{s}_{m_{s},n_{s}}}{\hookrightarrow}\lambda^{r}_{p^{k}[sp-(s+1)],k_{s}}
		\end{multline*}
		where
		\begin{itemize}
			\item $k_{s}=x^{n}_{k}+l^{n}_{k},$ when $s=2n$ and $k_{s}=x^{n}_{k}-p^{k}(p-1)+l^{n-1}_{k+1},$ when $s=2n-1$ with $x^{n}_{k}=p^{k}[np-(n+1)],~0\leq l^{n}_{k}<p^{k}$ and $0\leq l^{n-1}_{k+1} <p^{k+1}.$
			\item $m_{2}=p^{k}t_{2}$ and $n_{2}=p^{k}(p-2-t_{2}),~0\leq t_{2}\leq p-2$    
			\item $m_{j}= p^{k}t_{j}$ and $n_{j}=p^{k}(p-1-t_{j}),~0\leq t_{j}\leq p-1$ when $j$ is even
			\item $m_{j}=p^{k}(p-1)-l$ and $n_{j}=l,~0\leq l \leq p^{k}(p-1)$
			when $j$ is odd.
		\end{itemize}
		The path defined above is same thing as the standard $p$-Young tableau of shape  $\lambda^{r}_{p^{k}[sp-(s+1)],k_{s}}$ as in \cite{[TH1]}. 
		
		For the sake of simplicity, the above path is rewritten as follows $$(\lambda^{k+1}_{p^{k}(p-1),i},m_{2},m_{3},\dots,m_{s},\lambda^{r}_{p^{k}[sp-(s+1)],k_{s}}).$$ where $i_{k+1}=i.$
	\end{rem}
	\section{Inversions and Descents of paths }
	In this section, we define the notion of inversion and descent of a path in the $p$-Bratteli diagram using the blocks added along the path. Inversions are determined by comparing the horizontal nodes(or vertical nodes) of the blocks $B^{i}_{m_{i},n_{i}}$ and $B^{j}_{m_{j},n_{j}},$ for $i<j.$ Descents are determined by comparing the horizontal nodes(or vertical nodes) of the blocks $B^{i}_{m_{i},n_{i}}$ and $B^{i+1}_{m_{i+1},n_{i+1}},$ where $i>2.$ For the blocks $B^{1}_{m_{1},n_{1}}$ and $B^{2}_{m_{2},n_{2}},$ we introduce a specific convention, since the size of respective blocks are of different in size. We examine the descent at  $2s$ and $2s-1$ for all $s\geq 2,$ which helps to determine the descent at any integer($\geq3$) for every path. In particular, we analyze the descents in certain special paths of the $p$-Bratteli diagram.
	
	We compare any two consecutive blocks of same size added along the path. But first we make the following convention 
	
	\begin{enumerate}[label=\textbf{(c\arabic*)}]
		\item \label{c1} Consider the path $\lambda^{2r-1}_{p^{k}(p-1),i}\overset{B^{2}_{p^{k}t,p^{k}(p-2-t)}}{\hookrightarrow} \lambda^{2r}_{p^{k}(2p-3),x^{1}_{k}+l^{1}_{k}},~0\leq t \leq p-2,$ we say that there is a descent at 1, if it is obtained by adding the block $B^{2}_{p^{k}t,p^{k}(p-2-t)},$ where $0\leq t <\frac{p-1}{2},$ otherwise there is no descent. 
		\item \label{c2} We do not compare the blocks $B^{1}_{m_{1},n_{1}}$ and $B^{2}_{m_{2},n_{2}}$ with any another block.
	\end{enumerate}
	\begin{Def}\label{block1}
		For $i<j,$ $B^{i}_{m_{i},n_{i}}>B^{j}_{m_{j},n_{j}}$ if $m_{i}>m_{j}$ and $n_{i}<n_{j},$ $\forall~2<i<j.$
		
	\end{Def}
	\subsection{Inversions and Sign balances of paths}
	\begin{Def}\label{inversion}
		Let $P$ be the path $(\lambda^{k+1}_{p^{k}(p-1),i},m_{2},m_{3},\dots,m_{s},\lambda^{r}_{p^{k}[sp-(s+1)],k_{s}}).$ Also, let $B^{i}_{m_{i},n_{i}},~B^{j}_{m_{j},n_{j}}$ be the blocks added along the path $P$ with $i<j.$  The inversion of a path $P$ is defined as follows:
		$$Inv(P)=
		\begin{cases}
			\{(i,j):~ i<j ~\text{but}~B^{i}_{m_{i},n_{i}}>B^{j}_{m_{j},n_{j}}\} & \text{for}~j\geq3 \\  \text{followed as in \ref{c1}},~\ref{c2} & \text{for}~j=1,2
		\end{cases}$$
		and $inv(P)=|Inv(P)|$.
	\end{Def}
	\begin{Def}
		The sign of a path $P$ is defined by $$sgn(P) = (-1)^{inv(P)}.$$
	\end{Def}
	\begin{thm}
		The sign balance for any vertex of shape $\lambda^{2r}_{p^{k}[2sp-(2s+1)],x^{s}_{k}+l^{s}_{k}}\in V^{2r}_{k}$ (or $\lambda^{2r+1}_{p^{k}[(2s+1)p-2(s+1)],x^{s}_{k}+l^{s}_{k+1}}\in W^{2r+1}_{k}$) is $$I_{\lambda}=\sum_{P\in sh(\lambda)}sgn(P)=0$$
		where $\lambda = \lambda^{2r}_{p^{k}[2sp-(2s+1)],x^{s}_{k}+l^{s}_{k}}$ or $\lambda^{2r+1}_{p^{k}[(2s+1)p-2(s+1)],x^{s}_{k}+l^{s}_{k+1}}$ and $sh(\lambda)$ denotes the set of all paths ending at the vertex $\lambda.$
	\end{thm}
	\begin{proof}
		The proof follows by the convention \ref{c1} and  Definition \ref{inversion}.
	\end{proof}
	\subsection{Descents of paths}
	\begin{Def}\label{block}
		Let $P$ be the path $(\lambda^{k+1}_{p^{k}(p-1),i},m_{2},m_{3},\dots,m_{s},\lambda^{r}_{p^{k}[sp-(s+1)],k_{s}}).$ Also, let $B^{i}_{m_{i},n_{i}},~B^{i+1}_{m_{i+1},n_{i+1}}$ be the blocks added along the path $P.$ The descent of a path $P$ is defined as follow
		$$Des(P)=
		\begin{cases}
			\{i:~ B^{i}_{m_{i},n_{i}}>B^{i+1}_{m_{i+1},n_{i+1}}\} & \text{for}~i\geq3 \\  \text{followed as in \ref{c1}},~\ref{c2} & \text{for}~i=1,2
		\end{cases}$$
		and $des(P)=|Des(P)|$. If $i\in Des(P),$ then we say that there is a descent at $i.$
	\end{Def}
	
	\begin{lem}\label{2p-3}
		Let $P$ be a path $(\lambda^{2k+1}_{p^{k}(p-1),i},m_{2},\lambda^{2k+2}_{p^{k}(2p-3),x^{1}_{k}+l^{1}_{k}})$ and $P^{'}$ be a path $(\lambda^{2k+1}_{p^{k}(p-1),i},$ $ m_{2},m_{3},\lambda^{2k+3}_{p^{k}(3p-4),x^{1}_{k}+l^{1}_{k+1}})$ in the $p$-Bratteli diagram. Then 
		$$des(P) = 1~\text{or}~0~\text{and}~des(P^{'}) =1 ~\text{or}~0.$$ 	 
		Moreover, the total number of descent of all paths ending at the vertex $\lambda^{2k+2}_{p^{k}(2p-3),x^{1}_{k}+l^{1}_{k}}$ (and $\lambda^{2k+3}_{p^{k}(3p-4),x^{2}_{k}-p^{k}(p-1)+l^{s-1}_{k+1}}$) is $\frac{p-1}{2}.$
	\end{lem}
	\begin{proof}
		The figure shown below describes each path $P$ ending at the vertex  $\lambda^{2k+2}_{p^{k}(2p-3),x^{1}_{k}+l^{1}_{k}}$
		\begin{figure}[h!]
			\begin{center}
				\includegraphics[width=0.6\linewidth]{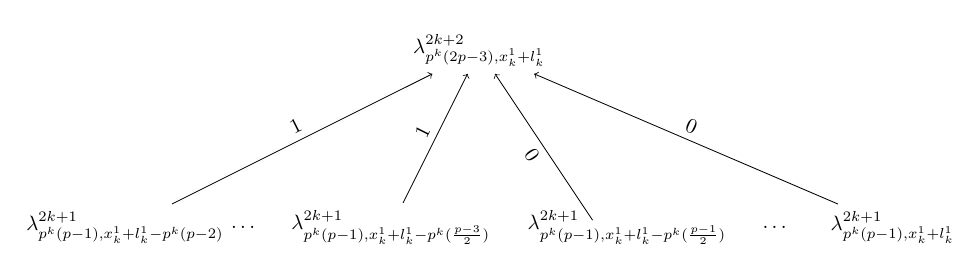}
			\end{center}
		\end{figure}
		
		The descent at 1 assigned between the vertices is followed by \ref{c1}. Therefore, $des(P)$ is either 1 or 0.
		By \ref{c2} the $des(P^{'})$ is also the same as $des(P),$ $des(P^{'})$ is either 1 or 0.  
		
		By adding the descents of all paths ending at the given vertex is $\frac{p-1}{2}.$ Hence the proof follows.
	\end{proof}
	\begin{lem}\label{thmlem}
		\begin{enumerate}
			\item  The blocks which are added to the partition $\lambda^{2r}_{p^{k}[2sp-(2s+1)],x^{s}_{k}+l^{s}_{k}}$ to obtain the partition $\lambda^{2r+1}_{p^{k}[(2s+1)p-2(s+1)],x^{s}_{k}+l^{s}_{k+1}}$ satisfying the following inequality
			$$B^{2s+1}_{p^{k}(p-1),0}>B^{2s+1}_{p^{k}(p-1)-1,1}>\dots>B^{2s+1}_{p^{k}(p-1)-l,l}>\dots>B^{2s+1}_{1,p^{k}(p-1)-1}>B^{2s+1}_{0,p^{k}(p-1)}$$
			where $x^{s}_{k}=p^{k}[sp-(s+1)],$ $0\leq l^{s}_{m} < p^{m},~ m=k,k+1$ for all $s\geq 1.$
			\item The blocks which are added to the partition $\lambda^{2r-1}_{p^{k}[(2s-1)p-2s],x^{s}_{k}-p^{k}(p-1)+l^{s-1}_{k+1}}$ to obtain the partition $\lambda^{2r}_{p^{k}[2sp-(2s+1)],x^{s}_{k}+l^{s-1}_{k}+p^{k}t}$ satisfying the following inequality
			$$B^{2s}_{0,p^{k}(p-1)}<B^{2s}_{p^{k},p^{k}(p-2)}<B^{2s}_{2p^{k},p^{k}(p-3)}<\dots<B^{2s}_{p^{k}(p-2),p^{k}}<B^{2s}_{p^{k}(p-1),0}$$
			where $0\leq t \leq p-1,$ $x^{s}_{k}=p^{k}[sp-(s+1)],$ $0\leq l^{s}_{m} < p^{m},~ m=k,k+1$ for all $s\geq 2.$
		\end{enumerate}
	\end{lem}
	\begin{proof}
		Proof follows by the Definition \ref{block}
	\end{proof}
	
	We consider the vertices of the floors $2r+1,~2r$ and $2r-1$ in the $p$-Bratteli diagram in the following Lemma and Theorem, which provide the complete information on each and every path by varying $r$.
	
	\begin{lem}\label{rules jt}
		For $k\geq 1,$ let $P^{2r+1}_{j_{t}}$ denote the path $(\lambda^{2k+1}_{p^{k}(p-1),i},m_{2},\dots,\lambda^{2r+1}_{p^{k}[(2s+1)p-2(s+1)],x^{s}_{k}+j_{t}}),$ where: 
		\begin{itemize}
			\item $m_{2d}=p^{k}t,$ for $2\leq d \leq s$ 
			\item  $m_{2d-1}=p^{k}(p-1-t),$ for $2\leq d\leq s+1.$
		\end{itemize}
		Then the following conditions hold for the path $P^{2r+1}_{j_{t}}$
		\begin{itemize}
			\item If $0\leq t <\frac{p-1}{2},$ then there is a descent at $2s-1$ and no descent at $2s$ of the path $P^{2r+1}_{j_{t}}.$
			\item If $ t =\frac{p-1}{2},$ then there is no descent at $2s-1$ and $2s$ for any $s,$ of the path $P^{2r+1}_{j_{t}}.$
			\item If $\frac{p-1}{2}<t\leq p-1,$ then there is no descent at $2s-1$ and there is a descent at $2s$ of the path $P^{2r+1}_{j_{t}}.$ 
		\end{itemize} 
	\end{lem}	
	\begin{proof}
		\begin{itemize}
			\item When $0\leq t <\frac{p-1}{2},$ we have
			\begin{equation}\label{eq3}
				m_{2s-1}>m_{2s} ~\text{and} ~m_{2s}<m_{2s+1}
			\end{equation}
			This implies that there is a descent at $2s-1$ and no descent at $2s$ of the path $P^{2r+1}_{j_{t}}.$
			\item When $t=\frac{p-1}{2},$ we have  
			\begin{equation*}\label{eq3}
				m_{2s-1}=m_{2s} ~\text{and} ~m_{2s}=m_{2s+1}
			\end{equation*}
			This implies there is no descent at $2s-1$ and $2s.$
			\item When $\frac{p-1}{2}<t\leq p-1,$ we have
			\begin{equation}\label{eq2}
				m_{2s-1}<m_{2s} ~\text{and} ~m_{2s}>m_{2s+1}
			\end{equation}
			This implies that there is no descent at $2s-1$ and there is a descent at $2s$ of the path $P^{2r+1}_{j_{t}}.$
		\end{itemize}
	\end{proof} 
	The above paths are \textbf{special paths} of the $p$-Bratteli diagram which are helpful in determining the descents of any paths.
	\begin{thm}\label{rules l}
		For $k\geq 1$ and $s\geq 2,$ let $P^{2r+1}_{l^{s}_{k},\beta^{'},t^{'}}$ denote the paths $(\lambda^{2k+1}_{p^{k}(p-1),i},m_{2},\dots,m_{2s+1},$ $\lambda^{2r+1}_{p^{k}[(2s+1)p-2(s+1)],x^{s}_{k}+pl^{s}_{k}+\beta^{'}}),$ in the $p$-Bratteli diagram such that 
		\begin{itemize}
			\item $m_{2s+1}=m^{\beta^{'}}_{2s+1}=p^{k}(p-1)-n_{\beta^{'}},~n_{\beta^{'}}=(p-1)l_{k}+\beta^{'}$
			\item $m_{2s}=m^{t^{'}}_{2s}=p^{k}t^{'}$
			\item $m_{2s-1}=n^{t^{'}}_{2s-1}=p^{k}(p-1)-n_{t^{'}},~n_{t^{'}} = (\alpha+p^{k-1}t^{'})(p-1)+\beta$
		\end{itemize}
		where $x^{s}_{k}=p^{k}[sp-(s+1)],$ $l^{s}_{k}=\alpha^{s-1}_{k} p+\beta^{s-1}_{k},$ $0\leq \alpha^{s-1}_{k}<p^{k-1},$ $0\leq \beta^{s-1}_{k},\beta^{'},t^{'}\leq p-1,$ $m=p^{k}[(2s-1)p-2(s-1)]$ and $m^{'}=p^{k}[2(s-1)p-(2s-1)].$
		\begin{figure}[h!]
			\centering
			\includegraphics[width=12cm,height=6cm]{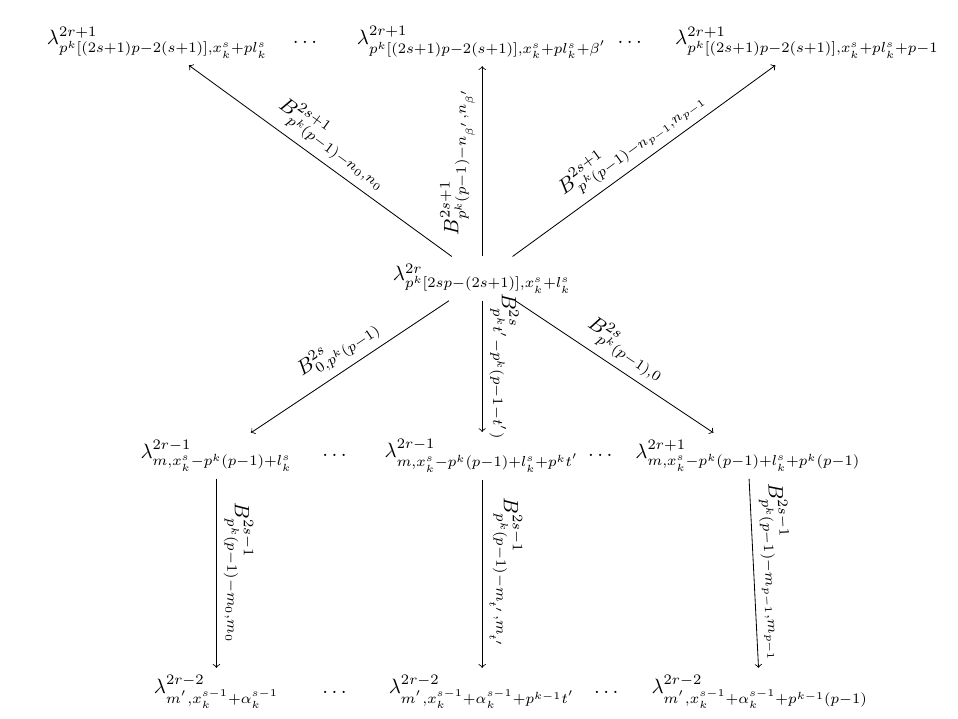}
		\end{figure}  
		Then
		\begin{enumerate}
			\item if $l^{s}_{k}<\frac{p^{k}-1}{2},$ then there is a descent at $2s-1$  of the path $P^{2r+1}_{l^{s}_{k},\beta^{'},t^{'}}$ only for  $0\leq t^{'}\leq\frac{p-1}{2}$ and $0\leq\beta^{'}\leq p-1$.   
			\item if $l^{s}_{k}\geq \frac{p^{k}-1}{2}$ then there is a descent at $2s-1$  of the path $P^{2r+1}_{l^{s}_{k},\beta^{'},t^{'}}$ only for $0\leq t^{'}<\frac{p-1}{2}$ and $0\leq\beta^{'}\leq p-1$.   
			\item if $j^{k}_{t-1}< l^{s}_{k}<j^{k}_{t},$ then there is a descent at $2s$ of the path $P^{2r+1}_{l^{s}_{k},\beta^{'},t^{'}}$ only for $p-1-t\leq t^{'}\leq p-1$ and $0\leq \beta^{'}\leq p-1.$
			\item if $ l^{s}_{k}=j^{k}_{t},$ then there is a descent at $2s$ of the path $P^{2r+1}_{l^{s}_{k},\beta^{'},t^{'}}$ only for $p-t\leq t^{'}\leq p-1$ when $\beta^{'}\leq t$ and $p-t-1\leq t^{'}\leq p-1$ when $\beta^{'}>t.$
		\end{enumerate} 
	\end{thm}	
	\begin{proof}
		\begin{enumerate}
			\item  When $l^{s}_{k}<\frac{p^{k}-1}{2},$ we have $l^{s}_{k}+p^{k}t^{'}<\frac{p^{k+1}-1}{2}$ for $0\leq t^{'}\leq \frac{p-1}{2}$ (i.e) $$\lambda^{2r-1}_{p^{k}[(2s-1)p-2(s-1)],x^{s}_{k}-p^{k}(p-1)+l^{s}_{k}+p^{k}t^{'}}\trianglerighteq\lambda^{2r-1}_{p^{k}[(2s-1)p-2(s-1)],x^{s}_{k}-p^{k}(p-1)+\frac{p^{k+1}-1}{2}}$$
			By Lemma \ref{thmlem}, we have $m^{t^{'}}_{2s-1}>\frac{p^{k}(p-1)}{2}$ and $m^{t^{'}}_{2s}\leq\frac{p^{k}(p-1)}{2},$ when $0\leq t^{'}\leq \frac{p-1}{2}.$ This implies that $$m^{t^{'}}_{2s-1}>m^{t^{'}}_{2s},~\forall~0\leq t^{'}\leq \frac{p-1}{2}$$
			Therefore, there is a descent at $2s-1$ of the path $P^{2r+1}_{l^{s}_{k},\beta^{'},t^{'}}$ for all such $t^{'}.$
			\item  When $l^{s}_{k}\geq\frac{p^{k}-1}{2},$ we have $l^{s}_{k}+p^{k}t^{'}<\frac{p^{k+1}-1}{2}$ for $0\leq t^{'}< \frac{p-1}{2}$ (i.e) $$\lambda^{2r-1}_{p^{k}[(2s-1)p-2(s-1)],x^{s}_{k}-p^{k}(p-1)+l^{s}_{k}+p^{k}t^{'}}\trianglerighteq\lambda^{2r-1}_{p^{k}[(2s-1)p-2(s-1)],x^{s}_{k}-p^{k}(p-1)+\frac{p^{k+1}-1}{2}}$$
			By Lemma \ref{thmlem}, we have $m^{t^{'}}_{2s-1}>\frac{p^{k}(p-1)}{2}$ and $m^{t^{'}}_{2s}\leq\frac{p^{k}(p-1)}{2},$ when $0\leq t^{'}< \frac{p-1}{2}.$ This implies that $$m^{t^{'}}_{2s-1}>m^{t^{'}}_{2s},~\forall~0\leq t^{'}< \frac{p-1}{2}$$
			Therefore, there is a descent at $2s-1$ of the path $P^{2r+1}_{l^{s}_{k},\beta^{'},t^{'}}$ for all such $t^{'}.$
			
			\noindent For $j^{k}_{t-1}<l^{s}_{k}<j^{k}_{t},~1\leq t \leq p-1$ and $j^{k+1}_{t}$ in the floor $2r-1,$ we observe the following:  
			
			\noindent When $t^{'}\leq t-1,$ we have 
			\begin{align}\label{t^{'}<t}
				j^{k}_{t-1}+p^{k}t^{'}&<l^{s}_{k}+p^{k}t^{'}<j^{k}_{t}+p^{k}t^{'} \notag \\ j^{k+1}_{t^{'}}&<l^{s}_{k}+p^{k}t^{'}<j^{k+1}_{t^{'}+1}
			\end{align}
			
			\noindent When $t^{'}\geq t,$ we have 
			\begin{align}\label{t^{'}>t}
				j^{k}_{t-1}+p^{k}t^{'}&<l^{s}_{k}+p^{k}t^{'}<j^{k}_{t}+p^{k}t^{'} \notag \\j^{k+1}_{t^{'}-1}&<l^{s}_{k}+p^{k}t^{'}<j^{k+1}_{t^{'}}
			\end{align}
			
			Note that $j^{k+1}_{t-1}<l^{s}_{k}+p^{k}t,~l^{s}_{k}+p^{k}(t-1)<j^{k+1}_{t}$, since $l^{s}_{k}<j^{k}_{t}$ this implies that $l^{s}_{k}+p^{k}t<j^{k}_{t}+p^{k}t=j^{k+1}_{t}.$ 
			
			Similarly, for $l^{s}_{k}=j^{k}_{t}$ and $j^{k+1}_{t}$ in the floor $2r-1,$ we observe the following:  
			
			\noindent When $t^{'}< t$ and $t^{'}> t,$ the number $j^{k}_{t}+p^{k}t^{'}$ is bounded by the same numbers as the above case. And for $t^{'}= t,$ we have 
			\begin{align}\label{t^{'}=t}
				j^{k}_{t}+p^{k}t=j^{k+1}_{t}	
			\end{align}
			\item If $j^{k}_{t-1}<l^{s}_{k}<j^{k}_{t},$ then $j^{k+1}_{t-1}<pl^{s}_{k}+\beta^{'}<j^{k+1}_{t}$  ($j^{k+1}_{t}$ is considered in the floor $2r+1,~\forall~ 1\leq t\leq p-1$) and by Lemma 2.3 and Definition \ref{block1}, we have the following
			\begin{align} \label{thm3}
				m^{j_{t-1}}_{2s+1}>m^{l^{s}_{k}}_{2s+1}>m^{j_{t}}_{2s+1}
			\end{align}
			where $m^{j_{t}}_{2s+1}=p^{k}(p-t)$ and $m^{l^{s}_{k}}_{2s+1}=p^{k}(p-1)-((p-1)l^{s}_{k}+\beta^{'}).$ There is a descent at $2s,$ when $m^{t^{'}}_{2s}> m^{l^{'}_{k}}_{2s+1}$ and from equation (\ref{thm3}), we have  
			$$m^{t^{'}}_{2s}\geq m^{j_{t-1}}_{2s+1}>m^{l^{s}_{k}}_{2s+1}\implies p^{k}t^{'}\geq p^{k}(p-t)$$ The above inequality holds for $t^{'} = p-1,p-2,\dots, p-t.$ Therefore, for such $t^{'}$, there is a descent at $2s$ of the path $P^{2r+1}_{l^{s}_{k},\beta^{'},t^{'}}$.
			
			\item If $l^{s}_{k}=j^{k}_{t},$ then $j^{k+1}_{t-1}<pj^{k}_{t}+\beta^{'}\leq j^{k+1}_{t}$ when $\beta^{'}\leq t$ and $j^{k+1}_{t}<pj^{k}_{t}+\beta^{'}<j^{k+1}_{t+1}$ when $\beta^{'}>t.$ This implies that 
			\begin{align}
				m^{j_{t-1}}_{2s+1}>m^{l^{s}_{k}}_{2s+1}\geq  m^{j_{t}}_{2s+1},~\text{when}~\beta^{'}\leq t \label{thm4.1}\\
				m^{j_{t}}_{2s+1}>m^{l^{s}_{k}}_{2s+1}>m^{j_{t+1}}_{2s+1},~\text{when}~\beta^{'}> t\label{thm4.2}			
			\end{align}
			where $m^{j_{t}}_{2s+1}=p^{k}(p-1-t)$ and $m^{l^{s}_{k}}_{2s+1}=p^{k}(p-1)-((p-1)j^{k}_{t}+\beta^{'}).$ There is a descent at $2s$ of the path $P^{2r+1}_{l^{s}_{k},\beta^{'},t^{'}}$ with $\beta^{'}\leq t,$ when $m^{t^{'}}_{2s}> m^{l^{s}_{k}}_{2s+1}$ and from equation (\ref{thm4.1}), we have
			$$m^{t^{'}}_{2s}\geq m^{j_{t-1}}_{2s+1}>m^{l^{s}_{k}}_{2s+1}\implies p^{k}t^{'}\geq p^{k}(p-t)$$ The above inequality holds for $t^{'} = p-1,p-2,\dots, p-t.$ Therefore, for such $t^{'}$, there is a descent at $2s$ of the path $P^{2r+1}_{l^{s}_{k},\beta^{'},t^{'}}$.
			
			There is a descent at $2s$ of the path $P^{2r+1}_{l^{s}_{k},\beta^{'},t^{'}}$ with $\beta^{'}> t,$ when $m^{t^{'}}_{2s}> m^{l^{s}_{k}}_{2s+1}$ and from equation (\ref{thm4.2}), we have
			$$m^{t^{'}}_{2s}\geq m^{j_{t}}_{2s+1}>m^{l^{s}_{k}}_{2s+1}\implies p^{k}t^{'}\geq p^{k}(p-1-t)$$ The above inequality holds for $t^{'} = p-1,p-2,\dots, p-t-1.$ Therefore, for such $t^{'}$, there is a descent at $2s$ of the path $P^{2r+1}_{l^{s}_{k},\beta^{'},t^{'}}$.
			
		\end{enumerate}		
		Note that, when $l^{s}_{k}=j^{k}_{0}$ and $\beta^{'}=0$ there is no descent at $2s$ for any $t^{'}.$
	\end{proof}
	\begin{rem}\label{k=0}
		For $k=0,$ and $s\geq 2,$ let $P^{2r+1}_{t,\beta^{'}}$ denote the paths $(\lambda^{1}_{p-1,i},m_{2},\dots,m_{2s+1},$ $\lambda^{2r+1}_{(2s+1)p-2(s+1),x^{s}_{0}+\beta^{'}})$ in the $p$-Bratteli diagram such that 
		$m_{2s+1}=(p-1-\beta^{'}),$ $m_{2s}=t$ and $m_{2s-1}=p-1-t,$ where $0\leq t, \beta^{'}\leq p-1.$ 
		\begin{enumerate}
			\item  There is a descent at $2s-1$ of the paths $P^{2r+1}_{t,\beta^{'}}$ whenever $0\leq t \leq \frac{p-3}{2}.$ \label{aa}
			\item 	  For $\beta^{'}\neq 0,$ we have $t>p-1-\beta^{'}$ for all $t=p-1,p-2,\dots,p-\beta^{'},$  there is a descent at $2s$ of the paths $P^{2r+1}_{t,\beta^{'}}.$ And for $\beta^{'}=0,$ there is no descent at $2s$ of the paths $P^{2r+1}_{t,0}.$ \label{bb}
			
		\end{enumerate}
	\end{rem}
	
	\begin{ex}\label{eg1}
		We now discuss the Theorem \ref{rules l} for $p=5$ and the particular paths $(\lambda^{5}_{5^{2}(4),k_{1}},m_{2},\dots,m_{7},\lambda^{11}_{5^{2}(27),x^{3}_{2}+5l^{3}_{2}+\beta^{'}}),$ where $x^{3}_{2}=175,~l^{3}_{2}=7,~\beta^{'}=1,~m_{7}=71,m_{6}=25t^{'}$ and $m_{5}=100-4(1+5t^{'})-2$ and this path is denoted by $P^{11}_{7,\beta^{'},t^{'}}.$
		
		Since $l^{3}_{2}<12,$ then by Theorem \ref{rules l}(1) there is a descent at 5 only for $0\leq t^{'}\leq2$ of the path $P^{11}_{7,\beta^{'},t^{'}}$  (i.e) 
		\begin{itemize}
			\item  when $t^{'}=0,$ we have $m_{5}=94$ and $m_{6}=0$ $\implies$ $m_{5}>m_{6}$ 
			
			\item	when $t^{'}=1,$ we have $m_{5}=74$ and $m_{6}=25$ $\implies$ $m_{5}>m_{6}$  
			
			\item	when $t^{'}=2,$ we have $m_{5}=54$ and $m_{6}=50$ $\implies$ $m_{5}>m_{6}$

			\item	when $t^{'}=3,$ we have $m_{5}=34$ and $m_{6}=75$ $\implies$ $m_{5}<m_{6}$  
			
			\item	when $t^{'}=4,$ we have $m_{5}=14$ and $m_{6}=100$ $\implies$ $m_{5}<m_{6}$  
			
		\end{itemize}
		
		Given that $j^{2}_{1}\leq l_{2}\leq j^{2}_{2}(6<7<12),$ then by Theorem \ref{rules l}(3)
		there is a descent at 6 only for $3\leq t^{'}\leq 4$ and for all $0\leq \beta^{'}\leq 4$ (i.e)
		
		when $\beta^{'}= 0,$ we have $m^{0}_{7}=72,$ we need to find $t^{'}$ such that  $25t^{'}>72,$ the possible $t^{'}$'s are 4,3.
		
		when $\beta^{'}= 1,$ we have $m^{1}_{7}=71,$ we need to find $t^{'}$ such that  $25t^{'}>71,$ the possible $t^{'}$'s are 4,3.
		
		Similarly, for any $0\leq \beta^{'}\leq 4,$ there is a descent at 6 only for $t^{'}=4,3$ of the path $P^{11}_{7,\beta^{'},t^{'}}.$

	\end{ex}
	\section{$p^{k}$-Eulerian polynomials, $k\geq 0$}
	In this section, we define the $p^{k}$-Eulerian polynomial for each vertex in the set $V^{2r}_{k}$ (and $W^{2r}_{k}$), where $k\geq 0$. We give the inductive method to construct the $p^{k}$-Eulerian polynomial for each vertex. Additionally, we compute these polynomials for the initial stages of the path.
	
	Let $P^{2r}_{i,m_{2},\dots,m_{s},l^{s}_{k}}$ (or $P^{2r+1}_{i,m_{2},\dots,m_{s},l^{s}_{k+1}}$) denote the path $$(\lambda^{2k+1}_{p^{k}(p-1),i},m_{2},m_{3},\dots,m_{2s},\lambda^{2r}_{p^{k}[2sp-(2s+1)],x^{s}_{k}+l^{s}_{k}})$$ (or) $$(\lambda^{2k+1}_{p^{k}(p-1),i},m_{2},m_{3},\dots,m_{2s+1},\lambda^{2r+1}_{p^{k}[(2s+1)p-2(s+1)],x^{s}_{k}+l^{s}_{k+1}}).$$ This is the path starting at the vertex $\lambda^{2k+1}_{p^{k}(p-1),i}$ and ending at the vertex $\lambda^{2r}_{p^{k}[2sp-(2s+1)],x^{s}_{k}+l^{s}_{k}}$ (or $\lambda^{2r+1}_{p^{k}[(2s+1)p-2(s+1)],x^{s}_{k}+l^{s}_{k+1}}$) depending on $m_{2},\dots,m_{2s},m_{2s+1},$ where $0\leq i<p^{k}(p-1),$ $0\leq l^{s}_{k}<p^{k}$ with $x^{s}_{k}=p^{k}[sp-(s+1)],$ and $m_{d}$'s are same as in Remark \ref{syt def} for $2\leq d \leq 2s+1.$ 
	
	Let $P^{2r}_{i,l^{s}_{k}}$ (or $P^{2r+1}_{i,l^{s}_{k+1}}$) denote the set of all paths starting at the vertex $\lambda^{2k+1}_{p^{k}(p-1),i}$ and ending at the vertex $\lambda^{2r}_{p^{k}[2sp-(2s+1)],x^{s}_{k}+l^{s}_{k}}$ (or $\lambda^{2r+1}_{p^{k}[(2s+1)p-2(s+1)],x^{s}_{k}++l^{s}_{k+1}}$),  which are obtained by adding the blocks $B^{d}_{m_{d},n_{d}}$ successcively. Let us denote $P^{2r}_{l^{s}_{k}}=\cup_{i}P^{2r}_{i,l^{s}_{k}}$ (or $P^{2r+1}_{l^{s}_{k+1}}=\cup_{i}P^{2r+1}_{i,l^{s}_{k+1}}$), $0\leq i<p^{k}(p-1).$
	\begin{Def}\label{eulerian polynomial}
		We now define the $p^{k}$-Eulerian polynomial for the descents of paths belongs to the set $P^{2r}_{l^{s}_{k}}$ (and $P^{2r+1}_{l^{s}_{k+1}}$) as $$\mathcal{F}^{\lambda}(q)=\sum_{P^{2r}_{i,m_{2},\dots,m_{2s},l^{s}_{k}}\in P^{2r}_{l^{s}_{k}}} q^{des(P^{2r}_{i,m_{2},\dots,m_{2s},l^{s}_{k}})}$$  $$\mathcal{F}^{\lambda}(q)=\sum_{P^{2r+1}_{i,m_{2},\dots,m_{2s+1},l^{s}_{k+1}}\in P^{2r+1}_{l^{s}_{k+1}}} q^{des(P^{2r+1}_{i,m_{2},\dots,m_{2s+1},l^{s}_{k+1}})}$$ where $\lambda=\lambda^{2r}_{p^{k}[2sp-(2s+1)],x^{s}_{k}+l^{s}_{k}}$ (or $\lambda^{2r+1}_{p^{k}[(2s+1)p-2(s+1)],x^{s}_{k}+l^{s}_{k+1}}$). It is equivalent to say $\mathcal{F}^{\lambda}(q)$ is the $p^{k}$-Eulerian polynomial corresponding to the vertex $\lambda.$
	\end{Def}
	The total number of paths having exactly $'g'$ descents will be called as $p^{k}$-Eulerian number, which will be calculated in subsequent paper.  
	
	The derivative of the $p^{k}$-Eulerian polynomial evaluated at 1 will be called as $p^{k}$-Fibonacci numbers(defined in the next section).
	\subsection{Method of constructing the $p^{k}$-Eulerian polynomial} \label{method}
	Assume that we have constructed the $p^{k}$-Eulerian polynomial upto $(2r-1)$-th floor (i.e) for all paths ending at a vertex on the $(2r-1)$-th floor.  Let $\mathcal{F}^{\lambda^{2r-1}_{p^{k}((2s-1)p-2s),x^{s}_{k}-p^{k}(p-1)+l^{s-1}_{k+1}}}(q)$ be the $p^{k}$-Eulerian polynomial for each vertex on the floor $2r-1,$ where $x^{s}_{k}=p^{k}[sp-(s+1)]$ and $0\leq l^{s-1}_{k+1}<p^{k+1}.$ For simplicity we denote $\mathcal{F}^{\lambda^{2r-1}_{p^{k}((2s-1)p-2s),x^{s}_{k}-p^{k}(p-1)+l^{s-1}_{k+1}}}(q)$ as $Q_{l^{s-1}_{k+1}},$ where $k\geq1$ and $ r\geq 2.$ Now we define the $p^{k}$-Eulerian polynomial on the $2r$-th floor and $(2r+1)$-th floor as follows.
	
	\noindent \textbf{Type I:}
	
	First we define the $p^{k}$-Eulerian polynomial for each vertex on the $2r$-th floor. Let $\lambda^{2r}_{p^{k}(2sp-(2s+1)),x^{s}_{k}+l^{s}_{k}}$ be the vertex on $2r$-th floor, where $x^{s}_{k}=p^{k}[sp-(s+1)]$ and $0\leq l^{s}_{k}<p^{k}.$
	
	\noindent \textbf{Case 1:}
	When $0\leq l^{s}_{k}<\frac{p^{k}(p-1)}{2}.$
	
	\begin{figure}[h!]
		\begin{center}
		\includegraphics{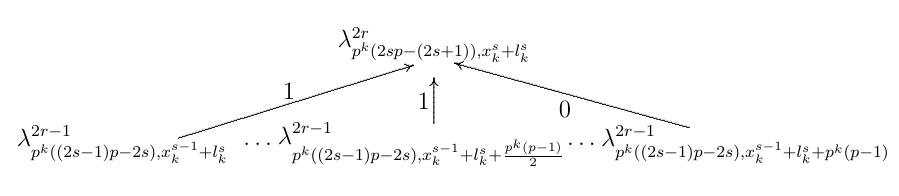}
			\end{center}
	\end{figure}
	where $x^{s-1}_{k}=x^{s}_{k}-p^{k}(p-1).$
	
	By Theorem \ref{rules l}(1), we have found the descent at $2s-1$ for each path ending at $\lambda^{2r}_{p^{k}(2sp-(2s+1)),x^{s}_{k}+l^{s}_{k}}.$ Whenever there is a descent at $2s-1$ on the $2r$-th floor of the path $P^{2r+1}_{l^{s}_{k},\beta^{'},t^{'}},$ we multiply the respective polynomial on the $(2r-1)$-th floor by $q.$ Then, the $p^{k}$-Eulerian polynomial would be 
	\begin{align}\label{cons1}
		\mathcal{F}^{\lambda^{2r}_{p^{k}[2sp-(2s+1)],x^{s}_{k}+l^{s}_{k}}}(q) =\sum_{t^{'}=0}^{\frac{p-1}{2}}q\cdot Q_{l^{s}_{k}+p^{k}t^{'}}(q)+\sum^{p-1}_{t^{'}=\frac{p+1}{2}} Q_{l^{s}_{k}+p^{k}t^{'}}(q) 		
	\end{align}
	
	\noindent \textbf{Case 2:}
	When $\frac{p^{k}(p-1)}{2} \leq l^{s}_{k}<p^{k}.$
	
	\begin{figure}[h!]
		\begin{center}
	\includegraphics{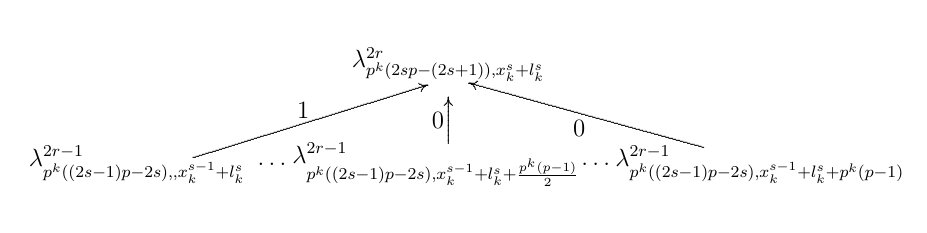}
		\end{center}
	\end{figure}
	where $x^{s-1}_{k}=x^{s}_{k}-p^{k}(p-1).$
	
	By Theorem \ref{rules l}(2), we have found the descent at $2s-1$ for each path ending at $\lambda^{2r}_{p^{k}(2sp-(2s+1)),x^{s}_{k}+l^{s}_{k}}.$ Similarly, we multiply the respective polynomial on the $(2r-1)$-th floor by $q,$ then the $p^{k}$-Eulerian polynomial would be 
	\begin{align}\label{cons2}
		\mathcal{F}^{\lambda^{2r}_{p^{k}[2sp-(2s+1)],x^{s}_{k}+l^{s}_{k}}}(q) =\sum_{t^{'}=0}^{\frac{p-3}{2}}q\cdot Q_{l^{s}_{k}+p^{k}t^{'}}(q)+\sum^{p-1}_{t^{'}=\frac{p-1}{2}} Q_{l^{s}_{k}+p^{k}t^{'}}(q) 		
	\end{align}
	
	\noindent \textbf{Type II:}
	
	Now we define the $p^{k}$-Eulerian polynomial for each vertex on the $(2r+1)$-th floor. Let $\lambda^{2r+1}_{p^{k}((2s+1)p-2(s+1)),x^{s}_{k}+pl^{s}_{k}+\beta^{'}}$ be the vertex on $(2r+1)$-th floor, where $0\leq \beta^{'}\leq p-1.$
	
	\noindent \textbf{Case 1:}
	
	When $j^{k}_{t-1}<l^{s}_{k}<j^{k}_{t}$ and $0\leq \beta^{'}\leq p-1.$
	
	\begin{figure}[h!]
		\begin{center}
	\includegraphics{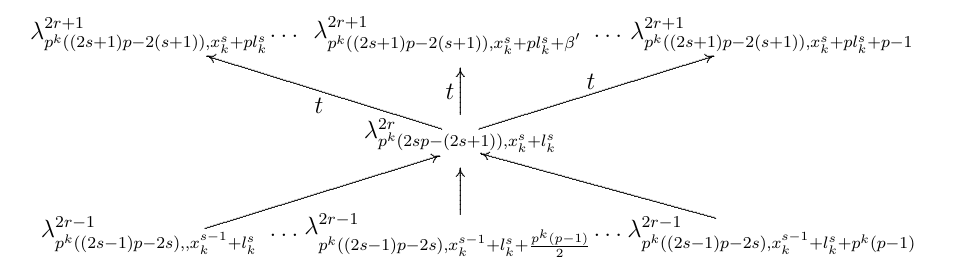}
		\end{center}
	\end{figure}
	
	By Theorem \ref{rules l}(3), we have found the descent at $2s$ for each path ending at the vertex $\lambda^{2r+1}_{p^{k}((2s+1)p-2(s+1)),x^{s}_{k}+pl^{s}_{k}+\beta^{'}}.$ Whenever there is a descent at $2s$ on the $(2r+1)$-th floor of the path $P^{2r+1}_{l^{s}_{k},\beta^{'},t^{'}},$ we multiply the respective polynomial on the $(2r-1)$-th floor by $q$ into the last $t$ terms of the sum in equation $(\ref{cons1})$ or (\ref{cons2}). Then, the $p^{k}$-Eulerian polynomial would be 
	\begin{align}\label{cons3}
		\mathcal{F}^{\lambda^{2r+1}_{p^{k}[(2s+1)p-2(s+1)],x^{s}_{k}+pl^{s}_{k}+\beta^{'}}}(q) =\sum_{t^{'}=0}^{\frac{p-1}{2}}q\cdot Q_{l^{s}_{k}+p^{k}t^{'}}(q)+\sum^{p-1-t}_{t^{'}=\frac{p+1}{2}} Q_{l^{s}_{k}+p^{k}t^{'}}(q) +\sum^{p-1}_{t^{'}=p-t} q\cdot Q_{l^{s}_{k}+p^{k}t^{'}}(q) 		
	\end{align}
	where $0< l^{s}_{k}<\frac{p^{k}(p-1)}{2}.$
	\begin{align}\label{cons4}
		\mathcal{F}^{\lambda^{2r+1}_{p^{k}[(2s+1)p-2(s+1)],x^{s}_{k}+pl^{s}_{k}+\beta^{'}}}(q) =\sum_{t^{'}=0}^{p-t-1}q\cdot Q_{l^{s}_{k}+p^{k}t^{'}}(q)+\sum^{\frac{p-3}{2}}_{t^{'}=p-t} q^{2}\cdot Q_{l^{s}_{k}+p^{k}t^{'}}(q) +q\cdot \sum^{p-1}_{t^{'}=\frac{p-1}{2}} Q_{l^{s}_{k}+p^{k}t^{'}}(q)		
	\end{align}
	where $\frac{p^{k}(p-1)}{2}< l^{s}_{k}<p^{k}.$
	
	\noindent \textbf{Case 2:}
	
	When $l^{s}_{k}=j^{k}_{t},$ there will be two cases one is $\beta^{'}\leq t$ and $\beta^{'}>t.$  
	
	\begin{figure}[h!]
		\begin{center}
		\includegraphics{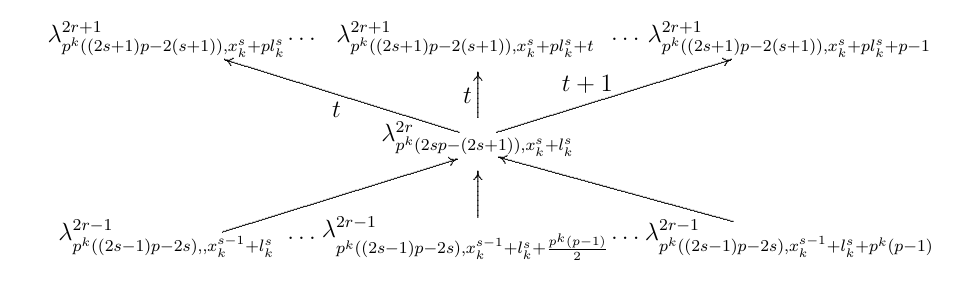}
		\end{center}
	\end{figure}
	
	By Theorem \ref{rules l}(4), we have found the descent at $2s$ for each path ending at the vertex $\lambda^{2r+1}_{p^{k}((2s+1)p-2(s+1)),x^{s}_{k}+pl^{s}_{k}+\beta^{'}}.$ The $p^{k}$-Eulerian polynomial is same as above case when $\beta^{'}\leq t,~t\neq 0.$ For $l^{s}_{k}=j^{0}_{0},$ there is no descent at $2s$ for any path, so the polynomial would be same as equation $(\ref{cons1}).$  
	
	And for $\beta^{'}>t,$ we multiply the respective polynomial on the $(2r-1)$-th floor by $q$ into the last $t+1$ terms of the sum in equation $(\ref{cons1})$ or (\ref{cons2}). Then, the $p^{k}$-Eulerian polynomial would be 
	
	\begin{align}\label{cons5}
		\mathcal{F}^{\lambda^{2r+1}_{p^{k}[(2s+1)p-2(s+1)],x^{s}_{k}+pl^{s}_{k}+\beta^{'}}}(q) =\sum_{t^{'}=0}^{\frac{p-1}{2}}q\cdot Q_{l^{s}_{k}+p^{k}t^{'}}(q)+\sum^{p-2-t}_{t^{'}=\frac{p+1}{2}} Q_{l^{s}_{k}+p^{k}t^{'}}(q) +\sum^{p-1}_{t^{'}=p-t-1} q\cdot Q_{l^{s}_{k}+p^{k}t^{'}}(q) 		
	\end{align}
	where $0\leq l^{s}_{k}<\frac{p^{k}(p-1)}{2}.$
	\begin{align}\label{cons6}
		\mathcal{F}^{\lambda^{2r+1}_{p^{k}[(2s+1)p-2(s+1)],x^{s}_{k}+pl^{s}_{k}+\beta^{'}}}(q) =\sum_{t^{'}=0}^{p-t-2}q\cdot Q_{l^{s}_{k}+p^{k}t^{'}}(q)+\sum^{\frac{p-3}{2}}_{t^{'}=p-t-1} q^{2}\cdot Q_{l^{s}_{k}+p^{k}t^{'}}(q) +q\cdot \sum^{p-1}_{t^{'}=\frac{p-1}{2}} Q_{l^{s}_{k}+p^{k}t^{'}}(q)		
	\end{align}
	where $\frac{p^{k}(p-1)}{2}\leq l^{s}_{k}<p^{k}.$

	\subsubsection*{Computing the $p^{k}$-Eulerian polynomial for initial stages:}	
	Now we compute the $p^{k}$-Eulerian polynomial for the vertices at the floors $2k+2,~2k+3,~2k+4,2k+5,~2k+6,~2k+7$ and $2k+8$ for all $k\geq 1,$ by using the method defined in Subsection \ref{method}.
	
	Consider the paths $(\lambda^{2k+1}_{p^{k}(p-1),i},m_{2},\lambda^{2(k+1)}_{p^{k}(2p-3),x^{1}_{k}+l^{1}_{k}})$,where $x^{1}_{k}=p^{k}(p-2)$ and $0\leq l^{1}_{k}<p^{k}.$ For each $l^{1}_{k},$ there are $p-1$ paths ending at the vertex $\lambda^{2(k+1)}_{p^{k}(2p-3),x^{1}_{k}+l^{1}_{k}}.$ By Lemma \ref{2p-3}, there are $\frac{p-1}{2}$ paths having descent at 1 and $\frac{p-1}{2}$ paths do not have descent at 1. Hence, the corresponding $p^{k}$-Eulerian polynomial would be
	$$\mathcal{F}^{\lambda^{2(k+1)}_{p^{k}(2p-3),x^{1}_{k}+l^{1}_{k}}}(q) =\frac{p-1}{2}(q+1),~\forall~0\leq l^{1}_{k}<p^{k}.$$
	
	By convention \ref{c2}, there is no descent at 2 for all paths $(\lambda^{2k+1}_{p^{k}(p-1),i},m_{2},m_{3}, \lambda^{2k+3}_{p^{k}(3p-4),x^{1}_{k}+l^{1}_{k+1}}),$ where $0\leq l^{1}_{k+1}<p^{k+1},$ thus the $p^{k}$-Eulerian polynomial would be the same i.e.
	$$\mathcal{F}^{\lambda^{2k+3}_{p^{k}(3p-4),x^{1}_{k}+l^{1}_{k}}}(q) =\frac{p-1}{2}(q+1),~\forall~0\leq l^{1}_{k+1}<p^{k+1}.$$
	
	By using Theorem \ref{rules l} and the method defined in Subsection \ref{method}, we obtain the following polynomial for each path on every floor.
	
	The $p^{k}$-Eulerian polynomial at the vertex $\lambda^{2(k+2)}_{p^{k}(4p-5),x^{2}_{k}+l^{2}_{k}}$ where $x^{2}_{k}=p^{k}(2p-3)$ and $0\leq l^{2}_{k}<p^{k}$ is given by
	\begin{itemize}
		\item 	$\mathcal{F}^{\lambda^{2(k+2)}_{p^{k}(4p-5),x^{2}_{k}+l^{2}_{k}}}(q) =\frac{p-1}{2}(q+1)\left[\frac{p+1}{2}q+\frac{p-1}{2}\right],~\forall~0\leq l^{2}_{k}<\frac{p^{k}(p-1)}{2}$
		\item 	$	\mathcal{F}^{\lambda^{2(k+2)}_{p^{k}(4p-5),x^{2}_{k}+l^{2}_{k}}}(q) =\frac{p-1}{2}(q+1)\left[\frac{p-1}{2}q+\frac{p+1}{2}\right],~\forall~\frac{p^{k}(p-1)}{2}\leq l^{2}_{k}<p^{k}$
	\end{itemize}		
	The $p^{k}$-Eulerian polynomial  at the vertex  $\lambda^{2k+5}_{p^{k}(5p-6),x^{2}_{k}+l^{2}_{k+1}},$ where  $0\leq l^{2}_{k+1}<p^{k+1}$ is given by
	\begin{itemize}
		\item $\mathcal{F}^{\lambda^{2k+5}_{p^{k}(5p-6),x^{2}_{k}+j^{k+1}_{0}}}(q) =\frac{p-1}{2}(q+1)\left[\frac{p+1}{2}q+\frac{p-1}{2}\right]$ when $l^{2}_{k+1}=j^{k+1}_{0}.$
		\item
		$\mathcal{F}^{\lambda^{2k+5}_{p^{k}(5p-6),x^{2}_{k}+l^{2}_{k+1}}}(q) \frac{p-1}{2}(q+1)\left[\frac{p+(2t+1)}{2}q+\frac{p-(2t+1)}{2}\right]$ when $1\leq t\leq\frac{p-3}{2}$ and $j^{k+1}_{t-1}<l^{2}_{k+1}\leq j^{k+1}_{t}.$
		\item
		$\mathcal{F}^{\lambda^{2k+5}_{p^{k}(5p-6),x^{2}_{k}+l^{2}_{k+1}}}(q) =\frac{p-1}{2}(q+1)\left(pq\right)$ when $j^{k+1}_{\frac{p-3}{2}}<l^{2}_{k+1}<j^{k+1}_{\frac{p-1}{2}}-\frac{p-1}{2}.$ 
		\item 
		$	\mathcal{F}^{\lambda^{2k+5}_{p^{k}(5p-6),x^{2}_{k}+l^{2}_{k+1}}}(q)=\frac{p-1}{2}(q+1)\left((p-1)q+1\right)$ when $j^{k+1}_{\frac{p-1}{2}}-\frac{p-1}{2}\leq l^{2}_{k+1}\leq j^{k+1}_{\frac{p-1}{2}}.$
		\item $\mathcal{F}^{\lambda^{2k+5}_{p^{k}(5p-6),x^{2}_{k}+l^{2}_{k+1}}}(q) =\frac{p-1}{2}(q+1)\left(pq\right)$ when  $j^{k+1}_{\frac{p-1}{2}}<l^{2}_{k+1}\leq j^{2}_{\frac{p+1}{2}}.$
		\item 
		$\mathcal{F}^{\lambda^{2k+5}_{p^{k}(5p-6),x^{2}_{k}+l^{2}_{k+1}}}(q) =\frac{p-1}{2}(q+1)\left(\frac{-(p-2t+1)q^{2}}{2}+\frac{(3p-2t+1)}{2}\right)$ when $\frac{p-1}{2}\leq t\leq p-1$ and $j^{k+1}_{t-1}<l_{k+1}\leq j^{k+1}_{t}.$
	\end{itemize}
	The $p^{k}$-Eulerian polynomial at the vertex $\lambda^{2(k+3)}_{p^{k}(6p-7),x^{3}_{k}+l^{3}_{k}}$ where $x^{3}_{k}=p^{k}(4p-5)$ and $0\leq l^{3}_{k}<p^{k}$ is given by 	
	
	\noindent\textbf{Case 1:} $k=1$
	\begin{itemize}
		\item $\mathcal{F}^{\lambda^{8}_{p(6p-7),x^{3}_{1}+l^{3}_{1}}}(q) = \frac{p-1}{2}(q+1)\left[\left(\frac{p^{2}+2t-1}{2}\right)q^{2}+\left(\frac{p^{2}-2t+1}{2}\right)q\right]$ when $l^{3}_{1}=j^{1}_{t},~0\leq t \leq \frac{p-3}{2}.$
		\item $\mathcal{F}^{\lambda^{8}_{p(6p-7),x^{3}_{1}+l^{3}_{1}}}(q) = \frac{p-1}{2}(q+1)\left[\left(\frac{p^{2}-p}{2}\right)q^{2}+\left(\frac{p^{2}+p-2}{2}\right)q+1\right]$ when $l^{3}_{1}=j^{1}_{\frac{p-1}{2}}.$
		\item  $\mathcal{F}^{\lambda^{8}_{p(6p-7),x^{3}_{1}+l^{3}_{1}}}(q) = \frac{p-1}{2}(q+1)\left[\left(\frac{p^{2}+2t-1-2p}{2}\right)q^{2}+\left(\frac{p^{2}-2t+1+2p}{2}\right)q\right]$ when $l^{3}_{1}=j^{1}_{t},~\frac{p+1}{2}\leq t \leq p-1.$
	\end{itemize}
	\noindent\textbf{Case 2:} $k>1$
	\begin{itemize}
		\item $\mathcal{F}^{\lambda^{2(k+3)}_{p^{k}(6p-7),x^{3}_{k}+l^{3}_{k}}}(q) = \frac{p-1}{2}(q+1)\left[\left(\frac{p^{2}+1}{2}\right)q^{2}+\left(\frac{p^{2}-1}{2}\right)q\right]$ when $l_{k}=j^{k}_{0}.$ 
		\item $\mathcal{F}^{\lambda^{2(k+3)}_{p^{k}(6p-7),x^{3}_{k}+l^{3}_{k}}}(q) = \frac{p-1}{2}(q+1)\left[\left(\frac{p^{2}+2t+1}{2}\right)q^{2}+\left(\frac{p^{2}-2t-1}{2}\right)q\right]$ when $j^{k}_{t-1}<l^{3}_{k}\leq j^{k}_{t},~1\leq t\leq \frac{p-3}{2}.$
		\item $\mathcal{F}^{\lambda^{2(k+3)}_{p^{k}(6p-7),x^{3}_{k}+l^{3}_{k}}}(q) = \frac{p-1}{2}(q+1)\left[\left(\frac{p(p+1)}{2}\right)q^{2}+\left(\frac{p(p-1)}{2}\right)q\right]$ when $j^{k}_{\frac{p-3}{2}}<l^{3}_{k}< j^{k}_{\frac{p-1}{2}}-\frac{p-1}{2}.$
		\item $\mathcal{F}^{\lambda^{2(k+3)}_{p^{k}(6p-7),x^{3}_{k}+l^{3}_{k}}}(q) = \frac{p-1}{2}(q+1)\left[\left(\frac{p^{2}+p-2}{2}\right)q^{2}+\left(\frac{p^{2}-p+2}{2}\right)q\right]$ when $j^{k}_{\frac{p-1}{2}}-\frac{p-1}{2}\leq l^{3}_{k}< j^{k}_{\frac{p-1}{2}}.$ 
		\item $\mathcal{F}^{\lambda^{2(k+3)}_{p^{k}(6p-7),x^{3}_{k}+l^{3}_{k}}}(q) = \frac{p-1}{2}(q+1)\left[\left(\frac{p^{2}-p}{2}\right)q^{2}+\left(\frac{p^{2}+p-2}{2}\right)q+1\right]$ when $l^{3}_{k}=j^{k}_{\frac{p-1}{2}}.$ 
		\item $\mathcal{F}^{\lambda^{2(k+3)}_{p^{k}(6p-7),x^{3}_{k}+l^{3}_{k}}}(q) = \frac{p-1}{2}(q+1)\left[\left(\frac{p^{2}+2t-2p-1}{2}\right)q^{2}+\left(\frac{p^{2}-2t+2p+1}{2}\right)q\right]$ when $j^{k}_{t-1}<l^{3}_{k}\leq j^{k}_{t},~\frac{p-1}{2}\leq t\leq p-1.$
	\end{itemize}
	The $p^{k}$-Eulerian polynomial at the vertex $\lambda^{2k+7}_{p^{k}(7p-8),x^{3}_{k}+l^{3}_{k+1}},$ where $0\leq l^{3}_{k+1}<p^{k+1}$ is given by
	
	\noindent \textbf{Case 1:} $k=1$
	\begin{itemize}
		\item $\mathcal{F}^{\lambda^{9}_{p(7p-8),x^{3}_{1}+l^{3}_{2}}}(q) = \frac{p-1}{2}(q+1)\left[\left(\frac{p^{2}-1}{2}\right)q^{2}+\left(\frac{p^{2}+1}{2}\right)q\right],$ when $l^{3}_{2}=j^{2}_{0}.$
		\item $\mathcal{F}^{\lambda^{9}_{p(7p-8),x^{3}_{1}+l^{3}_{2}}}(q) = \frac{p-1}{2}(q+1)\left[\left(\frac{(p-2)t-t^2}{2}\right)q^{3}+\left(\frac{p^{2}+2t^{2}+3t-3}{2}\right)q^{2} +\left(\frac{p^{2}-(p+4)t-t^{2}+3}{2}\right)q\right]$ 
		
		when $j^{2}_{t-1}<l^{3}_{2}<j^{2}_{t}-t,~1\leq t \leq \frac{p-1}{2}.$
		\item $\mathcal{F}^{\lambda^{9}_{p(7p-8),x^{3}_{1}+l^{3}_{2}}}(q) = \frac{p-1}{2}(q+1)\left[\left(\frac{(p-2)t-t^2}{2}\right)q^{3}+\left(\frac{p^{2}+2t^{2}+3t-1}{2}\right)q^{2} +\left(\frac{p^{2}-(p+4)t-t^{2}+1}{2}\right)q\right]$ 
		
		when $j^{2}_{t}-t\leq l^{3}_{2}\leq j^{2}_{t},~1\leq t \leq \frac{p-3}{2}.$
		\item $\mathcal{F}^{\lambda^{9}_{p(7p-8),x^{3}_{1}+l^{3}_{2}}}(q) = \frac{p-1}{2}(q+1)\left[\left(\frac{p^{2}-4p+3}{8}\right)q^{3}+\left(\frac{6p^{2}-6}{8}\right)q^{2} +\left(\frac{p^{2}+4p-5}{8}\right)q+1\right]$ 
		
		when $j^{2}_{\frac{p-1}{2}}-\frac{p-1}{2}\leq l^{3}_{2}\leq j^{2}_{\frac{p-1}{2}}.$
		\item $\mathcal{F}^{\lambda^{9}_{p(7p-8),x^{3}_{1}+l^{3}_{2}}}(q) = \frac{p-1}{2}(q+1)\left[\left(\frac{p^{2}-4p+3}{8}\right)q^{3}+\left(\frac{6p^{2}+8p-14}{8}\right)q^{2} +\left(\frac{p^{2}+4p-11}{8}\right)q\right]$
		
		when $j^{2}_{\frac{p-1}{2}}< l^{3}_{2}< j^{2}_{\frac{p+1}{2}}-\frac{p+1}{2}.$
		\item $\mathcal{F}^{\lambda^{9}_{p(7p-8),x^{3}_{1}+l^{3}_{2}}}(q) =	 -\frac{p-1}{2}(q+1)\left[\left(\frac{p^{2}+t^2-(3p+4)t+4p+3}{2}\right)q^{3}-\left(\frac{-2t^{2}+(4p+6)t-3(p+1)^{2}}{2}\right)q^{2}\right]$ 
		
		\hspace{3.4cm}$-\frac{p-1}{2}(q+1)\left[\left(\frac{(t-2)(t-p)}{2}\right)q \right]	$
		
		when $j^{2}_{t-1}<l_{2}<j^{2}_{t}-t,~\frac{p+3}{2}\leq t \leq p-1.$ 
		\item $\mathcal{F}^{\lambda^{9}_{p(7p-8),x^{3}_{1}+l^{3}_{2}}}(q) =	 -\frac{p-1}{2}(q+1)\left[\left(\frac{p^{2}+t^2-(3p+4)t+4p+1}{2}\right)q^{3}-\left(\frac{-2t^{2}+(4p+6)t-3p^{2}-6p-1}{2}\right)q^{2}\right]$
		
		\hspace{3.4cm}$ -\frac{p-1}{2}(q+1)\left[\left(\frac{(t-2)(t-p)}{2}\right)q\right]$
		
		\noindent when $j^{2}_{t}-t\leq l^{3}_{2}\leq j^{2}_{t},~\frac{p+1}{2}\leq t \leq p-1.$	
	\end{itemize}
	\noindent\textbf{Case 2: $k>1$}
	\begin{itemize}
		\item $\mathcal{F}^{\lambda^{2k+7}_{p^{k}(7p-8),x^{3}_{k}+l^{3}_{k+1}}}(q) = \frac{p-1}{2}(q+1)\left[\left(\frac{p^{2}+1}{2}\right)q^{2}+\left(\frac{p^{2}-1}{2}\right)q\right]$ when $l^{3}_{k+1}=j^{k+1}_{0}.$
		\item $\mathcal{F}^{\lambda^{2k+7}_{p^{k}(7p-8),x^{3}_{k}+l^{3}_{k+1}}}(q) = \frac{p-1}{2}(q+1)\left[\left(\frac{(p-2)t-t^2}{2}\right)q^{3}+\left(\frac{p^{2}+2t^{2}+6t-1}{2}\right)q^{2} +\left(\frac{p^{2}+1-(p+4)t-t^{2}}{2}\right)q\right]$
		
		when $j^{k+1}_{t-1}<l^{3}_{k+1}<j^{k+1}_{t}-t,~1\leq t \leq \frac{p-1}{2}.$
		\item $\mathcal{F}^{\lambda^{2k+7}_{p^{k}(7p-8),x^{3}_{k}+l^{3}_{k+1}}}(q) = \frac{p-1}{2}(q+1)\left[\left(\frac{(p-2)t-t^2}{2}\right)q^{3}+\left(\frac{(p-1)^{2}+2t^{2}+6t}{2}\right)q^{2} +\left(\frac{p^{2}-1-(p+4)t-t^{2}}{2}\right)q\right]$ 
		
		when $j^{k+1}_{t}-t\leq l^{3}_{k+1}\leq j^{k+1}_{t},~1\leq t \leq \frac{p-3}{2}.$
		\item $\mathcal{F}^{\lambda^{2k+7}_{p^{k}(7p-8),x^{3}_{k}+l^{3}_{k+1}}}(q) = \frac{p-1}{2}(q+1)\left[\left(\frac{p^{2}-4p+3}{8}\right)q^{3}+\left(\frac{6p^{2}-8p+14}{8}\right)q^{2} +\left(\frac{p^{2}-4p+11}{8}\right)q\right],$ 
		
		when $j^{k+1}_{\frac{p-1}{2}}-\frac{p(p-1)}{2}\leq l^{3}_{k+1}< j^{k+1}_{\frac{p-1}{2}}-\frac{p-1}{2}.$
		\item $\mathcal{F}^{\lambda^{2k+7}_{p^{k}(7p-8),x^{3}_{k}+l^{3}_{k+1}}}(q) = \frac{p-1}{2}(q+1)\left[\left(\frac{p^{2}-4p+3}{8}\right)q^{3}+\left(\frac{6p^{2}-6}{8}\right)q^{2} +\left(\frac{p^{2}+4p-5}{8}\right)q+1\right],$
		
		when $j^{k+1}_{\frac{p-1}{2}}-\frac{p-1}{2}\leq l^{3}_{k+1}\leq j^{k+1}_{\frac{p-1}{2}}.$
		\item $\mathcal{F}^{\lambda^{2k+7}_{p^{k}(7p-8),x^{3}_{k}+l^{3}_{k+1}}}(q) = \frac{p-1}{2}(q+1)\left[\left(\frac{p^{2}-4p+3}{8}\right)q^{3}+\left(\frac{6p^{2}+8p-14}{8}\right)q^{2} +\left(\frac{p^{2}-4p+11}{8}\right)q\right],$ 
		
		when $j^{k+1}_{\frac{p-1}{2}}< l^{3}_{k+1}< j^{k+1}_{\frac{p+1}{2}}-\frac{p+1}{2}.$
		\item $\mathcal{F}^{\lambda^{2k+1}_{p^{k}(7p-8),x^{3}_{k}+l^{3}_{k+1}}}(q) =-\frac{p-1}{2}(q+1)\left[\left(\frac{p^{2}+t^2-(3p+4)t+4p+3}{2}\right)q^{3}-\left(\frac{-2t^{2}+(4p+6)t-3(p+1)^{2}}{2}\right)q^{2}\right]$ 
		
		\hspace{3.9cm}$-\frac{p-1}{2}(q+1)\left[\left(\frac{(t-2)(t-p)}{2}\right)q \right]	$
		
		when $j^{k+1}_{t-1}<l_{k+1}<j^{k+1}_{t}-t,~\frac{p+3}{2}\leq t \leq p-1.$ 
		\item $\mathcal{F}^{\lambda^{2k+1}_{p^{k}(7p-8),x^{3}_{k}+l^{3}_{k+1}}}(q) = -\frac{p-1}{2}(q+1)\left[\left(\frac{p^{2}+t^2-(3p+4)t+4p+1}{2}\right)q^{3}-\left(\frac{-2t^{2}+(4p+6)t-3p^{2}-6p-1}{2}\right)q^{2}\right]$
		
		\hspace{3.9cm}$ -\frac{p-1}{2}(q+1)\left[\left(\frac{(t-2)(t-p)}{2}\right)q\right]$
		
		\noindent when $j^{k+1}_{t}-t\leq l^{3}_{k+1}\leq j^{k+1}_{t},~\frac{p+1}{2}\leq t \leq p-1.$	
	\end{itemize}
	The $p^{k}$-Eulerian polynomial at the vertex $\lambda^{2(k+4)}_{p^{k}(8p-9),x^{4}_{k}+l^{4}_{k}},$ where $x^{4}_{k}=p^{k}(4p-5)$ and $0\leq l^{4}_{k}<p^{k}$ is given by
	
	\noindent \textbf{Case 1:} $k=1$
	\begin{itemize}
		\item $\mathcal{F}^{\lambda^{10}_{p(8p-9),x^{4}_{1}+l^{4}_{1}}}(q) = \frac{p-1}{2}(q+1)\left[\left(\frac{p^{3}-3p^{2}-p+3+12t(p-2)-12t^{2}}{24}\right)q^{4}+\left(\frac{11p^{3}+3p^{2}-35p+21+24t(t+2)}{24}\right)q^{3}\right]$
		
		\hspace{3.3cm}$+\frac{p-1}{2}(q+1)\left[\left(\frac{11p^{3}+3p^{2}+37p-51-12t(t+2)}{24}\right)q^{2}+\left(\frac{p^{3}-3p^{2}-p+27}{24}\right)q\right]$ 
		
		\noindent when $l^{4}_{1}=j^{1}_{t},~0\leq t \leq \frac{p-3}{2}.$
		\item $\mathcal{F}^{\lambda^{10}_{p(8p-9),x^{4}_{1}+l^{4}_{1}}}(q) = \frac{p-1}{2}(q+1)\left[\left(\frac{p^{3}-3p^{2}-p+3}{24}\right)q^{4}+\left(\frac{11p^{3}-6p^{2}-35p+30}{24}\right)q^{3}+\left(\frac{11p^{3}+9p^{2}+25p-45}{24}\right)q^{2}\right]$
		
		\hspace{3.3cm} $+\frac{p-1}{2}(q+1)\left[\left(\frac{p^{3}+11p-12}{24}\right)q+1\right]$ 
		
		\noindent when $l^{4}_{1}=j^{1}_{\frac{p-1}{2}}.$
		
		\item  $\mathcal{F}^{\lambda^{10}_{p(8p-9),x^{4}_{1}+l^{4}_{1}}}(q) =\frac{p-1}{2}(q+1)\left[\left(\frac{p^{3}-3p^{2}-p+3}{24}\right)q^{4}+\left(\frac{11p^{3}-21p^{2}-59p+21-12t(t-2)+36tp}{24}\right)q^{3}\right]$
		
		\hspace{3cm}$+\frac{p-1}{2}(q+1)\left[\left(\frac{11p^{3}+27p^{2}+85p-51+24t(t-2)-48tp}{24}\right)q^{2}+\left(\frac{p^{3}-3p^{2}-25p+27-12t(t-2)+12tp}{24}\right)q\right]$ 
		
		when $l^{4}_{1}=j^{1}_{t},~\frac{p+1}{2}\leq t \leq p-1.$
	\end{itemize}
	\noindent\textbf{Case 2:} $k=2$
	\begin{itemize}
		\item $\mathcal{F}^{\lambda^{12}_{p^{2}(8p-9),x^{4}_{k}+l^{4}_{k}}}(q) = \frac{p-1}{2}(q+1)\left[\left(\frac{p^{3}-3p^{2}-p+3}{24}\right)q^{4}+\left(\frac{11p^{3}+3p^{2}+p-15}{24}\right)q^{3}+\left(\frac{11p^{3}+3p^{2}+p+9}{24}\right)q^{2}\right]$
		
		\hspace*{3.5cm} $+\frac{p-1}{2}(q+1)\left[\left(\frac{p^{3}-3p^{2}-p+3}{24}\right)q\right]$
		
		when $l^{4}_{k}=j^{2}_{0}.$ 
		\item $\mathcal{F}^{\lambda^{12}_{p^{2}(8p-9),x^{4}_{k}+l^{4}_{k}}}(q) = \frac{p-1}{2}(q+1)\left[\left(\frac{p^{3}-3p^{2}-p+3-12t(t-p+2)}{24}\right)q^{4}+\left(\frac{11p^{3}+3p^{2}+p-39+24t(t+3)}{24}\right)q^{3}\right]$
		
		\hspace*{3.5cm} $+\frac{p-1}{2}(q+1)\left[\left(\frac{11p^{3}+3p^{2}+p+33-12t(t+p+2)}{24}\right)q^{2}+\left(\frac{p^{3}-3p^{2}-p+3}{24}\right)q\right]$
		
		when $l^{4}_{k}\in [(t-1)(p+1)+1,\dots,tp-1],~1\leq t\leq \frac{p-1}{2}.$
		\item $\mathcal{F}^{\lambda^{12}_{p^{2}(8p-9),x^{4}_{k}+l^{4}_{k}}}(q) = \frac{p-1}{2}(q+1)\left[\left(\frac{p^{3}-3p^{2}-p+3-12t(t-p+2)}{24}\right)q^{4}+\left(\frac{11p^{3}+3p^{2}+p-15+24t(t+3)}{24}\right)q^{3}\right]$
		
		\hspace*{3.5cm} $+\frac{p-1}{2}(q+1)\left[\left(\frac{11p^{3}+3p^{2}+p+9-12t(t+p+2)}{24}\right)q^{2}+\left(\frac{p^{3}-3p^{2}-p+3}{24}\right)q\right]$
		
		when $l^{4}_{k}\in [tp,\dots,t(p+1)],~1\leq t\leq \frac{p-3}{2}.$
		\item $\mathcal{F}^{\lambda^{12}_{p^{2}(8p-9),x^{4}_{k}+l^{4}_{k}}}(q) = \frac{p-1}{2}(q+1)\left[\left(\frac{p^{3}-13p+12}{24}\right)q^{4}+\left(\frac{11p^{3}+9p^{2}+p-21}{24}\right)q^{3}+\left(\frac{11p^{3}-6p^{2}+13p-18}{24}\right)q^{2}\right]$
		
		\hspace*{3.5cm} $+\frac{p-1}{2}(q+1)\left[\left(\frac{p^{3}-3p^{2}-p+27}{24}\right)q\right]$
		
		when $l^{4}_{k}\in [(\frac{p-1}{2})p,\dots,(\frac{p-1}{2})(p+1)-1].$
		\item $\mathcal{F}^{\lambda^{12}_{p^{2}(8p-9),x^{4}_{k}+l^{4}_{k}}}(q) = \frac{p-1}{2}(q+1)\left[\left(\frac{p^{3}-3p^{2}-p+3}{24}\right)q^{4}+\left(\frac{11p^{3}-6p^{2}-11p+6}{24}\right)q^{3}+\left(\frac{11p^{3}+9p^{2}+p-21}{24}\right)q^{2}\right]$
		
		\hspace*{3.5cm} $+\frac{p-1}{2}(q+1)\left[\left(\frac{p^{3}+11p-12}{24}\right)q+1\right]$
		
		when $l^{4}_{k}=(\frac{p-1}{2})(p+1).$
		\item $\mathcal{F}^{\lambda^{12}_{p^{2}(8p-9),x^{4}_{k}+l^{4}_{k}}}(q) = \frac{p-1}{2}(q+1)\left[\left(\frac{p^{3}-3p^{2}-p+3}{24}\right)q^{4}+\left(\frac{11p^{3}-6p^{2}-11p+6}{24}\right)q^{3}+\left(\frac{11p^{3}+9p^{2}+25p-45}{24}\right)q^{2}\right]$
		
		\hspace*{3.5cm} $+\frac{p-1}{2}(q+1)\left[\left(\frac{p^{3}-13p+36}{24}\right)q\right]$
		
		when $l^{4}_{k}\in [(\frac{p-1}{2})(p+1)+1,\dots,(\frac{p+1}{2})p-1],~\frac{p+1}{2}\leq t\leq p-1 .$
		\item $\mathcal{F}^{\lambda^{12}_{p^{2}(8p-9),x^{4}_{k}+l^{4}_{k}}}(q) = \frac{p-1}{2}(q+1)\left[\left(\frac{p^{3}-3p^{2}-p+3}{24}\right)q^{4}+\left(\frac{11p^{3}-21p^{2}-47p-39-12t(t-3p-4)}{24}\right)q^{3}\right]$
		
		\hspace*{3.5cm} $+\frac{p-1}{2}(q+1)\left[\left(\frac{11p^{3}+27p^{2}+73p+33+24t(t-2p-3)}{24}\right)q^{2}+\left(\frac{p^{3}-3p^{2}-25p+3-12t(t-p-2)}{24}\right)q\right]$
		
		when $l^{4}_{k}\in [(t-1)(p+1)+1,\dots,tp-1],~\frac{p+3}{2}\leq t\leq p-1 .$
		\item $\mathcal{F}^{\lambda^{12}_{p^{2}(8p-9),x^{4}_{k}+l^{4}_{k}}}(q) = \frac{p-1}{2}(q+1)\left[\left(\frac{p^{3}-3p^{2}-p+3}{24}\right)q^{4}+\left(\frac{11p^{3}-21p^{2}-47p-15-12t(t-3p-4)}{24}\right)q^{3}\right]$
		
		\hspace*{3.5cm} $+\frac{p-1}{2}(q+1)\left[\left(\frac{11p^{3}+27p^{2}+73p+9+24t(t-2p-3)}{24}\right)q^{2}+\left(\frac{p^{3}-3p^{2}-25p+3-12t(t-p-2)}{24}\right)q\right]$
		
		when $l^{4}_{k}\in [tp,\dots,t(p+1)],~\frac{p+1}{2}\leq t\leq p-1.$
	\end{itemize}
	\noindent\textbf{Case 2:} $k>2$
	\begin{itemize}
		\item $\mathcal{F}^{\lambda^{2(k+4)}_{p^{k}(8p-9),x^{4}_{k}+l^{4}_{k}}}(q) = \frac{p-1}{2}(q+1)\left[\left(\frac{p^{3}-3p^{2}-p+3}{24}\right)q^{4}+\left(\frac{11p^{3}+3p^{2}+p+9}{24}\right)q^{3}+\left(\frac{11p^{3}+3p^{2}+p-15}{24}\right)q^{2}\right]$
		
		\hspace*{3.5cm} $+\frac{p-1}{2}(q+1)\left[\left(\frac{p^{3}-3p^{2}-p+3}{24}\right)q\right]$
		
		when $l^{4}_{k}=j^{k}_{0}.$ 
		\item $\mathcal{F}^{\lambda^{2(k+4)}_{p^{k}(8p-9),x^{4}_{k}+l^{4}_{k}}}(q) = \frac{p-1}{2}(q+1)\left[\left(\frac{p^{3}-3p^{2}-p+3-12t(t-p+2)}{24}\right)q^{4}+\left(\frac{11p^{3}+3p^{2}+p-15+24t(t+3)}{24}\right)q^{3}\right]$
		
		\hspace*{3.5cm} $+\frac{p-1}{2}(q+1)\left[\left(\frac{11p^{3}+3p^{2}+p+9-12t(t+p+2)}{24}\right)q^{2}+\left(\frac{p^{3}-3p^{2}-p+3}{24}\right)q\right]$
		
		when $l^{4}_{k}\in [(t-1)\sum\limits_{\iota=0}^{k-1}p^{\iota}+1,\dots,(t-1)\sum\limits_{\iota=1}^{k-1}p^{\iota}+p-1],~1\leq t\leq \frac{p-1}{2}.$
		\item $\mathcal{F}^{\lambda^{2(k+4)}_{p^{k}(8p-9),x^{4}_{k}+l^{4}_{k}}}(q) = \frac{p-1}{2}(q+1)\left[\left(\frac{p^{3}-3p^{2}-p+3-12t(t-p+2)}{24}\right)q^{4}+\left(\frac{11p^{3}+3p^{2}+p+9+24t(t+3)}{24}\right)q^{3}\right]$
		
		\hspace*{3.5cm} $+\frac{p-1}{2}(q+1)\left[\left(\frac{11p^{3}+3p^{2}+p-15-12t(t+p+2)}{24}\right)q^{2}+\left(\frac{p^{3}-3p^{2}-p+3}{24}\right)q\right]$
		
		when $l^{4}_{k}\in [(t-1)\sum\limits_{\iota=1}^{k-1}p^{\iota}+p,\dots,t\sum\limits_{\iota=0}^{k-1}p^{\iota}],~1\leq t\leq \frac{p-3}{2}.$
		\item $\mathcal{F}^{\lambda^{2(k+4)}_{p^{k}(8p-9),x^{4}_{k}+l^{4}_{k}}}(q) = \frac{p-1}{2}(q+1)\left[\left(\frac{p^{3}-13p+12}{24}\right)q^{4}+\left(\frac{11p^{3}+9p^{2}+25p-21}{24}\right)q^{3}+\left(\frac{11p^{3}-6p^{2}-11p+6}{24}\right)q^{2}\right]$
		
		\hspace*{3.5cm} $+\frac{p-1}{2}(q+1)\left[\left(\frac{p^{3}-3p^{2}-p+3}{24}\right)q\right]$
		
		when $l^{4}_{k}\in [(\frac{p-3}{2})\sum\limits_{\iota=1}^{k-1}p^{\iota}+p,\dots,(\frac{p-1}{2})\sum\limits_{\iota=2}^{k-1}p^{\iota}-1].$
		\item $\mathcal{F}^{\lambda^{2(k+4)}_{p^{k}(8p-9),x^{4}_{k}+l^{4}_{k}}}(q) = \frac{p-1}{2}(q+1)\left[\left(\frac{p^{3}-13p+12}{24}\right)q^{4}+\left(\frac{11p^{3}+9p^{2}+25p-45}{24}\right)q^{3}+\left(\frac{11p^{3}-6p^{2}-11p+30}{24}\right)q^{2}\right]$
		
		\hspace*{3.5cm} $+\frac{p-1}{2}(q+1)\left[\left(\frac{p^{3}-3p^{2}-p+3}{24}\right)q\right]$
		
		when $l^{4}_{k}\in [(\frac{p-1}{2})\sum\limits_{\iota=2}^{k-1}p^{\iota}+p,\dots,(\frac{p-1}{2})\sum\limits_{\iota=1}^{k-1}p^{\iota}-1].$
		
		\item $\mathcal{F}^{\lambda^{2(k+4)}_{p^{k}(8p-9),x^{4}_{k}+l^{4}_{k}}}(q) = \frac{p-1}{2}(q+1)\left[\left(\frac{p^{3}-13p+12}{24}\right)q^{4}+\left(\frac{11p^{3}+9p^{2}+p-21}{24}\right)q^{3}+\left(\frac{11p^{3}-6p^{2}+13p-18}{24}\right)q^{2}\right]$
		
		\hspace*{3.5cm} $+\frac{p-1}{2}(q+1)\left[\left(\frac{p^{3}-3p^{2}-p+27}{24}\right)q\right]$
		
		when $l^{4}_{k}\in [(\frac{p-1}{2})\sum\limits_{\iota=1}^{k-1}p^{\iota},\dots,(\frac{p-1}{2})\sum\limits_{\iota=0}^{k-1}p^{\iota}-1].$
		\item $\mathcal{F}^{\lambda^{2(k+4)}_{p^{k}(8p-9),x^{4}_{k}+l^{4}_{k}}}(q) = \frac{p-1}{2}(q+1)\left[\left(\frac{p^{3}-3p^{2}-p+3}{24}\right)q^{4}+\left(\frac{11p^{3}-6p^{2}-11p+6}{24}\right)q^{3}+\left(\frac{11p^{3}+9p^{2}+p-21}{24}\right)q^{2}\right]$
		
		\hspace*{3.5cm} $+\frac{p-1}{2}(q+1)\left[\left(\frac{p^{3}+11p-12}{24}\right)q+1\right]$
		
		when $l^{4}_{k}=(\frac{p-1}{2})\sum\limits_{\iota=0}^{k-1}p^{\iota}.$
		\item $\mathcal{F}^{\lambda^{2(k+4)}_{p^{k}(8p-9),x^{4}_{k}+l^{4}_{k}}}(q) = \frac{p-1}{2}(q+1)\left[\left(\frac{p^{3}-3p^{2}-p+3}{24}\right)q^{4}+\left(\frac{11p^{3}-6p^{2}-11p+6}{24}\right)q^{3}+\left(\frac{11p^{3}+9p^{2}+25p-45}{24}\right)q^{2}\right]$
		
		\hspace*{3.5cm} $+\frac{p-1}{2}(q+1)\left[\left(\frac{p^{3}-13p+36}{24}\right)q\right]$
		
		when $l^{4}_{k}\in [(\frac{p-1}{2})\sum\limits_{\iota=0}^{k-1}p^{\iota}+1,\dots,(\frac{p-1}{2})\sum\limits_{\iota=1}^{k-1}p^{\iota}+p-1].$
		\item $\mathcal{F}^{\lambda^{2(k+4)}_{p^{k}(8p-9),x^{4}_{k}+l^{4}_{k}}}(q) = \frac{p-1}{2}(q+1)\left[\left(\frac{p^{3}-3p^{2}-p+3}{24}\right)q^{4}+\left(\frac{11p^{3}-21p^{2}-47p-39-12t(t-3p-4)}{24}\right)q^{3}\right]$
		
		\hspace*{3.5cm} $+\frac{p-1}{2}(q+1)\left[\left(\frac{11p^{3}+27p^{2}+73p+33+24t(t-2p-3)}{24}\right)q^{2}+\left(\frac{p^{3}-3p^{2}-25p+3-12t(t-p-2)}{24}\right)q\right]$
		
		when $l^{4}_{k}\in[(t-1)\sum\limits_{\iota=0}^{k-1}p^{\iota}+1,\dots,(t-1)\sum\limits_{\iota=1}^{k-1}p^{\iota}+p-1],~\frac{p+3}{2}\leq t\leq p-1 .$
		\item $\mathcal{F}^{\lambda^{12}_{p^{k}(8p-9),x^{4}_{k}+l^{4}_{k}}}(q) = \frac{p-1}{2}(q+1)\left[\left(\frac{p^{3}-3p^{2}-p+3}{24}\right)q^{4}+\left(\frac{11p^{3}-21p^{2}-47p-15-12t(t-3p-4)}{24}\right)q^{3}\right]$
		
		\hspace*{3.5cm} $+\frac{p-1}{2}(q+1)\left[\left(\frac{11p^{3}+27p^{2}+73p+9+24t(t-2p-3)}{24}\right)q^{2}+\left(\frac{p^{3}-3p^{2}-25p+3-12t(t-p-2)}{24}\right)q\right]$
		
		when $l^{4}_{k}\in [(t-1)\sum\limits_{\iota=1}^{k-1}p^{\iota}+p,\dots,t\sum\limits_{\iota=0}^{k-1}p^{\iota}],~\frac{p+1}{2}\leq t\leq p-1.$
	\end{itemize}
	
	\section{$p^{k}$-Fibonacci numbers:}
	In this section, we define the $p^{k}$-Fibonacci number for each vertex in the set $V^{2r}_{k}$  where $k\geq 0.$ We compute the  $p^{k}$-Fibonacci number using mathematical induction for $k=0$ and $1.$ For  $k\geq 2,$ we provide an inductive method for computing these numbers for each vertex. Futhermore, we identify the subclasses of the vertex set $V^{2r}_{k}$ in which the $p^{k}$-Fibonacci number remains the same. We also prove that the derivative of the $p^{k}$-Eulerian polynomial evaluated at 1 is equal to the $p^{k}$-Fibonacci number.
	
	Let $P^{2r}_{i,m_{2},\dots,m_{2s},l^{s}_{k}}$ denote the path $(\lambda^{2k+1}_{p^{k}(p-1),i},m_{2},m_{3},\dots,m_{2s},\lambda^{2r}_{p^{k}[2sp-(2s+1)],x^{s}_{k}+l^{s}_{k}}).$ This is the path starting at the vertex $\lambda^{2k+1}_{p^{k}(p-1),i}$ and ending at the vertex $\lambda^{2r}_{p^{k}[2sp-(2s+1)],x^{s}_{k}+l^{s}_{k}}$ depending on $m_{2},\dots,m_{2s},$ where $0\leq i<p^{k}(p-1),$ $0\leq l^{s}_{k}<p^{k}$ with $x^{s}_{k}=p^{k}[sp-(s+1)],$ and $m_{d}$'s are same as in Remark \ref{syt def} for $2\leq d \leq 2s.$ 
	
	Let $P^{2r}_{i,l^{s}_{k}}$ denote the set of all paths starting at the vertex $\lambda^{k+1}_{p^{k}(p-1),i}$ and ending at the vertex $\lambda^{2r}_{p^{k}[2sp-(2s+1)],x^{s}_{k}+l^{s}_{k}}$ which are obtained by adding the blocks $B^{d}_{m_{d},n_{d}}$ successcively. Let us denote $P^{2r}_{l^{s}_{k}}=\cup_{i}P^{2r}_{i,l^{s}_{k}},~0\leq i<p^{k}(p-1).$
	\begin{Def}
		We now define the $p^{k}$-Fibonacci number for the descents of paths $P^{2r}_{l^{s}_{k}} $ as $$\mathcal{M}_{\lambda}=\sum_{P^{2r}_{i,m_{2},\dots,m_{2s},l^{s}_{k}}\in P^{2r}_{l^{s}_{k}}} des(P^{2r}_{i,m_{2},\dots,m_{2s},l^{s}_{k}})$$ where $\lambda=\lambda^{2r}_{p^{k}[2sp-(2s+1)],x^{s}_{k}+l^{s}_{k}}.$ It is equivalent to say $\mathcal{M}_{\lambda}$ is the $p^{k}$-Fibonacci number at the vertex $\lambda.$
	\end{Def}
	
	\begin{thm}\label{p^{0}}
		The $p^{0}$-Fibonacci number at the vertex $\lambda^{2r}_{(2sp-(2s+1)),sp-(s+1)}$ is given by $\frac{p-1}{2}(2(s-1)p^{s-1}-(2s-3)p^{s-2}),~s\geq 2.$
	\end{thm}
	\begin{proof}
		Let us prove this theorem by mathematical induction.  
		
		When $s=1,$ by using Lemma \ref{2p-3} we have
		$$\mathcal{M}_{\lambda^{2}_{2p-3,p-2}} = \frac{p-1}{2}.$$
		When $s=2,$ there are $p$-paths from $\lambda^{2}_{2p-3,p-2}$ to $\lambda^{4}_{4p-5,2p-3}$, so we add the $p^{0}$-Fibonacci number of $\lambda^{2}_{2p-3,p-2}$ $p$ times, together with the number of descents at $2$ and $3$ of all paths $P^{4}_{i,m_{2},m_{3},p^{k}t^{'},0}~(m_{4}=p^{k}t^{'},~l^{s}_{k}=0).$  
		
		By using the Lemma \ref{2p-3}, there is no descent at $2$ of any such path belongs to the set $P^{4}_{0}$  and by Remark \ref{k=0}, there is a descent at $3$ of the path $P^{4}_{i,m_{2},m_{3},p^{k}t^{'},0},$ when $0\leq t^{'}\leq\frac{p-3}{2}$ and for all $i,m_{2}$ added along the paths. For each such descent, there are $p-1$ paths connected to the vertex  $\lambda^{2}_{2p-3,p-2}.$ The sum of the number of descents at 3 over all paths in $P^{4}_{0}$ is $(p-1)\frac{p-1}{2}.$ Hence the $p^{0}$-Fibonacci number at the vertex $\lambda^{4}_{4p-5,2p-3}$ is
		$$\mathcal{M}_{\lambda^{4}_{4p-5,2p-3}} =p\left(\frac{p-1}{2}\right)+(p-1)\frac{p-1}{2}+=\frac{p-1}{2}(2p-1).$$
		
		Assume that, the result is true upto $n=s,$ (i.e) $$\mathcal{M}_{\lambda^{2r}_{2sp-(2s+1),sp-(s+1)}} =\frac{p-1}{2}(2(s-1)p^{s-1}-(2s-3)p^{s-2}).$$
		Now we need to prove that the result is true for $n=s+1.$
		There are $p$-paths from $\lambda^{2r}_{2sp-(2s+1),sp-(s+1)}$ to $\lambda^{2(r+1)}_{2(s+1)p-(2s+3),(s+1)p-(s+2)}$, so we add the $p^{0}$-Fibonacci numbers of $\mathcal{M}_{\lambda^{2r}_{2sp-(2s+1),sp-(s+1)}}$ $p$ times, together with the number of descents at $2s$ and $2s+1$ of all paths $P^{4}_{i,m_{2},\dots,m_{2s+1},p^{k}t^{'},0}~(m_{2s+2}=p^{k}t^{'},~l^{s}_{k}=0).$ 
		
		By using the by Remark \ref{k=0}(\ref{aa}), there is a descent at $2s+1$ of the path $P^{2(r+1)}_{i,m_{2},\dots,m_{2s+1},p^{k}t^{'},0},$ when $0\leq t^{'}\leq\frac{p-3}{2}$ and for all $i,m_{2},\dots,m_{2s}$ and $m_{2s+1}=p^{k}(p-1-t^{'})$ added along the paths. For each such descent, there are $p^{s-1}(p-1)$ paths connected to the vertex  $\lambda^{2r}_{2sp-(2s+1),sp-(s+1)}.$ The sum of the number of decents at $2s+1$ over all paths in $P^{2r}_{0}$ is $p^{s-1}(p-1)\frac{p-1}{2}.$
		
		Again by Remark \ref{k=0}(\ref{bb}), for each $t^{'}$ there are $t^{'}$ descents at $2s$ of the paths $P^{2(r+1)}_{i,m_{2},\dots,m_{2s+1},p^{k}t^{'},0},$ for all $i,m_{2},\dots,m_{2s}$ and $m_{2s+1}=p^{k}(p-1-t^{'})$ added along the paths. For each such descent, there are $p^{s-2}(p-1)$ paths connected to the vertex  $\lambda^{2r-1}_{(2s-1)p-2s,x^{s}_{0}-(p-1),l^{s-1}_{1}},$ where $x^{s}_{0}=sp-(s+1)$ and $0\leq l^{s-1}_{1}<p.$ The sum of the number of decents at $2s$ of over all paths in $P^{2r}_{0}$ is $p^{s-2}(p-1)\frac{p(p-1)}{2}.$

		Hence the $p^{0}$-Fibonacci number at the vertex $\lambda^{2(r+1)}_{2(s+1)p-(2s+3),(s+1)p-(s+2)}$ is
		
		\noindent $\mathcal{M}_{\lambda^{2(r+1)}_{2(s+1)p-(2s+3),(s+1)p-(s+2)}}$
		\begin{align*}
			&=p\left(\frac{p-1}{2}(2(s-1)p^{s-1}-(2s-3)p^{s-2})\right)+2p^{s-1}(p-1)\left(\frac{p-1}{2}\right)\\
			&=\left(\frac{p-1}{2}\right)(2s\cdot p^{s}-(2s-1)p^{s-1})
		\end{align*}
		Therefore, the result holds for all $s\geq 1.$	
	\end{proof}
	Now, we compute the $p^{k}$-Fibonacci numbers for $s=1,2$ and $k\geq 1:$
	
	By using Lemma $\ref{2p-3},$ we have
	\begin{align*}
		\mathcal{M}_{\lambda^{2(k+1)}_{p^{k}(2p-3),x^{1}_{k}+l^{1}_{k}}} =\frac{p-1}{2},~\forall~0\leq l^{1}_{k}<p^{k} 
	\end{align*}
	where $x^{1}_{k}=p^{k}(p-2).$
	
	For each $t^{'},$ there is a edge between the vertices  $\lambda^{2(k+2)}_{p^{k}(4p-5),x^{2}_{k}+l^{2}_{k}}$ and  $\lambda^{2(k+1)}_{p^{k}(2p-3),x^{1}_{k}+\alpha^{1}_{k}+p^{k-1}t^{'}}$ where $0\leq t^{'}\leq p-1,$ and $l^{2}_{k}=\alpha^{1}_{k}p+\beta^{1}_{k}$ with $0\leq \alpha^{1}_{k}<p^{k-1}$ and $0\leq \beta^{1}_{k}\leq p-1.$ Therefore, the $p^{k}$-Fibonacci number at the vertex $\lambda^{2(k+2)}_{p^{k}(4p-5),x^{2}_{k}+l^{2}_{k}}$ is the sum of the $p^{k}$-Fibonacci numbers  $\mathcal{M}_{\lambda^{2(k+1)}_{p^{k}(2p-3),x^{1}_{k}+\alpha^{1}_{k}+p^{k-1}t^{'}}}$ along with the number of descent at $2$ and $3$ of all paths ending at the vertex $\lambda^{2(k+2)}_{p^{k}(4p-5),x^{2}_{k}+l^{2}_{k}}.$
	
	By Lemma $\ref{2p-3},$ there is no descent at $2.$ By Theorem \ref{rules l}, 
	\begin{itemize}
		\item If $l^{2}_{k}<\frac{p^{k}-1}{2},$ then there is a descent at $3$ of the path $P^{2(k+2)}_{i,m_{2},m_{3},p^{k}t^{'},l^{2}_{k}},$ when $0\leq t^{'}\leq \frac{p-1}{2}$ and for all $i, m_{2}$ and $m_{3}=p^{k}(p-1)-(\alpha^{1}_{k}+p^{k-1}t^{'}+\beta^{1}_{k})$ added along the path.
		\item If $l^{2}_{k}\geq\frac{p^{k}-1}{2},$ then there is a descent at $3$ of the path $P^{2(k+2)}_{i,m_{2},m_{3},p^{k}t^{'},l^{2}_{k}},$ when $0\leq t^{'}\leq \frac{p-3}{2}$ and for all $i, m_{2}$ and $m_{3}=p^{k}(p-1)-(\alpha^{2}_{k}+p^{k-1}t^{'}+\beta^{2}_{k})$ added along the path.
	\end{itemize}
	For each such descent, there are $p-1$ paths connected to the vertex $\lambda^{2(k+1)}_{p^{k}(2p-3),x^{1}_{k}+\alpha^{1}_{k}+p^{k-1}t^{'}}.$ 
		
	Therefore, the $p^{k}$-Fibonacci number at the vertex $\lambda^{2(k+2)}_{p^{k}(4p-5),x^{2}_{k}+l^{2}_{k}}$ is 
	\begin{align}\label{4p-5 i}
		\mathcal{M}_{\lambda^{2(k+2)}_{p^{k}(4p-5),x^{2}_{k}+l^{2}_{k}}} & = \sum_{t^{'}=1}^{p-1}\mathcal{M}_{\lambda^{2(k+1)}_{p^{k}(2p-3),x^{1}_{k}+\alpha^{1}_{k}+p^{k-1}t^{'}}} + \frac{p+1}{2}(p-1) \notag\\ &=\left(\frac{p-1}{2}\right)\left(\frac{p+1}{2}\right)+\left(\frac{p-1}{2}\right)\left(\frac{p-1}{2}\right)+\left(\frac{p+1}{2}\right)(p-1)\notag \\ &
		=\frac{p-1}{2}\left(2p+1\right)
	\end{align}
	when $l^{2}_{k}<\frac{p^{k}-1}{2}.$
	\begin{align}\label{4p-5 ii}
		\mathcal{M}_{\lambda^{2(k+2)}_{p^{k}(4p-5),x^{2}_{k}+l^{2}_{k}}} & = \sum_{t^{'}=1}^{p-1}\mathcal{M}_{\lambda^{2(k+1)}_{p^{k}(2p-3),x^{1}_{k}+\alpha^{1}_{k}+p^{k-1}t^{'}}} + \frac{p-1}{2}(p-1) \notag\\ &=\left(\frac{p-1}{2}\right)\left(\frac{p+1}{2}\right)+\left(\frac{p-1}{2}\right)\left(\frac{p-1}{2}\right)+\left(\frac{p-1}{2}\right)(p-1)\notag \\ & =\frac{p-1}{2}\left(2p-1\right)
	\end{align}
	when $l^{2}_{k}\geq\frac{p^{k}-1}{2}.$
	\subsection{Method of computing the $p^{k}$-Fibonacci numbers}\label{Fibonacci number}
	Assume that we have computed the $p^{k}$-Fibonacci numbers for all vertices up to $2r$-th floor. For $k\geq 1$ and $s\geq 3,$ let  $\mathcal{M}_{\lambda^{2r}_{p^{k}(2sp-(2s+1)),x^{s}_{k}+l^{s}_{k}}}$ denote the $p^{k}$-Fibonacci number corresponding to each vertex $\lambda^{2r}_{p^{k}(2sp-(2s+1)),x^{s}_{k}+l^{s}_{k}}\in V^{2r}_{k}$ on the $2r$-th floor. 
	
	There are $p$ edges from the vertex  $\lambda^{2(r+1)}_{p^{k}(2(s+1)p-(2s+3)),x^{s+1}_{k}+l^{s+1}_{k}}$ to the vertices on the $2r$-th floor, as shown in Figure \ref{5}.
	
	\begin{figure}[h!]
		\begin{center}
			\includegraphics[width=11cm,height = 4cm]{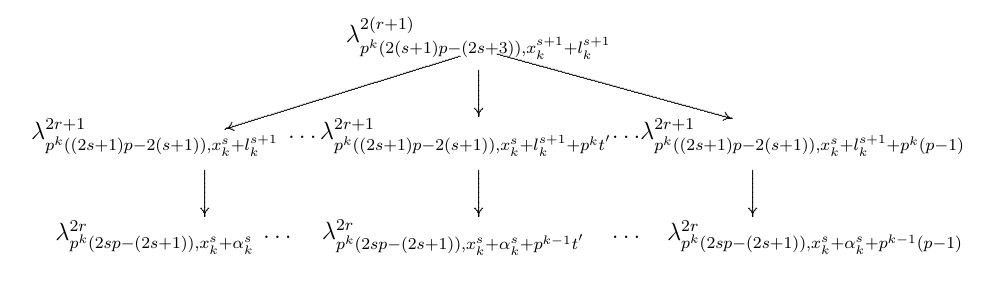}
		\end{center}
		\caption{}
		\label{5}
	\end{figure}
	
	The $p^{k}$-Fibonacci number corresponding to the vertex  $\lambda^{2(r+1)}_{p^{k}(2(s+1)p-2(s+1)),x^{s+1}_{k}+l^{s+1}_{k}},$ where $x^{s+1}_{k}=p^{k}[(s+1)p-(s+2)],~ 0\leq l^{s+1}_{k} \leq p^{k},$ and $l^{s+1}_{k}=\alpha^{s}_{k} p+\beta^{s}_{k}$ with $0\leq \alpha^{s}_{k}<p^{k-1}$ and $0\leq \beta^{s}_{k}\leq p-1,$ is computed using the formula 
	\begin{align}\label{sum11}
		\mathcal{M}_{\lambda^{2(r+1)}_{p^{k}(2(s+1)p-(2s+3)),x^{s+1}_{k}+l^{s+1}_{k}}} &= \sum_{t^{'}=0}^{p-1} \mathcal{M}_{\lambda^{2r}_{p^{k}(2sp-(2s+1)),x^{s}_{k}+\alpha^{s}_{k}+p^{k-1}t^{'}}}+p^{s-1}(p-1)\left(\frac{p+1}{2}\right)\notag\\ &+p^{s-2}(p-1)\left(\frac{p(p-1)}{2}+t\right)
	\end{align}
	where $0\leq l^{s+1}_{k}<\frac{p^{k}-1}{2}-\frac{p-1}{2}.$
	\begin{align}\label{sum2}
		\mathcal{M}_{\lambda^{2(r+1)}_{p^{k}(2(s+1)p-(2s+3)),x^{s+1}_{k}+l^{s+1}_{k}}} &= \sum_{t^{'}=0}^{p-1} \mathcal{M}_{\lambda^{2r}_{p^{k}(2sp-(2s+1)),x^{s}_{k}+\frac{p^{k}-1}{2}+p^{k-1}t^{'}}}+p^{s-1}(p-1)\left(\frac{p+1}{2}\right)\notag \\ &+p^{s-2}(p-1)\left(\frac{p(p-1)}{2}+\frac{p-1}{2}\right)
	\end{align}
	where $\frac{p^{k}-1}{2}-\frac{p-1}{2}\leq l^{s+1}_{k}<\frac{p^{k}-1}{2}.$
	\begin{align}\label{sum3}
		\mathcal{M}_{\lambda^{2(r+1)}_{p^{k}(2(s+1)p-(2s+3)),x^{s+1}_{k}+l^{s+1}_{k}}} &= \sum_{t^{'}=0}^{p-1} \mathcal{M}_{\lambda^{2r}_{p^{k}(2sp-(2s+1)),x^{s}_{k}+\frac{p^{k}-1}{2}+p^{k-1}t^{'}}}+p^{s-1}(p-1)\left(\frac{p-1}{2}\right)\notag\\ &+p^{s-2}(p-1)\left(\frac{p(p-1)}{2}+\frac{p-1}{2}\right)
	\end{align}
	where $ l^{s+1}_{k}=\frac{p^{k}-1}{2}.$
	\begin{align}\label{sum4}
		\mathcal{M}_{\lambda^{2(r+1)}_{p^{k}(2(s+1)p-(2s+3)),x^{s+1}_{k}+l^{s+1}_{k}}} &= \sum_{t^{'}=0}^{p-1} \mathcal{M}_{\lambda^{2r}_{p^{k}(2sp-(2s+1)),x^{s}_{k}+\frac{p^{k}-1}{2}+p^{k-1}t^{'}}}+p^{s-1}(p-1)\left(\frac{p-1}{2}\right)\notag\\ &+p^{s-2}(p-1)\left(\frac{p(p-1)}{2}+\frac{p+1}{2}\right)
	\end{align}
	where $ \frac{p^{k}-1}{2}< l^{s+1}_{k}\leq\frac{p^{k}-1}{2}+\frac{p-1}{2}.$
	\begin{align}\label{sum5}
		\mathcal{M}_{\lambda^{2(r+1)}_{p^{k}(2(s+1)p-(2s+3)),x^{s+1}_{k}+l^{s+1}_{k}}} &= \sum_{t^{'}=0}^{p-1} \mathcal{M}_{\lambda^{2r}_{p^{k}(2sp-(2s+1)),x^{s}_{k}+\alpha^{s}_{k}+p^{k-1}t^{'}}}+p^{s-1}(p-1)\left(\frac{p-1}{2}\right) \notag\\ &+p^{s-2}(p-1)\left(\frac{p(p-1)}{2}+t\right)
	\end{align}
	where $ \frac{p^{k}-1}{2}+\frac{p-1}{2}<l^{s+1}_{k}<p^{k}.$
	
	We now derive the formula step by step
	
	\noindent\textbf{Case 1:} When $0\leq l^{s+1}_{k}<\frac{p^{k}-1}{2}-\frac{p-1}{2}.$ 
	
	By Theorem \ref{rules l}, there is a descent at $2s+1$ of the path $P^{2(r+1)}_{i,m_{2},\dots,m_{2s+1},p^{k}t^{'},l^{s+1}_{k}}$ when $0\leq t^{'}\leq \frac{p-1}{2}$ and for all $m_{2},m_{3},\dots,m_{2s},$ $m_{2s+1}=p^{k}(p-1)-(\alpha^{s}_{k}+p^{k-1}t^{'}+\beta^{s}_{k}).$ 
	For each such descent in the path $P^{2(r+1)}_{i,m_{2},\dots,m_{2s-1},p^{k}t^{'},l^{s+1}_{k}}$ which ends at the vertex
	
	\noindent $\lambda^{2(r+1)}_{p^{k}(2(s+1)p-(2s+3)),x^{s+1}_{k}+l^{s+1}_{k}}\in V^{2(r+1)}_{k},$ there are $p^{s-1}(p-1)$ paths connected to the vertex $\lambda^{2r}_{p^{k}(2sp-(2s+1)),x^{s}_{k}+\alpha^{s}_{k}+p^{k-1}t^{'}}\in V^{2r}_{k},$ where $0\leq t^{'}\leq p-1.$ By adding the number of descents at $2s+1$ of all paths belongs to the set $P^{2(r+1)}_{l^{s+1}_{k}}$ is $p^{s-1}(p-1)\left(\frac{p+1}{2}\right).$  
	
	\noindent\textbf{Case 2:} When $\frac{p^{k}-1}{2}-\frac{p-1}{2}\leq l^{s+1}_{k}<\frac{p^{k}-1}{2}.$ 
	
	By Theorem \ref{rules l}, the total number of descent at $2s+1$ of all paths belongs to the set $P^{2(r+1)}_{l^{s+1}_{k}}$ is the same as in Case 1. 
	
	\noindent\textbf{Case 3:} When $\frac{p^{k}-1}{2}\leq l^{s+1}_{k}\leq\frac{p^{k}-1}{2}+\frac{p-1}{2}.$ 
	
	By Theorem \ref{rules l}, there is a descent at $2s+1$ of the path $P^{2(r+1)}_{i,m_{2},\dots,m_{2s+1},p^{k}t^{'},l^{s+1}_{k}}$ when $0\leq t^{'}\leq \frac{p-3}{2}$ and for all $m_{2},m_{3},\dots,m_{2s},$ $m_{2s+1}=p^{k}(p-1)-(\alpha^{s}_{k}+p^{k-1}t^{'}+\beta^{s}_{k}).$ 
	For each such descent in the path $P^{2(r+1)}_{i,m_{2},\dots,m_{2s-1},p^{k}t^{'},l^{s+1}_{k}}$ which ends at the vertex 
	
	\noindent$\lambda^{2(r+1)}_{p^{k}(2(s+1)p-(2s+3)),x^{s+1}_{k}+l^{s+1}_{k}}\in V^{2(r+1)}_{k},$ there are $p^{s-1}(p-1)$ paths connected to the vertex $\lambda^{2r}_{p^{k}(2sp-(2s+1)),x^{s}_{k}+\frac{p^{k}-1}{2}+p^{k-1}t^{'}}\in V^{2r}_{k},$ where $0\leq t^{'}\leq p-1.$ By adding the number of descents at $2s+1$ of all paths belongs to the set $P^{2(r+1)}_{l^{s+1}_{k}}$ is $p^{s-1}(p-1)\left(\frac{p-1}{2}\right).$  
	
	\noindent\textbf{Case 4:} When $\frac{p^{k}-1}{2}-\frac{p-1}{2}< l^{s+1}_{k}<p^{k}.$ 
	
	By Theorem \ref{rules l}, the total number of descent at $2s+1$	of all paths belongs to the set $P^{2(r+1)}_{l^{s+1}_{k}}$ is the same as in Case 3.
	
	Secondly, we compute the descent at $2s$ in the following case:
	
	\noindent\textbf{Case a:} When $j^{k}_{t-1}<l^{s+1}_{k}< j^{k}_{t}, ~1\leq t \leq p-1.$ 
	
	By Theorem \ref{rules l}, and using equation (\ref{t^{'}<t}), (\ref{t^{'}>t}), there are $t^{'}+1$ descents at $2s$ of the path $P^{2(r+1)}_{i,m_{2},\dots,m_{2s-1},p^{k}t^{'},l^{s+1}_{k}}$ when $t^{'}\leq t-1,$ and $t^{'}$ descents when $t^{'}\geq t$ for all $m_{2},m_{3},\dots,m_{2s},$ $m_{2s+1}=p^{k}(p-1)-(\alpha^{s}_{k}+p^{k-1}t^{'}+\beta^{s}_{k}).$ 
	
	For each such descent in the path$P^{2(r+1)}_{i,m_{2},\dots,m_{2s-1},p^{k}t^{'},l^{s+1}_{k}},$ which ends at
	the vertex 
	
	\noindent$\lambda^{2(r+1)}_{p^{k}(2(s+1)p-(2s+3)),x^{s+1}_{k}+l^{s+1}_{k}}\in V^{2(r+1)}_{k},$ there are $p^{s-2}(p-1)$ paths connected to the vertex $\lambda^{2r-1}_{p^{k}((2s-1)p-2s),x^{s-1}_{k}+l^{s-1}_{k}}\in V^{2r-1}_{k},$ and contained within the path $P^{2(r+1)}_{i,m_{2},\dots,m_{2s-1},p^{k}t^{'},l^{s+1}_{k}}.$ By adding the number of descents at $2s$ of all the paths belongs to the set $P^{2(r+1)}_{l^{s+1}_{k}}$ is $p^{s-2}(p-1)\left(\frac{p(p-1)}{2}+t\right).$  
	
	\noindent\textbf{Case b:} When $l^{s+1}_{k}= j^{k}_{t},~0\leq t \leq p-1.$ 
	
	By Theorem \ref{rules l}, and using equation (\ref{t^{'}<t}), (\ref{t^{'}>t}) and (\ref{t^{'}=t}) the number of descents at $2s$ of the path $P^{2(r+1)}_{i,m_{2},\dots,m_{2s-1},p^{k}t^{'},l^{s+1}_{k}}$ is as follows:
	\begin{itemize}
		\item $t^{'}+1$ descents when $t^{'}<t,$
		\item exactly $t$ descents when $t^{'}=t,$
		\item $t^{'}$ descents when $t^{'}>t.$
	\end{itemize}
	
	For each such descent in the path $P^{2(r+1)}_{i,m_{2},\dots,m_{2s-1},p^{k}t^{'},l^{s+1}_{k}},$ which ends at the vertex $\lambda^{2(r+1)}_{p^{k}(2(s+1)p-(2s+3)),x^{s+1}_{k}+l^{s+1}_{k}}\in V^{2(r+1)}_{k},$ there are $p^{s-2}(p-1)$ paths connected to the vertex $\lambda^{2r-1}_{p^{k}((2s-1)p-2s),x^{s-1}_{k}+l^{s-1}_{k}}\in V^{2r}_{k},$ and contained within the path $P^{2(r+1)}_{i,m_{2},\dots,m_{2s-1},p^{k}t^{'},l^{s+1}_{k}}.$ By adding the number of descents at $2s$ of all the paths belongs to the set $P^{2(r+1)}_{l^{s+1}_{k}}$ is $p^{s-2}(p-1)\left(\frac{p(p-1)}{2}+t\right).$  
	
	We obtain the above formula for each vertex $\lambda^{2(r+1)}_{p^{k}[2(2s+1)p-(2s+3)],x^{s+1}_{k}+l^{s+1}_{k}},$ by summing the $p^{k}$-Fibonacci number corresponding to the vertex $\lambda^{2r}_{p^{k}(2sp-(2s+1)),x^{s}_{k}+\alpha^{s}_{k}+p^{k-1}t^{'}},$ where $0\leq t^{'}\leq p-1,$ together with the number of descents at $2s$ and $2s+1$ as previously calculated in the above cases. 
	\begin{thm}\label{p^{1}}
		For $s\geq 3,$ the $p$-Fibonacci number at the vertex $\lambda^{2r}_{p(2sp-(2s+1)),x^{s}_{1}+l^{s}_{1}}$ is given by:
		\begin{align*}
			\mathcal{M}_{\lambda^{2r}_{p(2sp-(2s+1)),x^{s}_{1}+l^{s}_{1}}}&=\frac{p^{s-3}(p-1)}{2}(2(s-1)p^{2}+2t-(2s-5)),~0\leq t<\frac{p-1}{2} ~\text{and}~l^{s}_{1}=j^{1}_{t}
			\\
			\mathcal{M}_{\lambda^{2r}_{p(2sp-(2s+1)),x^{s}_{1}+l^{s}_{1}}}&=\frac{p^{s-3}(p-1)}{2}(2(s-1)p^{2}+2t-2p-(2s-5)),~\frac{p-1}{2}\leq t\leq p-1 ~\text{and}~l^{s}_{1}=j^{1}_{t}
		\end{align*}
		where $x^{s}_{1}=p(sp-(s+1))$ and $0\leq l^{s}_{1}<p.$
	\end{thm}
	\begin{proof}		
		When $s=3,$ by using the method of computing $p$-Fibonacci number defined in Subsection \ref{Fibonacci number}, we get
		\begin{align*}
			\mathcal{M}_{\lambda^{8}_{p(6p-7),x^{3}_{1}+l^{3}_{1}}} &=\left(\frac{p-1}{2}\right)(4p^{2}+2t-1)~\text{when}~0\leq t<\frac{p-1}{2} ~\text{and}~l^{3}_{1}=j^{1}_{t}
		\end{align*}
		\begin{align*}
			\mathcal{M}_{\lambda^{8}_{p(6p-7),x^{3}_{1}+l^{3}_{1}}} &=\left(\frac{p-1}{2}\right)(4p^{2}+2t-2p-1)~\text{when}~ \frac{p-1}{2}\leq t\leq p-1 ~\text{and}~l^{3}_{1}=j^{1}_{t}
		\end{align*}
		Assume that, the result is true upto $n=s.$ Now we prove that the result holds for $n=s+1.$ By using the method of computing $p$-Fibonacci number defined in Subsection \ref{Fibonacci number}, we get 
		\begin{align*}
			\mathcal{M}_{\lambda^{2(r+1)}_{p(2(s+1)p-(2s+3)),x^{s+1}_{1}+l^{s+1}_{1}}} &= \sum_{t^{'}=0}^{p-1} \mathcal{M}_{\lambda^{2r}_{p(2sp-(2s+1)),x^{s}_{1}+\alpha^{s}_{1}+t^{'}}}+p^{s-1}(p-1)\left(\frac{p+1}{2}\right)\\ &+p^{s-2}(p-1)\left(\frac{p(p-1)}{2}+t\right) \\ &= \left(\frac{p^{s-3}(p-1)}{2}\right)\left(\sum\limits_{t^{'}=1}^{\frac{p-3}{2}}2(s-1)p^{2}+2t^{'}-(2s-5)+p^{2}(p-1)\left(\frac{p+1}{2}\right)\right. \\ & \left.+\sum\limits_{t^{'}=\frac{p-1}{2}}^{p-1}2(s-1)p^{2}+2t^{'}-2p-(2s-5)+p(p-1)\left(\frac{p(p-1)}{2}+t\right)\right) \\ 
			&= \frac{p^{s-2}(p-1)}{2}(2sp^{2}+2t-(2s-3))
		\end{align*}
		when $0\leq t <\frac{p-1}{2}$ and $l^{s}_{1}=j^{1}_{t}$
		
		Similarly, for $\frac{p-1}{2}\leq t\leq p-1$ and $l^{s}_{1}=j^{1}_{t},$ we get 
		\begin{align*}
			\mathcal{M}_{\lambda^{2(r+1)}_{p(2(s+1)p-(2s+3)),x^{s+1}_{1}+l^{s+1}_{1}}} &= \sum_{t^{'}=0}^{p-1} \mathcal{M}_{\lambda^{2r}_{p(2sp-(2s+1)),x^{s}_{1}+\alpha^{s}_{1}+t^{'}}}+p^{s-1}(p-1)\left(\frac{p-1}{2}\right)\\ &+p^{s-2}(p-1)\left(\frac{p(p-1)}{2}+t\right) \\ 
			&= \frac{p^{s-2}(p-1)}{2}(2sp^{2}+2t-2p-(2s-3))
		\end{align*}
		Hence by using the principle of mathematical induction the result holds for all $s\geq 3.$
	\end{proof}
	\begin{lem}\label{split}
		For each $k\geq 2$ and $3\leq s< k+2,$ the $p^{k}$-Fibonacci number are the same on the subclasses of the vertex set $$\{\lambda^{2r}_{p^{k}[2sp-(2s+1)],x^{s}_{k}+l^{s}_{k}}/0\leq l^{s}_{k} <p^{k},~x^{s}_{k}=p^{k}[sp-(s+1)]\}$$ on the $2r$-th floor. This set is partitioned into $(s-2)p+2$ disjoint subclasses, where each subclass is determined by the value of $l^{s}_{k}.$ Specifically, the subclasses are as follows:
		\begin{itemize}
			\item[(a)] The singleton subclass containing only the element where $l^{s}_{k}=0.$
			\item[(b)] For $0\leq i \leq s-4,$ and $1\leq t\leq p-1,$ a subclass consists of those $l^{s}_{k}$ in the interval 
			$$\left[(t-1)\displaystyle\sum_{\iota=i}^{k-1}p^{\iota}+p^{i},\dots,(t-1)\displaystyle\sum_{\iota=i+1}^{k-1}p^{\iota}+p^{i+1}-1\right].$$			\item[(c)]For $1\leq t\leq p-1,~t\neq \frac{p-1}{2},$ a subclass consists of those $l^{s}_{k}$ in the interval 
			$$\left[(t-1)\displaystyle\sum_{\iota=s-3}^{k-1}p^{\iota}+p^{s-3},\dots,t\displaystyle\sum_{\iota=0}^{k-1}p^{\iota}\right].$$
			\item[(d)] The subclass where $l^{s}_{k}$ ranges over $\left[\frac{p-3}{2}\sum\limits_{\iota=s-3}^{k-1}p^{\iota}+p^{s-3},\dots,\frac{p-1}{2}\sum\limits_{\iota=s-2}^{k-1}p^{\iota}\right]$
			\item[(e)] For $1\leq i^{'}\leq s-2,$ a subclass consists of those $l^{s}_{k}$ in the interval 
			$$\left[\frac{p-1}{2}\sum\limits_{\iota=i^{'}}^{k-1}p^{\iota},\dots,\frac{p-1}{2} \sum\limits_{\iota=i^{'}-1}^{k-1}p^{\iota}-1\right]$$
			\item[(f)] The singleton subclass where $l^{s}_{k}=\left[\frac{p-1}{2}\sum\limits_{\iota=0}^{k-1}p^{\iota}\right]$
		\end{itemize}
	\end{lem}
	\begin{proof}
		The vertex set is divided into a subclasses depending on $l^{s}_{k},$ so it is enough to consider the interval $[0,p^{k}-1]$ where $l^{s}_{k}\in [0,p^{k}-1] .$ The $p^{k}$-Fibonacci number defined on the interval $[0,p^{k}-1]$ will vary on the subintervals on each floor. These subintervals arise due to the behaviour of $p^{k}$-Fibonacci number across different floors.
		
		\noindent\textbf{Base case:}
		From equation (\ref{4p-5 i}) and (\ref{4p-5 ii}), the interval $[0,p^{k}-1]$ is initially divided into two subintervals on the $2(k+2)$ floor:
		$$\left[0,\frac{p^{k}-1}{2}-1\right]~\text{and}~\left[\frac{p^{k}-1}{2},p^{k}-1\right]$$
		
		These subintervals are further subdivided on floor $2(k+3)$ as follows:
		\begin{itemize}
			\item $[0]$
			\item $\left[(t-1)\sum\limits_{\iota=0}^{k-1}p^{\iota}+1,\dots,t\sum\limits_{\iota=0}^{k-1}p^{\iota}\right],~1\leq t\leq p-1 ,~t\neq \frac{p-1}{2}$
			\item $\left[\frac{p-3}{2}\sum\limits_{\iota=0}^{k-1}p^{\iota},\dots,\frac{p-1}{2}\sum\limits_{\iota=1}^{k-1}p^{\iota}-1\right]$
			\item $\left[\frac{p-1}{2}\sum\limits_{\iota=1}^{k-1}p^{\iota},\dots,\frac{p-1}{2}\sum\limits_{\iota=0}^{k-1}p^{\iota}-1\right]$
			\item $\left[\frac{p-1}{2}\sum\limits_{\iota=0}^{k-1}p^{\iota}\right]$
		\end{itemize}
		This yields $p+2$ sub intervals in total. We get these subintervals by computing the $p^{k}$-Fiboacci number at the vertex $\lambda^{2(k+3)}_{p^{k}(6p-7),x^{3}_{k}+l^{3}_{k}}.$ The $p^{k}$-Fibonacci number will be same on the each sub interval. 
		
		Now, we prove that the result is true for $s=4,$ on the $2(k+4)$-th floor. Let $j^{k}_{t-1}<l^{3}_{k}<j^{k}_{t},$ where $1\leq t\leq p-1.$ There is an edge between the vertices
		\begin{align}\label{beta^{'}}
			\lambda^{2(k+3)}_{p^{k}(6p-7),x^{3}_{k}+l^{3}_{k}}\hookrightarrow \lambda^{2k+7}_{p^{k}(7p-8),x^{4}_{k}-p^{k}(p-1)+pl^{3}_{k}+\beta^{3}_{k}} \hookrightarrow\lambda^{2(k+4)}_{p^{k}(8p-9),x^{4}_{k}+pl^{3}_{k}+\beta^{3}_{k}-p^{k}(t-1)}	
		\end{align}
		Here $l^{3}_{k}\neq j^{k}_{t}.$ For each $\beta^{3}_{k},$ with $0\leq \beta^{3}_{k}\leq p-1.$ For any $\beta^{3}_{k},$ the transformed index $pl^{3}_{k}-p^{k}(t-1)$ remains within the same interval (i.e) $j^{k}_{t-1}<pl^{3}_{k}+\beta^{3}_{k}-p^{k}(t-1)<j^{k}_{t}.$ 
		
		When $l^{3}_{k}= j^{k}_{t},$ where $0\leq t\leq p-2,~t\neq \frac{p-1}{2}$ 
		\begin{align}\label{j_{t}}
			\lambda^{2(k+3)}_{p^{k}(6p-7),x^{3}_{k}+j^{k}_{t}}\hookrightarrow \lambda^{2k+7}_{p^{k}(7p-8),x^{4}_{k}-p^{k}(p-1)+pj^{k}_{t}+\beta^{3}_{k}} \hookrightarrow\lambda^{2(k+4)}_{p^{k}(8p-9),x^{4}_{k}+pj^{k}_{t}+\beta^{3}_{k}-p^{k}t}
		\end{align}
		The way the interval is subdivided depends on whether $\beta^{3}_{k}\leq t$ or $\beta^{3}_{k}> t.$
		Note that $j^{k}_{t-1}<pj^{k}_{t}+\beta^{3}_{k}-p^{k}t\leq j^{k}_{t}$ when $\beta^{3}_{k}\leq t$ and $j^{k}_{t}<pj^{k}_{t}+\beta^{3}_{k}-p^{k}t\leq j^{k}_{t+1}$ when $\beta^{3}_{k}> t.$ Because of this, the interval divides into two subintervals.
	
			When $l^{3}_{k}= j^{k}_{\frac{p-1}{2}},$ it is same as above, but the interval divides into 3 subintervals because the number of descent at $7$ varies when $\beta^{3}_{k}< t$ and $\beta^{3}_{k}= t.$ 
			
		By the above argument, we get the following sub intervals
		\begin{itemize}
			\item $[0]$
			\item $\left[(t-1)\sum\limits_{\iota=0}^{k-1}p^{\iota}+1,\dots,(t-1)\sum\limits_{\iota=1}^{k-1}p^{\iota}+p-1\right],~1\leq t\leq p-1 ,~t\neq \frac{p-1}{2}$
			\item $\left[(t-1)\sum\limits_{\iota=1}^{k-1}p^{\iota}+p,\dots,t\sum\limits_{\iota=0}^{k-1}p^{\iota}\right],~1\leq t\leq p-1 ,~t\neq \frac{p-1}{2}$
			\item $\left[\frac{p-3}{2}\sum\limits_{\iota=0}^{k-1}p^{\iota},\dots,\frac{p-3}{2}\sum\limits_{\iota=1}^{k-1}p^{\iota}+p-1\right]$
			\item $\left[\frac{p-3}{2}\sum\limits_{\iota=1}^{k-1}p^{\iota}+p,\dots,\frac{p-1}{2}\sum\limits_{\iota=2}^{k-1}p^{\iota}-1\right]$
			\item $\left[\frac{p-1}{2}\sum\limits_{\iota=2}^{k-1}p^{\iota},\dots,\frac{p-1}{2}\sum\limits_{\iota=1}^{k-1}p^{\iota}-1\right]$
			\item $\left[\frac{p-1}{2}\sum\limits_{\iota=1}^{k-1}p^{\iota},\dots,\frac{p-1}{2}\sum\limits_{\iota=0}^{k-1}p^{\iota}-1\right]$
			\item $\left[\frac{p-1}{2}\sum\limits_{\iota=0}^{k-1}p^{\iota}\right]$
		\end{itemize} 
		The total number of sub intervals are $$p+2+(p-2)+2=2p+2$$ 
				
		These sub intervals will be further divided in to sub intervals again on each floor, which we explain by induction. Assume that the result holds for $n=s<k+1$ with $(s-2)p+2$ sub intervals.
		
		We now prove that there are $(s-1)p+2$ sub intervals on the floor $2(r+1).$ Consider the path: 
		\begin{align}\label{beta^{'}i}	 			\lambda^{2r}_{p^{k}(2sp-(2s+1)),x^{s}_{k}+l^{s}_{k}}\hookrightarrow \lambda^{2r+1}_{p^{k}((2s+1)p-2(s+1)),x^{s+1}_{k}-p^{k}(p-1)+pl^{s}_{k}+\beta^{s}_{k}} \hookrightarrow\lambda^{2(r+1)}_{p^{k}(2(s+1)p-(2s+3)),x^{s+1}_{k}+pl^{s}_{k}+\beta^{s}_{k}-p^{k}(t-1)}	
		\end{align}
		Here $l^{s}_{k}\neq j^{k}_{t}.$ For each $\beta^{s}_{k},$ with $0\leq \beta^{s}_{k}\leq p-1.$ For any $\beta^{s}_{k},$ the transformed index $pl^{s}_{k}+\beta^{s}_{k}-p^{k}(t-1)$ remains within the same interval 
		(i.e) $j^{k}_{t-1}<pl^{s}_{k}+\beta^{s}_{k}-p^{k}(t-1)<j^{k}_{t}.$ 
		
		When $l^{3}_{k}= j^{k}_{t}$ 
		\begin{align}\label{j_{t}i}
			\lambda^{2r}_{p^{k}(2sp-(2s+1)),x^{s}_{k}+j^{k}_{t}}\hookrightarrow \lambda^{2r+1}_{p^{k}((2s+1)p-2(s+1)),x^{s+1}_{k}-p^{k}(p-1)+pj^{k}_{t}+\beta^{s}_{k}} \hookrightarrow\lambda^{2(r+1)}_{p^{k}(2(s+1)p-(2s+3)),x^{s+1}_{k}+pj^{k}_{t}+\beta^{s}_{k}-p^{k}t}
		\end{align}
		As in earlier steps, based on the values of $\beta^{s}_{k}$ the interval is divided into two subintervals when $t\neq \frac{p-1}{2}$ and the interval is divided into three subintervals when $t= \frac{p-1}{2}.$

		Hence we get the following intervals on the $2(r+1)$-th floor
		\begin{itemize}
			\item[(a)] $[0]$
			\item[(b)] $\left[(t-1)\displaystyle\sum_{\iota=i}^{k-1}p^{\iota}+p^{i},\dots,(t-1)\displaystyle\sum_{\iota=i+1}^{k-1}p^{\iota}+p^{i+1}-1\right],$ where $0\leq i \leq s-3,$ and $1\leq t\leq p-1$
			\item[(c)] $\left[(t-1)\displaystyle\sum_{\iota=s-2}^{k-1}p^{\iota}+p^{s-2},\dots,t\displaystyle\sum_{\iota=0}^{k-1}p^{\iota}\right],$ where $1\leq t\leq p-1,~t\neq \frac{p-1}{2}.$
			\item[(d)] $\left[\frac{p-3}{2}\sum\limits_{\iota=s-2}^{k-1}p^{\iota}+p^{s-2},\dots,\frac{p-1}{2}\sum\limits_{\iota=s-1}^{k-1}p^{\iota}\right]$
			\item[(e)] $\left[\frac{p-1}{2}\sum\limits_{\iota=i^{'}}^{k-1}p^{\iota},\dots,\frac{p-1}{2} \sum\limits_{\iota=i^{'}-1}^{k-1}p^{\iota}-1\right],$ where $1\leq i^{'}\leq s-1,$
			\item[(f)] $\left[\frac{p-1}{2}\sum\limits_{\iota=0}^{k-1}p^{\iota}\right]$
		\end{itemize}	
		Therefore, the number of intervals on the $2(r+1)$-th floor is $$(s-2)p+2+(p-2)+2=(s-1)p+2.$$	
		By principle of mathematical induction, the result holds for all $s<k+2.$
	\end{proof}
	\begin{rem}
		\begin{itemize}
			\item 	When $s=k+2:$ we proceed using same argument as in the Lemma \ref{split}. However, at this stage, the number of subintervals remains unchanged when $l^{s}_{k}=j^{k}_{\frac{p-1}{2}}.$ When $t\neq \frac{p-1}{2},$ there are $p-2$ intervals which gets divided into two subintervals.
			
			By Lemma \ref{split}, when $s=k+1,$ the number of subintervals is $(k-1)p+2.$ Therefore, the number of intervals when $s=k+2,$ is $$(k-1)p+2+(p-2)=kp.$$ 
			\item For $s>k+2,$ the intervals remains same as the case $s=k+2.$ There is no further division in the intervals, since by Lemma \ref{split} we exhaust all the intervals.
		\end{itemize}
		
	\end{rem}
	\begin{thm} \label{k^{'}<k}
		For $k\geq2$ and $3\leq s\leq k+2,$ the $p^{k}$-Fibonacci number at the vertex $\lambda^{2r}_{p^{k}(2sp-(2s+1)),x^{s}_{k}+l^{s}_{k}},$ where $x^{s}_{k}=p^{k}(sp+(s+1)),$ $0\leq l^{s}_{k} <p^{k}$ is given by 
		\begin{description}
			\item[(a)]  If $l^{s}_{k}=0,$ then $$\mathcal{M}_{\lambda^{2r}_{p^{k}(2sp-(2s+1)),x^{s}_{k}+l^{s}_{k}}} = \frac{p-1}{2}(2(s-1)p^{s-1}+1)$$\label{(a)}
			\item[(b)] If $l^{s}_{k}\in\left[(t-1)\displaystyle\sum_{\iota=i}^{k-1}p^{\iota}+p^{i},\dots,(t-1)\displaystyle\sum_{\iota=i+1}^{k-1}p^{\iota}+p^{i+1}-1\right],~0\leq i \leq s-4,~s\neq3$ then for $1\leq t\leq \frac{p-1}{2}$
			\begin{align*}
				\mathcal{M}_{\lambda^{2r}_{p^{k}(2sp-(2s+1)),x^{s}_{k}+l^{s}_{k}}} =&\frac{p-1}{2}\left(2(s-1)p^{s-1}+1+2t\sum_{j=0}^{s-3}p^{j}
				-2\sum_{\mu=0}^{s-i-4}p^{\mu}\right)
			\end{align*} 
			and for $\frac{p+1}{2}\leq t\leq p-1$
			\begin{align*}
				\mathcal{M}_{\lambda^{2r}_{p^{k}(2sp-(2s+1)),x^{s}_{k}+l^{s}_{k}}}=&\frac{p-1}{2}\left(2(s-1)p^{s-1}-1+(2t-2p)\sum_{j=0}^{s-3}p^{j}
				-2\sum_{\mu=0}^{s-i-4}p^{\mu}\right)
			\end{align*} \label{(b)}
			\item[(c)] If $l^{s}_{k}\in\left[(t-1)\displaystyle\sum_{\iota=s-3}^{k-1}p^{\iota}+p^{s-3},\dots,t\displaystyle\sum_{\iota=0}^{k-1}p^{\iota}\right],$ then for $1\leq t\leq \frac{p-3}{2}$
			\begin{align*}
				\mathcal{M}_{\lambda^{2r}_{p^{k}(2sp-(2s+1)),x^{s}_{k}+l^{s}_{k}}} = 	&\frac{p-1}{2}\left(2(s-1)p^{s-1}+1+2t\sum_{j=0}^{s-3}p^{j}\right)
			\end{align*}
			and for $\frac{p+1}{2}\leq t\leq p-1$
			\begin{align*}
				\mathcal{M}_{\lambda^{2r}_{p^{k}(2sp-(2s+1)),x^{s}_{k}+l^{s}_{k}}} = 	 &\frac{p-1}{2}\left(2(s-1)p^{s-1}-1+(2t-2p)\sum_{j=0}^{s-3}p^{j}
				\right)
			\end{align*}\label{(c)}
			\item [(d)] If $l^{s}_{k}\in\left[\frac{p-3}{2}\sum\limits_{\iota=s-3}^{k-1}p^{\iota}+p^{s-3},\dots,\frac{p-1}{2}\sum\limits_{\iota=s-2}^{k-1}p^{\iota}\right],$ then 
			$$\mathcal{M}_{\lambda^{2r}_{p^{k}(2sp-(2s+1)),x^{s}_{k}+l^{s}_{k}}} = \frac{p-1}{2}(2(s-1)p^{s-1}+1+p^{s-2}-1)$$ \label{(d)}
			\item [(e)] If $l^{s}_{k}\in\left[\frac{p-1}{2}\sum\limits_{\iota=i^{'}}^{k-1}p^{\iota},\dots,\frac{p-1}{2} \sum\limits_{\iota=i^{'}-1}^{k-1}p^{\iota}-1\right],~1\leq i^{'}\leq s-2,$ then 
			$$\mathcal{M}_{\lambda^{2r}_{p^{k}(2sp-(2s+1)),x^{s}_{k}+l^{s}_{k}}} = \frac{p-1}{2}\left(2(s-1)p^{s-1}-1-p^{s-2}-2\sum_{\theta = 1}^{s-3}p^{\theta}+2\sum_{\mu =s-i^{'}-1}^{s-2}p^{\mu}-1\right)$$\label{(f)}
			\item [(f)] If $l^{k}_{s}=\frac{p-1}{2} \sum\limits_{\iota=0}^{k-1}p^{\iota},$ then 
			$$\mathcal{M}_{\lambda^{2r}_{p^{k}(2sp-(2s+1)),x^{s}_{k}+l^{s}_{k}}} = \frac{p-1}{2}\left(2(s-1)p^{s-1}-1-p^{s-2}-2\sum_{\theta = 1}^{s-3}p^{\theta}-1\right)$$\label{(g)}
		\end{description}
	\end{thm}
	\begin{proof}
		Let us prove this theorem by mathematical induction. Assume that the result hold till $s,$ where $s<k+1.$ We have to prove that the result holds for $s+1.$ 
		By Lemma \ref{split}, we have now identified all the intervals on the floor $2(r+1).$
		
		Now we compute $p^{k}$-Fibonacci for all the intervals, by using the method of computing $p^{k}$-Fibonacci numbers defined in Subsection \ref{Fibonacci number}.
		\begin{description}
			\item[(a)] When $l^{s+1}_{k}=0,$ by equation (\ref{sum11}) and substituting the values of $\mathcal{M}_{\lambda^{2r}_{p^{k}(2sp-(2s+1)),x^{s}_{k}+p^{k-1}t^{'}}},$ where $0\leq t^{'}\leq p-1,$ we obtain
			\begin{equation*}
				\mathcal{M}_{\lambda^{2(r+1)}_{p^{k}(2(s+1)p-(2s+3)),x^{s+1}_{k}+l^{s+1}_{k}}} = \frac{p-1}{2}\left(2s\cdot p^{s}+1\right)      
			\end{equation*}
			\item[(b)] When $l^{s+1}_{k}\in\left[(t-1)\displaystyle\sum_{\iota=i}^{k-1}p^{\iota}+p^{i},\dots,(t-1)\displaystyle\sum_{\iota=i+1}^{k-1}p^{\iota}+p^{i+1}-1\right]$ where $0\leq i \leq s-3$ and $1\leq t\leq p-1.$	
			
			From equation (\ref{sum11}), by substituting the values of $\mathcal{M}_{\lambda^{2r}_{p^{k}(2sp-(2s+1)),x^{s}_{k}+\alpha^{s}_{k}+p^{k-1}t^{'}}},$ where $0\leq t^{'}\leq p-1,$ we obtain 
			\begin{align*}			\mathcal{M}_{\lambda^{2(r+1)}_{p^{k}(2(s+1)p-(2(s+3))),x^{s+1}_{k}+l^{s+1}_{k}}} =&\frac{p-1}{2}\left(2s\cdot p^{s}+1+2t\sum_{j=0}^{s-2}p^{j}
				-2\sum_{\mu=0}^{s-i-3}p^{\mu}\right)
			\end{align*} 
			when $1\leq t\leq \frac{p-1}{2}.$
			
			From equation (\ref{sum11}), (\ref{sum4}) and (\ref{sum5}), substituting the values of $\mathcal{M}_{\lambda^{2r}_{p^{k}(2sp-(2s+1)),x^{s}_{k}+\alpha^{s}_{k}+p^{k-1}t^{'}}},$ where $0\leq t^{'}\leq p-1,$ we get
			\begin{align*}
				\mathcal{M}_{\lambda^{2(r+1)}_{p^{k}(2(s+1)p-(2(s+3))),x^{s+1}_{k}+l^{s+1}_{k}}}=&\frac{p-1}{2}\left(2s\cdot p^{s}-1+(2t-2p)\sum_{j=0}^{s-2}p^{j}
				-2\sum_{\mu=0}^{s-i-3}p^{\mu}\right)
			\end{align*} 
			when $\frac{p+1}{2}\leq t\leq p-1.$
			
			\item[(c)] When $l^{s+1}_{k}\in\left[(t-1)\displaystyle\sum_{\iota=s-2}^{k-1}p^{\iota}+p^{s-2},\dots,t\displaystyle\sum_{\iota=0}^{k-1}p^{\iota}\right]$ where $1\leq t\leq p-1,~t\neq \frac{p-1}{2},$ we have, from equation (\ref{sum11}), by substituting the values of $\mathcal{M}_{\lambda^{2r}_{p^{k}(2sp-(2s+1)),x^{s}_{k}+\alpha^{s}_{k}+p^{k-1}t^{'}}},$ where $0\leq t^{'}\leq p-1,$ we obtain
			\begin{align*}
				\mathcal{M}_{\lambda^{2(r+1)}_{p^{k}(2(s+1)p-(2s+3)),x^{s+1}_{k}+l^{s+1}_{k}}}&	=\frac{p-1}{2}\left(2s\cdot p^{s}+1+2t\sum_{j=0}^{s-2}p^{j}
				\right)  
			\end{align*}
			when $1\leq t\leq \frac{p-3}{2}. $
			
			From equation (\ref{sum5}), by substituting the values of $\mathcal{M}_{\lambda^{2r}_{p^{k}(2sp-(2s+1)),x^{s}_{k}+\alpha^{s}_{k}+p^{k-1}t^{'}}},$ where $0\leq t^{'}\leq p-1,$ we obtain
			\begin{align*}
				\mathcal{M}_{\lambda^{2(r+1)}_{p^{k}(2(s+1)p-(2s+3)),x^{s+1}_{k}+l^{s+1}_{k}}} &=\frac{p-1}{2}\left(2s\cdot p^{s}+1+(2t-2p)\sum_{j=0}^{s-2}p^{j}
				\right)  
			\end{align*}
			when  $\frac{p+1}{2}\leq t\leq p-1. $
			
			\item [(d)] When			
			$l^{s+1}_{k}\in\left[\frac{p-3}{2}\sum\limits_{\iota=s-2}^{k-1}p^{\iota}+p^{s-2},\dots,\frac{p-1}{2}\sum\limits_{\iota=s-1}^{k-1}p^{\iota}-1\right].$ From equation (\ref{sum11}), by substituting the values of $\mathcal{M}_{\lambda^{2r}_{p^{k}(2sp-(2s+1)),x^{s}_{k}+\alpha^{s}_{k}+p^{k-1}t^{'}}},$ where $0\leq t^{'}\leq p-1,$ we obtain
			\begin{align*}
				\mathcal{M}_{\lambda^{2(r+1)}_{p^{k}(2(s+1)p-(2s+3)),x^{s+1}_{k}+l^{s+1}_{k}}}&=\frac{p-1}{2}\left(2s\cdot p^{s}+1+p^{s-1}-1\right)  
			\end{align*}
			\item[(e)] When			
			$l^{s+1}_{k}\in\left[\frac{p-1}{2}\sum\limits_{\iota=i^{'}}^{k-1}p^{\iota},\dots,\frac{p-1}{2}\sum\limits_{\iota=i^{'}-1}^{k-1}p^{\iota}-1\right],$ where $1\leq i^{'} \leq s-1$. From equation (\ref{sum11}) and (\ref{sum2}), by substituting the values of $\mathcal{M}_{\lambda^{2r}_{p^{k}(2sp-(2s+1)),x^{s}_{k}+\alpha^{s}_{k}+p^{k-1}t^{'}}},$ where $0\leq t^{'}\leq p-1,$ we obtain
			\begin{align*}
				\mathcal{M}_{\lambda^{2(r+1)}_{p^{k}(2(s+1)p-(2s+3)),x^{s+1}_{k}+l^{s+1}_{k}}}	&=\frac{p-1}{2}\left(2s\cdot p^{s}-1-p^{s-1}-2\sum_{\theta = 1}^{s-2}p^{\theta}+2\sum_{\mu =s-i^{'}}^{s-1}p^{\mu}-1\right)  
			\end{align*}
			\item [(f)] When $l^{s}_{k}=\frac{p-1}{2}\displaystyle\sum_{\iota=0}^{k-1}p^{\iota}.$ From equation (\ref{sum3}), substituting the values of
			
			 $\mathcal{M}_{\lambda^{2r}_{p^{k}(2sp-(2s+1)),x^{s}_{k}+\alpha^{s}_{k}+p^{k-1}t^{'}}},$ where $0\leq t^{'}\leq p-1,$ we obtain
			\begin{align*}
				\mathcal{M}_{\lambda^{2(r+1)}_{p^{k}(2(s+1)p-(2s+3)),x^{s+1}_{k}+l^{s+1}_{k}}}&=\frac{p-1}{2}\left(2s\cdot p^{s}-1-p^{s-1}-2\sum_{\theta = 1}^{s-2}p^{\theta}-1\right)  
			\end{align*}
			Hence, by mathematical induction, the result holds for all $s<k+2.$ 
		\end{description}
	\end{proof}
	\begin{thm}\label{k^{'}=k}
		For $k\geq2$ and $s=k+2,$ the $p^{k}$-Fibonacci number at the vertex $\lambda^{2r}_{p^{k}(2sp-(2s+1)),x^{s}_{k}+l^{s}_{k}},$ where $x^{s}_{k}=p^{k}(sp+(s+1)),$ and $0\leq l^{s}_{k} <p^{k}$ is given by 
		\begin{description}
			\item[(a)] If $l^{s}_{k}=0,$ then $$\mathcal{M}_{\lambda^{2r}_{p^{k}(2sp-(2s+1)),x^{s}_{k}+l^{s}_{k}}}=\frac{p-1}{2}(2(k+1)p^{k+1}-1)$$
			\item[(b)] If $l^{s}_{k}\in\left[(t-1)\sum\limits_{\iota=i}^{k-1}p^{\iota}+p^{i},\dots,(t-1)\sum\limits_{\iota=i+1}^{k-1}p^{\iota}+p^{i+1}-1\right],~0\leq i \leq k-2$ then for $ 1\leq t\leq \frac{p-1}{2}$				\begin{align*}
				\mathcal{M}_{\lambda^{2r}_{p^{k}(2sp-(2s+1)),x^{s}_{k}+l^{s}_{k}}}=&\frac{p-1}{2}\left(2(k+1)p^{k+1}-1+2t\sum_{j=0}^{k-1}p^{j}
				-2\sum_{\mu=0}^{k-i-2}p^{\mu}\right)
			\end{align*}
			and for $\frac{p+1}{2}\leq t\leq p-1$
			\begin{align*}
				\mathcal{M}_{\lambda^{2r}_{p^{k}(2sp-(2s+1)),x^{s}_{k}+l^{s}_{k}}}= &\frac{p-1}{2}\left(2(k+1)p^{k+1}-1+(2t-2p)\sum_{j=0}^{k-1}p^{j}
				-2\sum_{\mu=0}^{k-i-2}p^{\mu}\right)
			\end{align*}
			\item[(c)] If $l^{s}_{k}\in\left[tp^{k-1},\dots,t\sum\limits_{\iota=0}^{k-1}p^{\iota}\right],$ then for $1\leq t\leq \frac{p-3}{2}$
			\begin{align*}
				\mathcal{M}_{\lambda^{2r}_{p^{k}(2sp-(2s+1)),x^{s}_{k}+l^{s}_{k}}} =		&\frac{p-1}{2}\left(2(k+1)p^{k+1}-1+2t\sum_{j=0}^{k-1}p^{j}\right)
			\end{align*}
			and for $\frac{p+1}{2}\leq t\leq p-1$
			\begin{align*}
				\mathcal{M}_{\lambda^{2r}_{p^{k}(2sp-(2s+1)),x^{s}_{k}+l^{s}_{k}}} =		&\frac{p-1}{2}\left(2(k+1)p^{k+1}-1+(2t-2p)\sum_{j=0}^{k-1}p^{j}
				\right)
			\end{align*} 
			\item[(d)] If $l^{s}_{k}\in\left[\frac{p-1}{2}\displaystyle\sum_{\iota=i^{'}}^{k-1}p^{\iota},\dots,\frac{p-1}{2} \displaystyle\sum_{\iota=i^{'}-1}^{k-1}p^{\iota}-1\right],~1\leq i^{'}\leq k-1$ then 
			$$\mathcal{M}_{\lambda^{2r}_{p^{k}(2sp-(2s+1)),x^{s}_{k}+l^{s}_{k}}} =\frac{p-1}{2}\left(2(k+1)p^{k+1}-1-p^{k}-2\sum_{\theta = 1}^{k-1}p^{\theta}+2\sum_{\mu =k-i^{'}+1}^{k}p^{\mu}-1\right)$$
			\item[(e)] If $l^{s}_{k}=\frac{p-1}{2} \displaystyle\sum_{\iota=0}^{k-1}p^{\iota},$ then 
			$$\mathcal{M}_{\lambda^{2r}_{p^{k}(2sp-(2s+1)),x^{s}_{k}+l^{s}_{k}}} = \frac{p-1}{2}\left(2(k+1)p^{k+1}-1-p^{k}-2\sum_{\theta = 1}^{k-1}p^{\theta}-1\right)$$
		\end{description}
		
	\end{thm}
	\begin{proof}
		The proof follows by using the recursive method for computing the $p^{k}$-Fibonacci as number defined in Subsection \ref{Fibonacci number}.
	\end{proof}
	\begin{thm}\label{k^{'}>k}
		For $k\geq2$ and $s>k+2,$ the $p^{k}$-Fibonacci number at the vertex $\lambda^{2r}_{p^{k}(2sp-(2s+1)),x^{s}_{k}+l^{s}_{k}},$ where $x^{s}_{k}=p^{k}(sp+(s+1)),$ and $0\leq l^{s}_{k} <p^{k}$ is given by 
		\begin{description}
			\item[(a)] If $l^{s}_{k}=0,$ then $$\mathcal{M}_{\lambda^{2r}_{p^{k}(2sp-(2s+1)),x^{s}_{k}+l^{s}_{k}}}=\frac{p^{s-2-k}(p-1)}{2}(2(s-1)p^{k+1}-2(s-k)+3)$$
			\item[(b)] If $l^{s}_{k}\in\left[(t-1)\sum\limits_{\iota=i}^{k-1}p^{\iota}+p^{i},\dots,(t-1)\sum\limits_{\iota=i+1}^{k-1}p^{\iota}+p^{i+1}-1\right],~0\leq i \leq k-2$ then for $1\leq t\leq \frac{p-1}{2}$
			\begin{align*}
				\mathcal{M}_{\lambda^{2r}_{p^{k}(2sp-(2s+1)),x^{s}_{k}+l^{s}_{k}}}=	&\frac{p^{s-2-k}(p-1)}{2}\left(2(s-1)p^{k+1}-2(s-k)+3+2t\sum_{j=0}^{k-1}p^{j}
				-2\sum_{\mu=0}^{k-i-2}p^{\mu}\right)
			\end{align*}
			and for $\frac{p+1}{2}\leq t\leq p-1$
			
			$\mathcal{M}_{\lambda^{2r}_{p^{k}(2sp-(2s+1)),x^{s}_{k}+l^{s}_{k}}} =$
			\begin{align*}					 	&\frac{p^{s-2-k}(p-1)}{2}\left(2(s-1)p^{k+1}-2(s-k)+3+(2t-2p)\sum_{j=0}^{k-1}p^{j}
				-2\sum_{\mu=0}^{k-i-2}p^{\mu}\right)
			\end{align*}
			\item[(c)] If $l^{s}_{k}\in\left[tp^{k-1},\dots,t\sum\limits_{\iota=0}^{k-1}p^{\iota}\right],$ then for $1\leq t\leq \frac{p-3}{2}$
			\begin{align*}
				\mathcal{M}_{\lambda^{2r}_{p^{k}(2sp-(2s+1)),x^{s}_{k}+l^{s}_{k}}}=&\frac{p^{s-2-k}(p-1)}{2}\left(2(s-1)p^{k+1}-2(s-k)+3+2t\sum_{j=0}^{k-1}p^{j}\right) 
			\end{align*}
			and for $\frac{p+1}{2}\leq t\leq p-1$
			\begin{align*}
				\mathcal{M}_{\lambda^{2r}_{p^{k}(2sp-(2s+1)),x^{s}_{k}+l^{s}_{k}}} = &\frac{p^{s-2-k}(p-1)}{2}\left(2(s-1)p^{k+1}-2(s-k)+3+(2t-2p)\sum_{j=0}^{k-1}p^{j}
				\right)
			\end{align*} 
			\item[(d)] If $l_{k}^{s}\in\left[\frac{p-1}{2}\sum\limits_{\iota=i^{'}}^{k-1}p^{\iota},\dots,\frac{p-1}{2} \sum\limits_{\iota=i^{'}-1}^{k-1}p^{\iota}-1\right],~1\leq i^{'}\leq k-1$ then 
			
			$\mathcal{M}_{\lambda^{2r}_{p^{k}(2sp-(2s+1)),x^{s}_{k}+l^{s}_{k}}}=$ $$\frac{p^{s-2-k}(p-1)}{2}\left(2(s-1)p^{k+1}-2(s-k)+3-p^{k}-2\sum_{\theta = 1}^{k-1}p^{\theta}+2\sum_{\mu =k-i^{'}+1}^{k}p^{\mu}-1\right)$$
			\item[(e)] If $l^{s}_{k}=\frac{p-1}{2} \displaystyle\sum_{\iota=0}^{k-1}p^{\iota},$ then 
			$$\mathcal{M}_{\lambda^{2r}_{p^{k}(2sp-(2s+1)),x^{s}_{k}+l^{s}_{k}}} = \frac{p^{s-2-k}(p-1)}{2}\left(2(s-1)p^{k+1}-2(s-k)+3-p^{k}-2\sum_{\theta = 1}^{k-1}p^{\theta}-1\right)$$
		\end{description}
	\end{thm}
	\begin{proof}
		Let us prove this theorem by mathematical induction. Assume that the result is holds till $s,$ where $s>k+2.$ Now we prove that the result holds for $s+1.$
		\begin{description}
			\item[(a)] When $l^{s+1}_{k}=0,$ we have from equation (\ref{sum11}) and by substituting the values of $\mathcal{M}_{\lambda^{2r}_{p^{k}(2sp-(2s+1)),x^{s}_{k}+p^{k-1}t^{'}}},$ $0\leq t^{'}\leq p-1,$ we obtain:
			\begin{align*}
				\hspace*{-1.5cm}	\mathcal{M}_{\lambda^{2(r+1)}_{p^{k}(2(s+1)p-(2s+3)),x^{s+1}_{k}+l^{s+1}_{k}}}	&=\frac{p^{s-1-k}(p-1)}{2}\left(2s\cdot p^{k+1}-2(s+1-k)+3\right)
			\end{align*}	
			\item[(b)] When $l^{s+1}_{k}\in\left[(t-1)\displaystyle\sum_{\iota=i}^{k-1}p^{\iota}+p^{i},\dots,(t-1)\displaystyle\sum_{\iota=i+1}^{k-1}p^{\iota}+p^{i+1}-1\right],$ where $0\leq i \leq s-4$ and $1\leq t\leq p-1,$ 
			we have from equation (\ref{sum11}), by substituting the values of $\mathcal{M}_{\lambda^{2r}_{p^{k}(2sp-(2s+1)),x^{s}_{k}+\alpha^{s}_{k}+p^{k-1}t^{'}}},$ $0\leq t^{'}\leq p-1,$ we obtain:
			\begin{align*}
				&\hspace*{-1.5cm}\mathcal{M}_{\lambda^{2(r+1)}_{p^{k}(2(s+1)p-(2s+3)),x^{s+1}_{k}+l^{s+1}_{k}}} \\ \quad\quad	&=\frac{p^{s-1-k}(p-1)}{2}\left(2s\cdot p^{k+1}-2(s+1-k)+3+2t\sum_{j=0}^{k-1}p^{j}
				-2\sum_{\mu=0}^{k-i-2}p^{\mu}\right)  
			\end{align*}	
			when $1\leq t\leq \frac{p-1}{2}.$
			
			From equation (\ref{sum4}) and (\ref{sum5}), we get
			\begin{align*}
				&\mathcal{M}_{\lambda^{2(r+1)}_{p^{k}(2(s+1)p-(2s+3)),x^{s+1}_{k}+l^{s+1}_{k}}} \\ &\quad\quad=\frac{p^{s-1-k}(p-1)}{2}\left(2s\cdot p^{k+1}-2(s+1-k)+3+(2t-2p)\sum_{j=0}^{k-1}p^{j}		-2\sum_{\mu=0}^{k-i-2}p^{\mu}\right)  
			\end{align*}
			when $\frac{p+1}{2}\leq t\leq p-1 $
			\item[(c)] When $l^{s+1}_{k}\in\left[tp^{k-1},\dots,t\displaystyle\sum_{\iota=0}^{k-1}p^{\iota}\right]$ where $1\leq t\leq p-1,~t\neq \frac{p-1}{2}.$ From equation (\ref{sum11}), by substituting the values of $\mathcal{M}_{\lambda^{2r}_{p^{k}(2sp-(2s+1)),x^{s}_{k}+\alpha^{s}_{k}+p^{k-1}t^{'}}},$ $0\leq t^{'}\leq p-1,$ we obtain:
			\begin{align*}
				\mathcal{M}_{\lambda^{2(r+1)}_{p^{k}(2(s+1)p-(2s+3)),x^{s+1}_{k}+l^{s+1}_{k}}}
				&=\frac{p^{s-1-k}(p-1)}{2}\left(2s\cdot p^{k+1}-2(s+1-k)+3+2t\sum_{j=0}^{k-1}p^{j}\right)  
			\end{align*}
			when $1\leq t\leq \frac{p-3}{2}. $
			
			From equation (\ref{sum4}) and (\ref{sum5}), we get
			\begin{align*}
				\mathcal{M}_{\lambda^{2(r+1)}_{p^{k}(2(s+1)p-(2s+3)),x^{s+1}_{k}+l^{s+1}_{k}}}
				&=\frac{p^{s-1-k}(p-1)}{2}\left(2s\cdot p^{k+1}-2(s+1-k)-1+(2t-2p)\sum_{j=0}^{k-1}p^{j}
				\right)  
			\end{align*}
			when  $\frac{p+1}{2}\leq t\leq p-1. $
			\item [(d)] When $l^{s+1}_{k}\in\left[\frac{p-1}{2}\sum\limits_{\iota=i^{'}}^{k-1}p^{\iota},\dots,\frac{p-1}{2}\sum\limits_{\iota=i^{'}-1}^{k-1}p^{\iota}-1\right]$ where $1\leq i^{'} \leq k-1$. From equation (\ref{sum11}) and (\ref{sum2}), by substituting the values of $\mathcal{M}_{\lambda^{2r}_{p^{k}(2sp-(2s+1)),x^{s}_{k}+\alpha^{s}_{k}+p^{k-1}t^{'}}},$ $0\leq t^{'}\leq p-1,$ we get	
			\begin{align*}
				&	\mathcal{M}_{\lambda^{2(r+1)}_{p^{k}(2(s+1)p-(2s+3)),x^{s+1}_{k}+l^{s+1}_{k}}}
				\\  &\quad \quad=\frac{p^{s-1-k}(p-1)}{2}\left(2s\cdot p^{k+1}-2(s+1-k)+3-p^{k}-2\sum_{\theta = 1}^{k-1}p^{\theta}+2\sum_{\mu =k-i^{'}+1}^{k}p^{\mu}-1\right)  
			\end{align*}
			\item[(e)] When $l^{s}_{k}=\frac{p-1}{2}\displaystyle\sum_{\iota=0}^{k-1}p^{\iota}.$ From equation (\ref{sum3}), by substituting the values of 
			
			$\mathcal{M}_{\lambda^{2r}_{p^{k}(2sp-(2s+1)),x^{s}_{k}+\alpha^{s}_{k}+p^{k-1}t^{'}}},$ $0\leq t^{'}\leq p-1,$ we get	
			\begin{align*}
				\mathcal{M}_{\lambda^{2(r+1)}_{p^{k}(2(s+1)p-(2s+3)),x^{s+1}_{k}+l^{s+1}_{k}}}
				&=\frac{p^{s-1-k}(p-1)}{2}\left(2s\cdot p^{k+1}-2(s+1-k)+3-p^{k}-2\sum_{\theta = 1}^{k-1}p^{\theta}-1\right)  
			\end{align*}
			
			Hence, by the principle of mathematical induction, the result holds for all $s>k+2.$
		\end{description} 
	\end{proof}
	\begin{ex}
		We now compute the $5^{2}$-Fibonacci number at the vertex $\lambda_{5^{2}(8(5)-9),375+10}$ and the $5^{2}$-Fibonacci number at the vertices $\lambda_{5^{2}(6(5)-7),275+l^{5}_{2}}$ are given as follows
		\begin{itemize}
			\item If $l^{5}_{2}=[0]$ then $\mathcal{M}_{\lambda_{5^{2}((6(5)-7),275}}=202$
			\item If $l^{5}_{2}=[1,2,\dots,6]$ then $\mathcal{M}_{\lambda_{5^{2}((6(5)-7),275+l}}=206$
			\item If $l^{5}_{2}=[7,8,9]$ then $\mathcal{M}_{\lambda_{5^{2}((6(5)-7),275+l}}=210$
			\item If $l^{5}_{2}=[10,11]$ then $\mathcal{M}_{\lambda_{5^{2}((6(5)-7),275+l}}=206$
			\item If $l^{5}_{2}=[12]$ then $\mathcal{M}_{\lambda_{5^{2}((6(5)-7),275+l}}=186$
			\item If $l^{5}_{2}=[13,\dots,19]$ then $\mathcal{M}_{\lambda_{5^{2}((6(5)-7),275+l}}=190$
			\item If $l^{5}_{2}=[20,\dots 25]$ then $\mathcal{M}_{\lambda_{5^{2}((6(5)-7),275+l}}=194$
		\end{itemize}
		We know that $\lambda_{5^{2}(8(5)-9),375+10}\hookleftarrow \lambda_{5^{2}(6(5)-7),275+2+5t^{'}},~0\leq t^{'}\leq 4.$ The corresponding $5^{2}$-Fibonacci number at this vertices are 206, 210, 186, 190 and 194. By Lemma \ref{rules l}, the number of descent at 7 is 3 with the multiplicity $5^{2}\cdot4$. Since $j^{2}_{1}<10<j^{2}_{2},$ the number of descent at 6 is $\frac{5(4)}{2}+2$ with the multiplicity $5\cdot4$, hence 
		$$\mathcal{M}_{\lambda_{5^{2}((8(5)-9),375+10}}= 206+210+186+190+194+300+240 = 1526 $$
	\end{ex}
	\begin{pro}
		Let $k\geq 0$ and $s\geq 1.$ Then $$\left. \frac{d}{dq}\mathcal{F}^{\lambda^{2r}_{p^{k}(2sp-(2s+1)),x^{s}_{k}+l^{s}_{k}}}(q)\right|_{q=1}=\mathcal{M}_{\lambda^{2r}_{p^{k}(2sp-(2s+1)),x^{s}_{k}+l^{s}_{k}}}$$
		where $x^{s}_{k}=p^{k}(sp-(s+1))$ and $0\leq l^{s}_{k}<p^{k}.$
	\end{pro}
	\begin{proof}
		Let $P^{2r}_{l^{s}_{k}}$ be the set of all paths ending at the vertex $\lambda^{2r}_{p^{k}(2sp-(2s+1)),x^{s}_{k}+l^{s}_{k}}.$ Then by Definitions \ref{eulerian polynomial} and \ref{Fibonacci number}, we have 
		\begin{align*}
			\mathcal{F}^{\lambda^{2r}_{p^{k}(2sp-(2s+1)),x^{s}_{k}+l^{s}_{k}}}(q)&= \sum_{P^{2r}_{i,m_{2},\dots,m_{2s},l^{s}_{k}}\in P^{2r}_{l^{s}_{k}}} q^{des(P^{2r}_{i,m_{2},\dots,m_{2s},l^{s}_{k}})} \\ \frac{d}{dq}\mathcal{F}^{\lambda^{2r}_{p^{k}(2sp-(2s+1)),x^{s}_{k}+l^{s}_{k}}}(q)&= \sum_{P^{2r}_{i,m_{2},\dots,m_{2s},l^{s}_{k}}\in P^{2r}_{l^{s}_{k}}} des(P^{2r}_{i,m_{2},\dots,m_{2s},l^{s}_{k}})q^{des(P^{2r}_{i,m_{2},\dots,m_{2s},l^{s}_{k}})-1}\\ \left. \frac{d}{dq}\mathcal{F}^{\lambda^{2r}_{p^{k}(2sp-(2s+1)),x^{s}_{k}+l^{s}_{k}}}(q)\right|_{q=1}&= \sum_{P^{2r}_{i,m_{2},\dots,m_{2s},l^{s}_{k}}\in P^{2r}_{l^{s}_{k}}} des(P^{2r}_{i,m_{2},\dots,m_{2s},l^{s}_{k}}) \\ & = \mathcal{M}_{\lambda^{2r}_{p^{k}(2sp-(2s+1)),x^{s}_{k}+l^{s}_{k}}}(q)		
		\end{align*}
	\end{proof}
	\section{Generating function}
	
	Finally, we end up the study by giving the generating function for the sequence of $p^{k}$-Fibonacci numbers $\{\mathcal{M}_{\lambda^{2r}_{p^{k}(2sp-(2s+1)),x^{s}_{k}+l^{s}_{k}}}/s\geq k+2\}$ 
	for each $k\geq 0$ and $0\leq l^{s}_{k}<p^{k}.$ Also, we provide the reason for naming these numbers as $p^{k}$-Fibonacci numbers.
	
	\begin{thm}\label{gf k=0}
		The generating function for the sequence of numbers $a_{s}=\mathcal{M}_{\lambda^{2r}_{(2sp-(2s+1)),sp-(s+1)}},$ $s\geq 2$ is $$F(x)=\frac{p-1}{2}\left(\frac{2p-2px}{(1-px)^{2}}-\frac{1}{1-px}\right).$$
	\end{thm}
		\begin{thm}\label{gf k=1}
		The generating function for the sequence of numbers $a_{s,t}=\mathcal{M}_{\lambda^{2r}_{p(2sp-(2s+1)),x^{s}_{1}+j^{k}_{t}}}$
		is
		\begin{align*}
			F_{t}(x)&=\frac{p-1}{2}\left(\frac{2p^{2}+2t-1}{1-px}+\frac{2p^{2}-2px}{(1-px)^{2}}\right),~0\leq t<\frac{p-1}{2} \\F_{t}(x)&=\frac{p-1}{2}\left(\frac{2p^{2}+2t-2p-1}{1-px}+\frac{2p^{2}-2px}{(1-px)^{2}}\right),~\frac{p-1}{2}\leq t\leq p-1
		\end{align*}
		where $x^{s}_{1}=p(sp-(s+1)).$
	\end{thm}
		\begin{thm}\label{gf k>2}
		The generating function for the sequence of numbers $a_{s}=\mathcal{M}_{\lambda^{2r}_{p^{k}(2sp-(2s+1)),x^{s}_{k}+l^{s}_{k}}},$ defined for each interval in Theorem \ref{k^{'}>k}, is given below:
		\begin{enumerate}
			\item If $a_{s}=\frac{p^{s-2-k}(p-1)}{2}(2(s-1)p^{k+1}-2(s-k)+3)$ then $$F(x)=\frac{p-1}{2}\left(\frac{2kp^{k+1}-1}{1-px}+\frac{2p^{k+1}-2px}{(1-px)^{2}}\right),~|x|<\frac{1}{p}.$$
			\item 
			\begin{enumerate}
				\item If $a_{s}=b_{s,t,i}=\frac{p^{s-2-k}(p-1)}{2}\left(2(s-1)p^{k+1}-2(s-k)+3+2t\sum\limits_{j=0}^{k-1}p^{j}
				-2\sum\limits_{\mu=0}^{k-i-2}p^{\mu}\right)$ then 
				\begin{align*}
					F_{t,i}(x)&=\frac{p-1}{2}\left(\left(2t\sum_{j=0}^{k-1}p^{j}-2\sum_{\mu=0}^{k-i-2}p^{\mu}+2kp^{k+1}-1\right)\frac{1}{1-px}+\frac{2p^{k+1}-2px}{(1-px)^{2}}\right)
				\end{align*}
				where $|x|<\frac{1}{p},~1\leq t\leq \frac{p-1}{2}$ and $0\leq i \leq k-2.$
				\item 	If $a_{s}=b_{s,t,i}=\frac{p^{s-2-k}(p-1)}{2}\left(2(s-1)p^{k+1}-2(s-k)+3+(2t-2p)\sum\limits_{j=0}^{k-1}p^{j}
				-2\sum\limits_{\mu=0}^{k-i-2}p^{\mu}\right)
				$ 
				
				then 
				\begin{align*}
					F_{t_{i}}(x)&=\frac{p-1}{2}\left(\left((2t-2p)\sum_{j=0}^{k-1}p^{j}-2\sum_{\mu=0}^{k-i-2}p^{\mu}+2kp^{k+1}-1\right)\frac{1}{1-px}+\frac{2p^{k+1}-2px}{(1-px)^{2}}\right)
				\end{align*}
				where $|x|<\frac{1}{p},~\frac{p+1}{2}\leq t\leq p-1$ and $0\leq i \leq k-2.$
			\end{enumerate}
			\item \begin{enumerate}
				\item If $a_{s}=b_{s,t}=\frac{p^{s-2-k}(p-1)}{2}\left(2(s-1)p^{k+1}-2(s-k)+3+2t\sum\limits_{j=0}^{k-1}p^{j}\right)$  then
				\begin{align*}
					F_{t}(x)&=\frac{p-1}{2}\left(\left(2t\sum\limits_{j=0}^{k-1}p^{j}+2kp^{k+1}-1\right)\frac{1}{1-px}+\frac{2p^{k+1}-2px}{(1-px)^{2}}\right),~|x|<\frac{1}{p}  
				\end{align*}
				where $1\leq t\leq \frac{p-3}{2}.$
				\item If $a_{s}=b_{s,t}=\frac{p^{s-2-k}(p-1)}{2}\left(2(s-1)p^{k+1}-2(s-k)+3+(2t-2p)\sum\limits_{j=0}^{k-1}p^{j}
				\right)
				$, then
				\begin{align*}
					F_{t}(x)&=\frac{p-1}{2}\left(\left((2t-2p
					)\displaystyle\sum_{j=0}^{k-1}p^{j}+2kp^{k+1}-1\right)\frac{1}{1-px}+\frac{2p^{k+1}-2px}{(1-px)^{2}}\right),~|x|<\frac{1}{p}			\end{align*} 
				where $\frac{p+1}{2}\leq t\leq p-1.$
			\end{enumerate} 
			\item If $a_{s}=b_{s,i^{'}}=\frac{p^{s-1-k}(p-1)}{2}\left(2(s-1) p^{k+1}-2(s+1-k)+3+p^{k}-2\sum\limits_{\theta = 1}^{k-1}p^{\theta}+2\sum\limits_{\mu =k-i^{'}+1}^{k-1}p^{\mu}-1\right)  
			,$ then $$F_{i^{'}}(x)=\frac{p-1}{2}\left(\left(p^{k}-2\displaystyle\sum_{\theta=1}^{k-1}p^{\theta}+2\sum_{\mu=k-i^{'}+1}^{k-1}p^{\mu}+2kp^{k+1}-2\right)\frac{1}{1-px}+\frac{2p^{k+1}-2px}{(1-px)^{2}}\right)$$
			where $0\leq i^{'}\leq k-1$ and $|x|<\frac{1}{p}.$
			\item If  $a_{s}= \frac{p^{s-2-k}(p-1)}{2}\left(2(s-1)p^{k+1}-2(s-k)+3-p^{k}-2\sum\limits_{\theta = 1}^{k-1}p^{\theta}-1\right)$ then $$F(x)=\frac{p-1}{2}\left(\left(2kp^{k+1}-2-p^{k}-2\sum_{\theta=1}^{k-1}p^{\theta}\right)\frac{1}{1-px}+\frac{2p^{k+1}-2px}{(1-px)^{2}}\right),~|x|<\frac{1}{p}$$
		\end{enumerate}
	\end{thm}
	\begin{rem}
		The proof of Theorems \ref{gf k=0}, \ref{gf k=1} and \ref{gf k>2} follow from straightforward calculation, which can be carried out manually or with the help of mathematical software such as Maple or Sage. We omit the details, as the computation is standard.
	\end{rem}
			\noindent\textbf{We observe the following:}
			\begin{thm}\label{rr}
				For each $k\geq 0$ and $s\geq k+2,$ the sequence of $p^{k}$-Fibonacci numbers satisfies the following recurrence relation: 
				\begin{enumerate}
					\item When $k=0,$ $$\mathcal{M}_{\lambda^{2(r+2)}_{(2(s+2)p-(2s+5)),(s+2)p-(s+3)}}=p\mathcal{M}_{\lambda^{2r}_{(2sp-(2s+1)),sp-(s+1)}}+(p-1)\mathcal{M}_{\lambda^{2(r+1)}_{(2(s+1)p-(2s+3)),(s+1)p-(s+2)}}+b_{s},$$ where $b_{s}=2p^{s-1}(p-1)\left(\frac{p^{2}-1}{2}\right).$
					\item When $k=1,$
					\begin{enumerate}
						\item For $0\leq t <\frac{p-1}{2}$
						$$\mathcal{M}_{\lambda^{2(r+2)}_{p(2(s+2)p-(2s+5)),x^{s+2}_{1}+j^{1}_{t}}}=\sum\limits_{t_{2}=0}^{p-1}\mathcal{M}_{\lambda^{2r}_{p(2sp-(2s+1)),x^{s}_{1}+j^{1}_{t_{2}}}}+\sum\limits_{t_{1}=1}^{p-1}\mathcal{M}_{\lambda^{2(r+1)}_{p(2(s+1)p-(2s+3)),x^{s+1}_{1}+j^{1}_{t_{1}}}}+b_{s},$$ where $b_{s}=p^{s-1}(p-1)\left(p^{2}+p+t\right).$
						\item For $\frac{p-1}{2}\leq t\leq p-1$
						$$\mathcal{M}_{\lambda^{2(r+2)}_{p(2(s+2)p-(2s+5)),x^{s+2}_{1}+j^{1}_{t}}}=\sum_{t_{2}=0}^{p-1}\mathcal{M}_{\lambda^{2r}_{p(2sp-(2s+1)),x^{s}_{1}+j^{1}_{t_{2}}}}+\sum_{t_{1}=1}^{p-1}\mathcal{M}_{\lambda^{2(r+1)}_{p(2(s+1)p-(2s+3)),x^{s+1}_{1}+j^{1}_{t_{1}}}}+b_{s},$$ where $b_{s}=p^{s-1}(p-1)\left(p^{2}+t\right).$
					\end{enumerate}
					\item When $k\geq 2,$
					\begin{enumerate}
						\item For $j^{k}_{t-1}<l^{s+2}_{k} <j^{k}_{t},~0\leq t\leq \frac{p-1}{2}$ and $l^{s+2}_{k} =j^{k}_{t},~0\leq t< \frac{p-1}{2}$
						\begin{align*}
							\mathcal{M}_{\lambda^{2(r+2)}_{p^{k}(2(s+2)p-(2s+5)),x^{s+2}_{k}+l^{s+2}_{k}}}&=\sum\limits_{t_{2}=0}^{p-1}\mathcal{M}_{\lambda^{2r}_{p^{k}(2sp-(2s+1)),x^{s}_{k}+\alpha^{s}_{k}+p^{k-1}t_{2}}}+ \\ &\quad\sum\limits_{t_{1}=1}^{p-1}\mathcal{M}_{\lambda^{2(r+1)}_{p^{k}(2(s+1)p-(2s+3)),x^{s+1}_{k}+\alpha^{s+1}_{k}+p^{k-1}t_{1}}}+b_{s}
						\end{align*}
						where $b_{s}=p^{s-1}(p-1)\left(p^{2}+p+t\right)$ and $l^{s+2}_{k} = \alpha^{s+1}_{k}p+\beta^{s+1}_{k}$ with $0\leq \alpha^{s+1}_{k} <p^{k-1},$ $\alpha^{s+1}_{k} = \alpha^{s}_{k}p+\beta^{s}_{k}$ with $0\leq \alpha^{s+1}_{k} <p^{k-2}$ and $0\leq \beta^{s+1}_{k},\beta^{s}_{k}\leq p-1.$
						\item For $j^{k}_{t-1}< l^{s+2}_{k} <j^{k}_{t},$ and $l^{s+2}_{k} =j^{k}_{t},$ $\frac{p+1}{2}\leq t \leq p-1$ 
						\begin{align*}
							\mathcal{M}_{\lambda^{2(r+2)}_{p^{k}(2(s+2)p-(2s+5)),x^{s+2}_{k}+l^{s+2}_{k}}}&=\sum\limits_{t_{2}=0}^{p-1}\mathcal{M}_{\lambda^{2r}_{p^{k}(2sp-(2s+1)),x^{s}_{k}+\alpha^{s}_{k}+p^{k-1}t_{2}}}+ \\ &\quad\sum\limits_{t_{1}=1}^{p-1}\mathcal{M}_{\lambda^{2(r+1)}_{p^{k}(2(s+1)p-(2s+3)),x^{s+1}_{k}+\alpha^{s+1}_{k}+p^{k-1}t_{1}}}+b_{s}
						\end{align*}
						where $b_{s}=p^{s-1}(p-1)\left(p^{2}+t\right)$ and $l^{s+2}_{k} = \alpha^{s+1}_{k}p+\beta^{s+1}_{k}$ with $0\leq \alpha^{s+1}_{k} <p^{k-1},$ $\alpha^{s+1}_{k} = \alpha^{s}_{k}p+\beta^{s}_{k}$ with $0\leq \alpha^{s+1}_{k} <p^{k-2}$ and $0\leq \beta^{s+1}_{k},\beta^{s}_{k}\leq p-1.$
					\end{enumerate}
					
				\end{enumerate}
			\end{thm}
			\begin{proof}
				\begin{enumerate}
					\item By Theorem \ref{p^{0}}, we have 
					\begin{align*}
						\mathcal{M}_{\lambda^{2(r+2)}_{(2(s+2)p-(2s+5)),(s+2)p-(s+3)}} &= 	p\mathcal{M}_{\lambda^{2(r+1)}_{(2(s+1)p-(2s+3)),(s+1)p-(s+2)}}+2p^{s}(p-1)\left(\frac{p-1}{2}\right) \\ &=\mathcal{M}_{\lambda^{2(r+1)}_{(2(s+1)p-(2s+3)),(s+1)p-(s+2)}}+(p-1)\mathcal{M}_{\lambda^{2(r+1)}_{(2(s+1)p-(2s+3)),(s+1)p-(s+2)}}\\ &\quad+2p^{s}(p-1)\left(\frac{p-1}{2}\right)\\ &= p\mathcal{M}_{\lambda^{2r}_{(2sp-(2s+1)),sp-(s+1)}}+ (p-1)\mathcal{M}_{\lambda^{2(r+1)}_{(2(s+1)p-(2s+3)),(s+1)p-(s+2)}}+
						b_{s}
					\end{align*}
					where $b_{s}=2p^{s-1}(p-1)\left(\frac{p^{2}-1}{2}\right).$
					\item \begin{enumerate}
						\item For $0\leq t<\frac{p-1}{2},$ By Theorem \ref{p^{1}}, we have
						\begin{align*}
							\mathcal{M}_{\lambda^{2(r+2)}_{p(2(s+2)p-(2s+5)),x^{s+2}_{1}+j^{1}_{t}}}&=\sum\limits_{t_{1}=0}^{p-1}\mathcal{M}_{\lambda^{2(r+1)}_{p(2(s+1)p-(2s+3)),x^{s+1}_{1}+j^{1}_{t_{1}}}}+p^{s}(p-1)\left(\frac{p+1}{2}\right)+\\&\quad p^{s-1}(p-1)\left(\frac{p(p-1)}{2}+t\right)\\ 
							&=\mathcal{M}_{\lambda^{2(r+1)}_{p(2(s+1)p-(2s+3)),x^{s+1}_{1}+j^{1}_{0}}}+\sum\limits_{t_{1}=1}^{p-1}\mathcal{M}_{\lambda^{2(r+1)}_{p(2(s+1)p-(2s+3)),x^{s+1}_{1}+j^{1}_{t_{1}}}}+ \\&\quad p^{s}(p-1)\left(\frac{p+1}{2}\right)+p^{s-1}(p-1)\left(\frac{p(p-1)}{2}+t\right)
							\\
							&=\sum\limits_{t_{2}=0}^{p-1}\mathcal{M}_{\lambda^{2r}_{p(2sp-(2s+1)),x^{s}_{1}+j^{1}_{t_{2}}}}+\sum\limits_{t_{1}=1}^{p-1}\mathcal{M}_{\lambda^{2(r+1)}_{p(2(s+1)p-(2s+3)),x^{s+1}_{1}+j^{1}_{t_{1}}}}+b_{s}
						\end{align*} 
						where $b_{s}=p^{s-1}(p-1)\left(p^{2}+p+t\right).$
						\item For $\frac{p-1}{2}\leq t \leq p-1,$ By Theorem \ref{p^{1}}, we have
						\begin{align*}
							\mathcal{M}_{\lambda^{2(r+2)}_{p(2(s+2)p-(2s+5)),x^{s+2}_{1}+j^{1}_{t}}}&=\sum\limits_{t_{1}=0}^{p-1}\mathcal{M}_{\lambda^{2(r+1)}_{p(2(s+1)p-(2s+3)),x^{s+1}_{1}+j^{1}_{t_{1}}}}+p^{s}(p-1)\left(\frac{p-1}{2}\right)+\\&\quad p^{s-1}(p-1)\left(\frac{p(p-1)}{2}+t\right)\\ 
							&=\mathcal{M}_{\lambda^{2(r+1)}_{p(2(s+1)p-(2s+3)),x^{s+1}_{1}+j^{1}_{0}}}+\sum\limits_{t_{1}=1}^{p-1}\mathcal{M}_{\lambda^{2(r+1)}_{p(2(s+1)p-(2s+3)),x^{s+1}_{1}+j^{1}_{t_{1}}}}+ \\&\quad p^{s}(p-1)\left(\frac{p-1}{2}\right)+p^{s-1}(p-1)\left(\frac{p(p-1)}{2}+t\right)
							\\
							&=\sum\limits_{t_{2}=0}^{p-1}\mathcal{M}_{\lambda^{2r}_{(2sp-(2s+1)),x^{s}_{1}+j^{1}_{t_{2}}}}+\sum\limits_{t_{1}=1}^{p-1}\mathcal{M}_{\lambda^{2(r+1)}_{p(2(s+1)p-(2s+3)),x^{s+1}_{1}+j^{1}_{t_{1}}}}+b_{s}
						\end{align*} 
						where $b_{s}=p^{s-1}(p-1)\left(p^{2}+t\right).$
						
					\end{enumerate}
					\item\begin{enumerate}
						\item For $j^{k}_{t-1}\leq l^{s+2}_{k} <j^{k}_{t},~0\leq t\leq \frac{p-1}{2}$ and $l^{s+2}_{k} =j^{k}_{t},~0\leq t< \frac{p-1}{2},$ from equations (\ref{sum11}), (\ref{sum2}) and (\ref{sum3}), we have 
						
						$\mathcal{M}_{\lambda^{2(r+2)}_{p^{k}(2(s+2)p-(2s+5)),x^{s+2}_{k}+l^{s+2}_{k}}}$
						\begin{align*}
							&\quad\quad=\sum\limits_{t_{1}=0}^{p-1}\mathcal{M}_{\lambda^{2(r+1)}_{p^{k}(2(s+1)p-(2s+3)),x^{s+1}_{k}+\alpha^{s+1}_{k}+p^{k-1}t_{1}}}+p^{s-1}(p-1)\left(\frac{p(p-1)}{2}+t\right)\\&\quad\quad\quad+p^{s}(p-1)\left(\frac{p+1}{2}\right),~\text{where} ~l^{s+2}_{k}= \alpha^{s+1}_{k}p+\beta^{s+1}_{k}\\ 
							&\quad\quad=\mathcal{M}_{\lambda^{2(r+1)}_{p^{k}(2(s+1)p-(2s+3)),x^{s+1}_{k}+\alpha^{s+1}_{k}}}+\sum\limits_{t_{1}=1}^{p-1}\mathcal{M}_{\lambda^{2(r+1)}_{p^{k}(2(s+1)p-(2s+3)),x^{s+1}_{k}+\alpha^{s+1}_{k}+p^{k-1}t_{1}}} \\&\quad\quad\quad+p^{s}(p-1)\left(\frac{p+1}{2}\right)+p^{s-1}(p-1)\left(\frac{p(p-1)}{2}+t\right)
							\\
							&\quad\quad=\sum\limits_{t_{2}=0}^{p-1}\mathcal{M}_{\lambda^{2r}_{p^{k}(2sp-(2s+1)),x^{s}_{k}+\alpha^{s}_{k}+p^{k-1}t_{2}}}+\sum\limits_{t_{1}=1}^{p-1}\mathcal{M}_{\lambda^{2(r+1)}_{p^{k}(2(s+1)p-(2s+3)),x^{s+1}_{k}+\alpha^{s+1}_{k}+p^{k-1}t_{1}}}+b_{s}, 
						\end{align*} 
						where $\alpha^{s+1}_{k}= \alpha^{s}_{k}p+\beta^{s}_{k}$ and $b_{s}=p^{s-1}(p-1)\left(p^{2}+p+t\right).$
						\item For $j^{k}_{t-1}< l^{s+2}_{k} <j^{k}_{t},$ and $l^{s+2}_{k} =j^{k}_{t},$ $\frac{p+1}{2}\leq t \leq p-1,$ 
						from equations (\ref{sum4}) and (\ref{sum5}), we have 
						
						$\mathcal{M}_{\lambda^{2(r+2)}_{p^{k}(2(s+2)p-(2s+5)),x^{s+2}_{k}+l^{s+2}_{k}}}$
						\begin{align*}
							&\quad\quad=\sum\limits_{t_{1}=0}^{p-1}\mathcal{M}_{\lambda^{2(r+1)}_{p^{k}(2(s+1)p-(2s+3)),x^{s+1}_{k}+\alpha^{s+1}_{k}+p^{k-1}t_{1}}}+p^{s-1}(p-1)\left(\frac{p(p-1)}{2}+t\right)\\&\quad\quad\quad+p^{s}(p-1)\left(\frac{p-1}{2}\right),~\text{where} ~l^{s+2}_{k}= \alpha^{s+1}_{k}p+\beta^{s+1}_{k}\\ 
							&\quad\quad=\mathcal{M}_{\lambda^{2(r+1)}_{p^{k}(2(s+1)p-(2s+3)),x^{s+1}_{k}+\alpha^{s+1}_{k}}}+\sum\limits_{t_{1}=1}^{p-1}\mathcal{M}_{\lambda^{2(r+1)}_{p^{k}(2(s+1)p-(2s+3)),x^{s+1}_{k}+\alpha^{s+1}_{k}+p^{k-1}t_{1}}} \\&\quad\quad\quad+p^{s}(p-1)\left(\frac{p-1}{2}\right)+p^{s-1}(p-1)\left(\frac{p(p-1)}{2}+t\right)
							\\
							&\quad\quad=\sum\limits_{t_{2}=0}^{p-1}\mathcal{M}_{\lambda^{2r}_{p^{k}(2sp-(2s+1)),x^{s}_{k}+\alpha^{s}_{k}+p^{k-1}t_{2}}}+\sum\limits_{t_{1}=1}^{p-1}\mathcal{M}_{\lambda^{2(r+1)}_{p^{k}(2(s+1)p-(2s+3)),x^{s+1}_{k}+\alpha^{s+1}_{k}+p^{k-1}t_{1}}}+b_{s}, 
						\end{align*} 
						where $\alpha^{s+1}_{k}= \alpha^{s}_{k}p+\beta^{s}_{k}$ and $b_{s}=p^{s-1}(p-1)\left(p^{2}+t\right).$
					\end{enumerate}
				\end{enumerate}
			\end{proof}
			\subsubsection*{Acknowledgement}
			The third author was financially supported by the Council of Scientific and Industrial Research (CSIR) of India, in the form of CSIR-JRF (File No: 09/0115(16295)/2023-EMR-I) to carry out this research.

	\end{document}